\documentclass[12pt,a4paper]{article}

\usepackage[utf8]{inputenc}
\usepackage[english]{babel}
\usepackage[T1]{fontenc}

\usepackage{amsmath}
\usepackage{amsfonts}
\usepackage{amssymb}
\usepackage{amsthm}
\usepackage{mathrsfs}

\usepackage{graphicx}

\usepackage{color}

\usepackage[normalem]{ulem}
\usepackage{comment}

\usepackage[colorlinks=true,linkcolor=blue]{hyperref}

\usepackage{emptypage}

\newtheorem{theorem}{Theorem}[section]
\newtheorem{lemma}[theorem]{Lemma}
\newtheorem{corollary}[theorem]{Corollary}
\newtheorem{proposition}[theorem]{Proposition}
\newtheorem{remark}[theorem]{Remark}
\newtheorem{conjecture}[theorem]{Conjecture}

\theoremstyle{definition}

\theoremstyle{remark}

\def \cB {\mathcal{B}}
\def \cC {\mathcal{C}}

\def \cE {\mathcal{E}}

\def \cG {\mathcal{G}}
\def \cH {\mathcal{H}}

\def \cL {\mathcal{L}}

\def \cR {\mathcal{R}}

\def \cV {\mathcal{V}}

\def \b {\beta}

\def \d {\delta}
\def \e {\varepsilon}

\def \l {\lambda}
\def \s {\sigma}

\def \G {\Gamma}

\def \E {\mathbb{E}}

\def \N {\mathbb{N}}

\def \P {\mathbb{P}}

\def \R {\mathbb{R}}

\def \Z {\mathbb{Z}}

\def\8{\infty}

\def \lra {\longrightarrow}

\def \un {\lbrace 0,1 \rbrace}

\title
{Scaling limit of the heavy-tailed ballistic deposition model with
	$p$-sticking}
\author
{Francis Comets\footnotemark[1]\thanks{Universit\'e de Paris and LPSM, Math\'ematiques, France. \textsc{email}: comets@lpsm.paris
} 
	\and Joseba Dalmau\thanks{NYU--ECNU Institute of Mathematics at NYU Shanghai, China. \newline \textsc{email}: dalmau.joseba@gmail.com}\and Santiago Saglietti\thanks{Pontificia Universidad Católica de Chile. \textsc{email}: sasaglietti@mat.uc.cl}}
\date{}

\begin{document}
	\maketitle
	
\begin{abstract}Ballistic deposition is a classical model for interface growth in which unit blocks fall down vertically at random on the different sites of $\Z$ and stick to the interface at the first point of contact, causing it to grow. We consider an alternative version of this model in which the blocks have random heights which are i.i.d. with a heavy (right) tail, and where each block sticks to the interface at the first point of contact with probability $p$ (otherwise, it falls straight~down until it lands on a block belonging to the interface). We study scaling limits of the resulting interface for the different values of $p$ and show that there is a phase transition as $p$ goes from $1$ to $0$.
\end{abstract}

\section{Introduction}\label{sec:intro}
In the last decades, random growth models have attracted a keen interest in the physics and mathematics communities \cite{BarabasiStanley,  DamronRAS16, DamronRAS-book}.  Typically, the height of a $d$-dimensional surface evolves subject to random increments which combine the following features:
\begin{enumerate}
	\item [(i)] locality: changes in height depends only on neighboring heights;
	\item [(ii)] spatial smoothing: valleys are quickly filled in, due to the influence of higher neighbors;
	\item [(iii)] nonlinear slope dependence: the effective growth rate increases in a nonlinear manner in the local slope;
\item [(iv)] space-time independent noise: the growth is driven by random noise with fast decay of correlations.
\end{enumerate}
Central examples of models with these characteristics include the Eden model, Diffusion Limited Aggregation, First and Last Passage Percolation (LPP), Directed Polymers in Random Environment (DPRE), PolyNuclear Growth and Ballistic Deposition (BD), among others.  
One reason for the flourishing of scientific contributions in models with one spatial dimension is that, under the extra assumption
\begin{enumerate}
	\item [(v)] the  noise has light tails,
\end{enumerate}
these models are expected to belong to the KPZ universality class, which is characterized by specific scaling exponents ruling fluctuations and path stretchings~\cite{Corwin16}, 
which are different from the Gaussian case. The most prominent member of this class is the KPZ equation~\cite{AmirCorwinQuastel}. Quite a few rigorous proofs that instances of these models belong to the KPZ class have been given, but they rely on an integrable structure.

To bypass the need of integrability, Hambly and Martin~\cite{HamblyMartin07} considered LPP in the first quadrant with heavy-tailed passage times. The limit only~retains the extreme statistics of the passage times field, 
and the scaling limit of the~path with largest passage time is given by the 1-Lipschitz path picking~up the largest sum of extremes.
Similar simplifying assumptions were later used to obtain scaling limits and characteristic exponents for DPRE \cite{AuffingerLouidor,BergerTorri,DeyZygouras}.

The model of ballistic deposition was first introduced in \cite{Vold59}: unit blocks fall independently from the sky and stick to the first point of contact, resulting in a lateral growth and creating overhangs.
Simulations and heuristic arguments strongly suggest ballistic deposition is in the KPZ class~\cite{KatzavSchwartz}, even if the block height is random with a light tail. However, a mathematical proof is far from reach at the moment and, furthermore, there is no hint of how stable to perturbations it is. In~\cite{Seppalainen} the height function is represented by a variational formula and a hydrodynamic limit is proved, and in~\cite{Penrose} laws of large numbers are established. The model of ballistic deposition can be seen as the $0$-temperature version of a general model in which falling blocks attach to the blocks deposited on neighboring sites with a probability that depends on an inverse temperature parameter $\beta$ in a Gibbsian fashion. Contrary to other models believed to belong to the KPZ class, the infinite-temperature version of this generalized ballistic deposition model does not belong to the Edward-Wilkinson universality class, but rather to to the newly found universality class of the Brownian Castle, see \cite{cannizzaro2021brownian}.

In this paper, instead of (v) we will assume that the noise has heavy tails, and that the block height is a random variable in the attracting domain of an $\alpha$-stable law, with $\alpha \in (0,2)$. For this case, we will derive the scaling limit of the height function.
Since this height function turns out to be given by a variational formula similar to that in (site) LPP, our limit is similar to that of~\cite{HamblyMartin07}. However, the picture and the arguments are now more involved due to the time-space random field of block depositions which determines a random cone of propagation. The region  $\alpha <2$ corresponds to the complete stretching of the optimal path according to the Flory argument, e.g. \cite{BiroliBouchaudPotters}.

Another related question we consider is that of the domain of attraction of this scaling limit. 
The random deposition model (RD) is one of the simplest models for a randomly growing one-dimensional interface. In this model, unit blocks fall down independently and vertically on the different sites of~$\Z$, depositing themselves on top of the last block to have fallen at the same point. Since there is no sticking to neighboring columns,
the height functions of the different columns are independent processes having a stable law as their limit for large times, and the corresponding model under assumption (v) belongs to the Gaussian class instead of KPZ. Our second main contribution is to study the competition between the corresponding universality classes under the assumption of heavy tails. We derive scaling limits for mixtures of RD and BD.  First, we show that any \textit{fixed} amount of BD in the mixture is enough for the scaling limit to be the same as pure BD. This result can be viewed as a small step towards the following conjecture:
\begin{conjecture}
For light-tailed block heights, any fixed amount of BD makes a deposition process fall into the KPZ universality class. 
\end{conjecture}
On the other hand, if we consider \textit{infinitesimal} amounts of BD in the~mixture, i.e. an amount tending to zero as time goes to infinity, then we show that a phase transition takes place: if this infinitesimal amount is not too small then one~recovers the scaling limit of pure BD, whereas for all smaller amounts the scaling of the height function becomes different. 

The key tool we use to derive these results is the aforementioned alternative representation of the height function in terms of a variational formula in the spirit of LPP. As such, one important part of our analysis is to obtain suitable bounds on the number of macroscopic weights collected by any optimal path in this LPP representation. We achieve this by studying an auxiliary ballistic deposition model in which the block heights are i.i.d. Bernoulli distributed. For this auxiliary model, we obtain upper bounds on the expected growth of its height function analogous to those given in~\cite{HamblyMartin07} for Bernoulli LPP, which the reader may find of independent interest. 

In this paper we always take the spatial dimension $d=1$. We could have~also treated the case $d > 1$ by performing a few simple changes, but we chose not to as this does not bring any new significant features. The paper is organized as follows: in Section 2 we formally introduce the ballistic deposition model, as well as its scaling limit; in Section 3 we state our main results; in Section 4 we give a rigorous construction of the process and then, in Section 5, we use this construction to give the alternative representation of the height function in terms of a last passage percolation problem; in Section~6 we give a general outline of the proofs, while Sections 7 through~10 are devoted to the proofs of various auxiliary technical results.

\section{Description of the models}\label{sec:desc}


We now formally introduce the ballistic deposition model we shall work with, as well as the continuous model which will act as its scaling limit.

\subsection{The Ballistic Deposition model}
On each site $x \in \Z$, rectangular blocks fall down at random with rate~$1$, independently of all other sites. Falling blocks each have width $1$ and their own random height, where the heights corresponding to the different blocks are i.i.d. with a common distribution function $F$.
These blocks will ``deposit'' themselves on top of the different sites in $\Z$ and thus form a growing cluster according to the \textit{deposition rules} we describe next.  
First, fix some $p\in [0,1]$, which will henceforth be referred to as the \textit{sticking parameter} of the model. Then, whenever a block falls on top of site $x$, it will do one of the following:
\begin{enumerate}
	\item [$\bullet$] with probability $1-p$, it will deposit itself directly on top of the last block that fell on $x$ (see Figure \ref{fig:stickiness});
	\item [$\bullet$] with probability $p$, it will \textit{stick} to the growing cluster of blocks at the first point of contact, which will belong to the last block deposited on any one of the sites $x-1$, $x$, $x+1$, whichever has been deposited at the largest height among the three
(see Figure \ref{fig:stickiness}). 
\end{enumerate}
\begin{figure}
	\includegraphics[width=\textwidth]{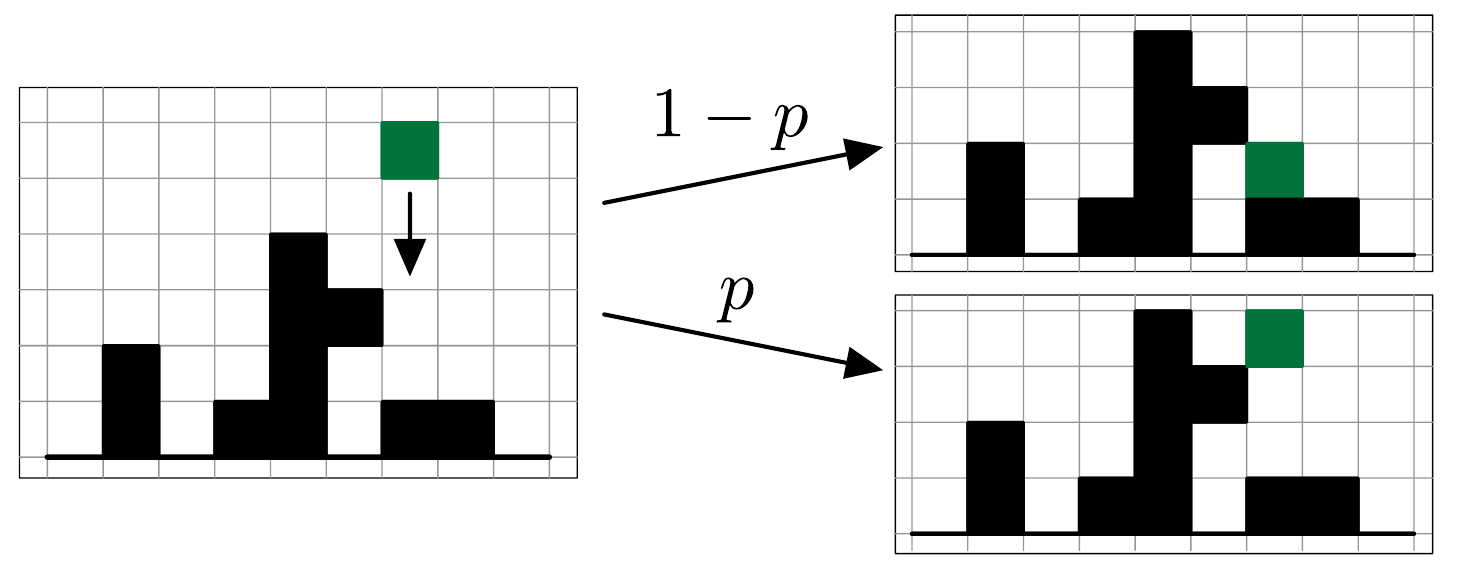}
	\caption{\textit{The Ballistic Deposition Model with unit blocks}}
	\label{fig:stickiness}
\end{figure}

If for $x \in \Z$ and $t \geq 0$ we denote by $h(x,t)$ the height of the growing cluster above the site $x$  at time $t$, we will call the random configuration 
\[
h=(h(x,t) : x \in \Z\,,\,t \geq 0)
\] the \textit{	ballistic deposition model} with $p$-\textit{sticking} and \textit{height distribution function} $F$. The particular case in which $p=1$ and $F=1_{[1,\infty)}$
corresponds to the usual ballistic deposition model, whereas the case $p=0$ corresponds to the \textit{random deposition model} (RD). 
In order to complete the description of our model, it remains to specify an initial condition.
Our methods allow for treatment of a wide range of
different initial conditions. In particular,
the usual \emph{flat}, \emph{seed} and \emph{random}
initial conditions can all be treated without any added difficulty. However, for the sake of simplicity,
throughout the article we~will always consider the model with a \textit{flat} initial condition, i.e. $h(\cdot,0) \equiv 0$.

If for any configuration $h \in \R^\Z$, $x \in \Z$ and $\eta \geq 0$ we define the configurations $R^{(0)}_x(h,\eta) \in \R^\Z$ and $R^{(1)}_x(h,\eta) \in \R^\Z$ by the formulas
\begin{equation}
	\label{eq:r0}
	R^{(0)}_x(h,\eta)(y)\,=\,\begin{cases}
		 h(y) &\quad \text{if}\ y\neq x\\
		 h(x)+\eta &\quad \text{if}\ y =x
	\end{cases}
\end{equation} and
\begin{equation}
	\label{eq:r1}
	R^{(1)}_x(h,\eta)(y)\,=\,\begin{cases}
		h(y) &\quad \text{if}\ y\neq x\\
		\max\lbrace 
		h(x-1),h(x),h(x+1)\rbrace+\eta &\quad \text{if}\ y =x
	\end{cases}
\end{equation} then the ballistic deposition model described above can be formally defined as the Markov process on $\R^\Z$ having a flat initial condition $h(\cdot,0) \equiv 0$ and with infinitesimal generator given by 
\begin{equation}\label{eq:gen}
	\mathscr{L}(f)(h)= \sum_{x \in \Z} \int_0^\infty\left[ pf(R^{(1)}_x(h,\eta)) + (1-p)f(R^{(0)}_x(h,\eta))-f(h)\right]\mathrm{d}F(\eta)
\end{equation} for any bounded and continuous function $f \in C_b(\R^\Z)$ which is also \textit{local}, i.e. $f(h)$ depends on $h$ only through the values of $h(x)$ for finitely many $x \in \Z$. 
In Section \ref{sec:const}, we will explicitly construct the process $h$ as a function of a marked Poisson process on $\Z \times \R$. 

\begin{figure}
	\includegraphics[width=\textwidth]{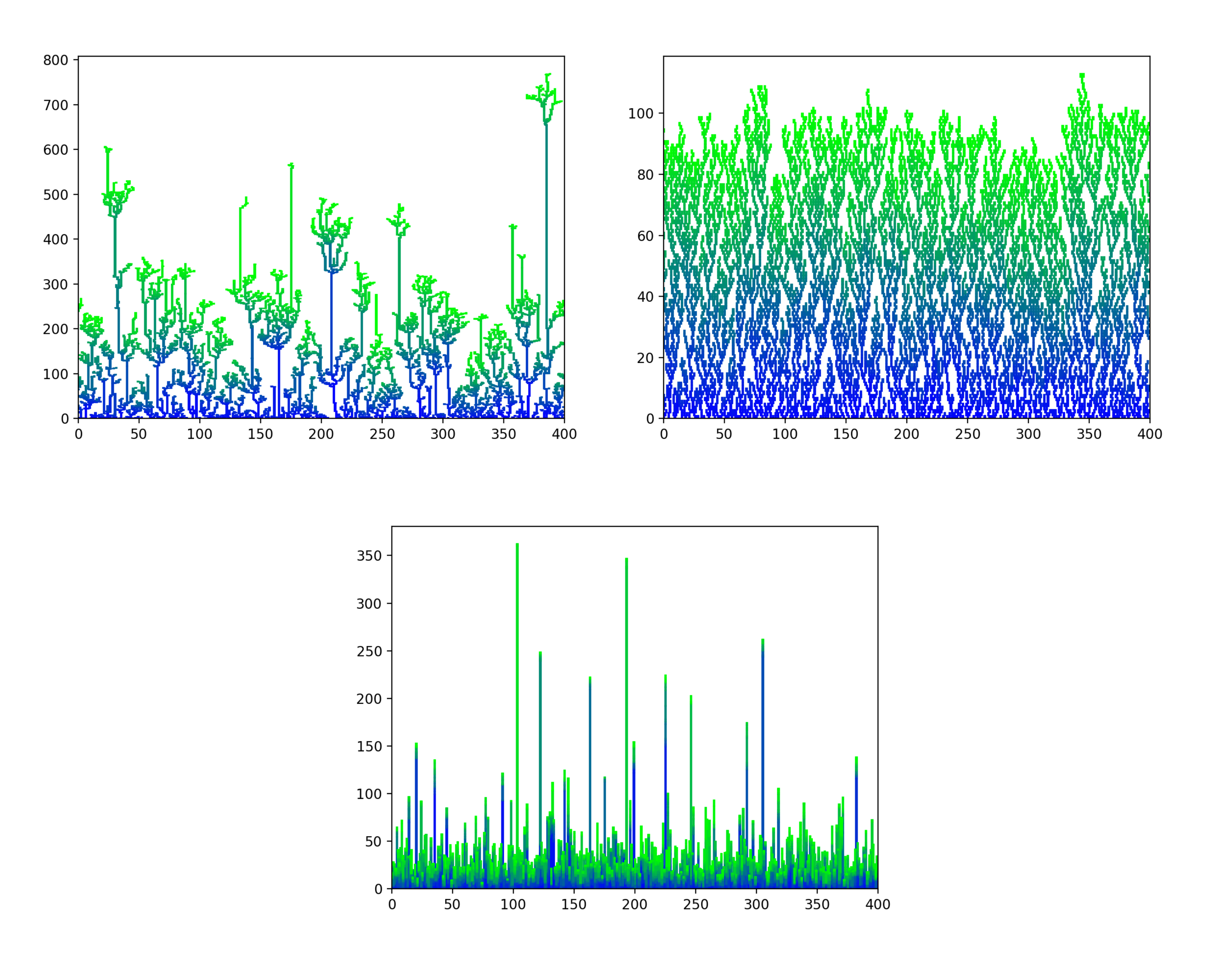}
	\caption{\textit{Simulations of three ballistic deposition models for different values of $p$ and $F$.} All simulations feature $10000$ block depositions (the color of each block represents time of deposition) on a torus with $400$ sites, starting from a flat initial condition. The upper-left picture corresponds to $p=1$ and $F$ given by a Pareto distribution with index $\alpha=1.5$ (heavy-tailed BD); the upper-right picture corresponds to $p=1$ and $F=1_{[1,\infty)}$ (standard BD); the bottom picture corresponds to $p=0$ and $F$ given by a Pareto distribution with index $\alpha=1.5$ (heavy tailed RD).}
	\label{fig:depfigures2}
\end{figure}

In the sequel, we will only work with models whose block height distribution function is continuous and heavy-tailed. More precisely, we will assume:

\textbf{Assumptions (F).} The block height distribution function $F$ satisfies:
\begin{itemize}
	\item [F1.] $F$ is continuous with $F(0)=0$;
	\item [F2.] $F$ is regularly varying at infinity with index $\alpha \in (0,2)$, i.e. 
	\begin{equation}
		\forall t> 0\qquad 
		\lim_{x\to\infty}\,\frac{1-F(tx)}{1-F(x)}\,=\, t^{-\alpha}.
	\end{equation}
\end{itemize} 


\subsection{The Continuous Last Passage Percolation model}\label{sec:cont2}

Our main objective in this article is to establish scaling limits as $t \rightarrow \infty$ for the height function $h(x,t)$ in different scenarios. The first step is to introduce the limiting object which will appear in all these scaling limits. We shall do this via a \emph{continuous last passage percolation} model.

Let us define the \emph{path space} 
\[
\cL=\big\lbrace
\gamma:[0,1]\to\R : \gamma(1)=0\,,\,|\gamma(t_1)-\gamma(t_2)|\leq |t_1 -t_2| \,\,\forall\, t_1,t_2 \in [0,1]\big\rbrace,
\] i.e. the set of $1$-Lipschitz paths on $[0,1]$ ending at $0$, together with the~\textit{triangle}
\[
\Delta:=\big\lbrace
(x,t)\in\R^2 : 0\leq t\leq 1\,,\, |x|\leq 1-t
\big\rbrace.
\] Observe that if we define the \text{graph} of any path $\gamma \in \cL$ by 
\[
\text{graph}(\gamma):=\{ (\gamma(t),t) : t \in [0,1]\}
\] (note the unconventional order of the coordinates in our definition of graph) then we have $\Delta = \cup_{\gamma \in \cL} \text{graph}(\gamma)$. That is, the triangle $\Delta$ is the minimal set with the property that $\text{graph}(\gamma) \in \cL$ for any $\gamma \in \cL$. 

To define our limiting object, let us first consider a Poisson process on $\R_+$ with rate $1$. If we order the points of this process in an increasing fashion, i.e. $0 \leq X_1 < X_2 <\dots < X_i < \dots$, then it is straightforward to check that the decreasing sequence of points $M=(M_i)_{i \in \N}$ given by 
\begin{equation}\label{eq:defM}
	M_i:=(X_i)^{-\tfrac{1}{\alpha}},
\end{equation} 
with $\alpha \in (0,2)$ as in Assumptions (F) is a non-homogeneous Poisson process on $\R_+$ with intensity measure having density $\alpha m^{-(\alpha+1)}1_{m > 0}$. 
Now, let us consider an i.i.d. sequence $(U_i)_{i \in \N}$ with uniform distribution on the triangle $\Delta$ and define the random point measure $\Pi$ on~$\Delta$ by the formula
\begin{equation}
	\label{eq:defpi0}
	\Pi := \sum_{i \in \N} M_i\delta_{U_i}.
\end{equation}
Observe that the collection $\{(U_i,M_i) : i \in \N\}$ is a Poisson process on $\Delta \times \R_+$ with intensity measure having density $1_{\Delta}(x) \times \alpha m^{-(\alpha+1)}1_{m>0}$. 

We define our limit as 
\begin{equation}
	\label{eq:H}
	H:=\sup_{\gamma\in\cL}\Pi\big(
	\text{graph}(\gamma)
	\big).
\end{equation} That is, $H$ represents the maximum sum of weights $M_k$ that can be collected by any path $\gamma\in\cL$ in the triangle $\Delta$.

Our first result, which we will prove in Section~\ref{sec:cont}, is then the following:

\begin{theorem}\label{theo:0} The quantity $H$ defined in \eqref{eq:H} is measurable and a.s. finite. Furthermore, there exists an a.s.-unique path $\gamma^* \in \cL$ with $H=\Pi(\text{graph}(\gamma^*))$ so that, in particular, with probability one we have
	\[
	H=\max_{\gamma\in\cL}\Pi\big(
	\text{graph}(\gamma)\big).
	\]
\end{theorem}

Observe that $H$ is obtained essentially via a continuous version of a standard last passage percolation model (LPP), where increasing paths are replaced by Lipschitz continuous functions and the weights are now randomly distributed throughout the region $\Delta$ according to a Poisson process. For this reason, in the sequel we will often refer to this limit setting as the \textit{continuous model} and, for reasons that will become even more evident in Section \ref{sec:lpprep}, call our original ballistic deposition model the \textit{discrete model}.  

We also point out that our limiting object $H$ is analogous to the one in \cite{HamblyMartin07}. Indeed, the only difference between our $H$ and the random variable $T$ defined in \cite[Equation~(2.5)]{HamblyMartin07} is that, in the latter, $\Delta$ is replaced by the unit square $[0,1]^2$ and $\mathcal{L}$ is replaced by the set of all increasing paths on this unit square. By performing translation by $(0,-1)$ and then a 135-degree counterclockwise rotation, 
the set $\Delta$ can be seen as one half of the square of side length $\sqrt{2}$, and maximizing over paths in $\mathcal{L}$ in this context is equivalent to doing so over the set of all \emph{increasing} paths (see Section~\ref{sec:cont} and Figure~\ref{fig:delta} therein).  In other words, $H$ is a continuous version of the \textit{point-to-line} LPP model, while $T$ corresponds to a continuous version of \textit{point-to-point} LPP. 
The only reason why we obtain $H$ instead of $T$ in our limits is because we have chosen to work with a flat initial 
condition, i.e. $h(\cdot,0)\equiv 0$, instead of
a \textit{seed}-type initial condition, i.e. $h(\cdot,0)=(-\infty)\mathbf{1}_{x=0}$.

\section{Main Results}\label{sec:res}

Having introduced the limiting object, we can now state our first main result, which states that the continuous object $H$ from \eqref{eq:H} is the scaling limit of the height function in the discrete model whenever the sticking parameter $p>0$ remains fixed. First, let us introduce the appropriate scaling. For each $t > 0$ let us define the quantity 
\begin{equation} \label{eq:defq}
a_t:=F^{-1}(1-1/t),
\end{equation} where $F^{-1}$ denotes the generalized inverse function of $F$. Observe that, under Assumptions (F), we have $a_t=t^{1/\alpha} L(t)$ for some slowly varying function~$L$, i.e. such that for all $t > 0$
\begin{equation}\label{eq:L}
\lim_{x \to \infty} \frac{L(tx)}{L(x)} = 1.
\end{equation} Our first main result is the following:

\begin{theorem}\label{theo:1} Given $p \in (0,1]$, the height function $h$ of the discrete model with sticking parameter $p$ satisfies, as $T \rightarrow \infty$, the convergence in distribution
\begin{equation}\label{eq:conv}
\frac{1}{a_{pT^2}}h(0,T) \overset{d}{\longrightarrow} H.
\end{equation}
\end{theorem}

\begin{remark}\label{rem:1} Under the assumption that $F$ is regularly varying at infinity, for any fixed $p > 0$ we have that 
	\[
	 a_{pT^2}=p^{1/\alpha}a_{T^2}(1+o_T(1)) = (pT^2)^\frac{1}{\alpha}L'(T),
	\] where $o_{T}(1) \rightarrow 0$ as $T \rightarrow \infty$ and $L'$ is some slowly varying function at infinity (which may not coincide with the function $L$ in \eqref{eq:L}). In particular, the limit in \eqref{eq:conv} still holds if we normalize $h(0,T)$ by $p^{1/\alpha}a_{T^2}$ instead of $a_{pT^2}$. However, this will not necessarily be true anymore in Theorem \ref{theo:2}~below, where we allow $p$ to vary with $T$.
\end{remark}

To understand why we obtain such a limit for our ballistic deposition model, it will be convenient to give an alternative formulation of our model in terms of a specific last passage percolation problem. We do this in Section \ref{sec:lpprep}.

Notice that the situation depicted in Theorem \ref{theo:1} is considerably different from that of the model without sticking, i.e. $p = 0$. In this case, the behavior for each column is independent and therefore $h(0,T)$ has the distribution of a random sum of i.i.d. random variables $\eta_i$ having distribution function $F$, where the number of summands is Poisson of parameter $T$ and independent of the $\eta_i$. It then follows from the generalized CLT (see \cite[Theorem 3.8.2]{DurrettBook2019}), that
\begin{equation} \label{eq:rd}
\frac{h(0,T) - c_T}{a_T} \overset{d}{\longrightarrow} G_\alpha
\end{equation} where $G_\alpha$ is a certain stable law of index $\alpha$, $a_T$ is as in \eqref{eq:defq} and $c_T$ is given by
\[
c_T:=\begin{cases} 0 & \text{ if }0 < \alpha < 1\\ \\ T\displaystyle{\int_0^{a_T} x\mathrm{d}F(x)} & \text{ if }\alpha = 1\\ \\ T\displaystyle{\int_0^\infty x \mathrm{d}F(x)} & \text{ if }1 < \alpha <2.
	\end{cases} 
\] In particular, as $T \rightarrow \infty$, $h(0,T)$ is of order 
\[
\begin{cases} T^{1/\alpha}L(T) & \text{ if }0 < \alpha < 1\\ \\ T\widetilde{L}(T) & \text{ if }\alpha = 1\\ \\ T & \text{ if }1 < \alpha <2
\end{cases}
\] for some slowly varying functions $L$ and $\widetilde{L}$. Thus, being the behavior of the height function $h(0,T)$ for $p = 0$ and $p>0$ drastically different as $T \rightarrow \infty$, 
it is natural to ask how will $h(0,T)$ behave if one takes $p \rightarrow 0$ and $T \rightarrow \infty$ \textit{simultaneously}. Our next result shows that, as long as $p$ does not tend to $0$ too fast, $h(0,T)$ behaves in the same way described in Theorem \ref{theo:1}.

\begin{theorem}\label{theo:2} Fix $\zeta \in (0, (2-\alpha) \wedge 1)$. Then, for any sequence $(p_T)_{T > 0}$ satisfying $p_T \in [T^{-\zeta},1]$ for all $T$ sufficiently large, as $T \rightarrow \infty$ we have that
\[
\frac{1}{a_{p_T T^2}} h^{(p_T)}(0,T) \overset{d}{\longrightarrow} H,
\] where, for each $p \in [0,1]$, $h^{(p)}(0,T)$ denotes the height of the growing cluster above $0$ at time $T$ for the discrete model with sticking parameter $p$. 
\end{theorem}  

We mention that Theorem \ref{theo:2} does not hold for smaller sticking parameters, i.e. if $\zeta \geq (2-\alpha) \wedge 1$. Indeed, on the one hand notice that if $p_T \leq \frac{1}{T}$ then, with probability at least $\mathrm{e}^{-1}$, none of the blocks that fell on $x=0$ stuck to a neighboring column. As a consequence, on this event the height $h^{(p_T)}(0,T)$ coincides with that of the model with $0$-sticking, so~that there is no hope of retaining the same limit (the asymptotics on this event are in fact given by \eqref{eq:rd}). Thus, we must always have $\zeta < 1$. On the other hand, it is easy to see that $h^{(p_T)}(0,T)$ is always \textit{at least} as large as $h^{(0)}(0,T)$, the height in the model with $0$-sticking. Therefore, if $\alpha \in (1,2)$ and $\zeta \in (2-\alpha,1)$, we have $h^{(p_T)}(0,T) \geq h^{(0)}(0,T) = O(T)$ and hence, since $a_{T^{2-\zeta}}=o(T)$,  $\tfrac{1}{a_{T^{2-\zeta}}}h^{(T^{-\zeta})}(0,T)$ cannot converge in distribution as $T \rightarrow \infty$ (when $\zeta = 2 - \alpha$ this may also be the case depending on the slowly-varying term $L(t)$ from~$a_t$). In particular, if we take $p_T=T^{-\zeta}$ then the model exhibits a phase transition at $\zeta_c:=(2-\alpha) \wedge 1$: for $\zeta < \zeta_c$ the model behaves in the same way as the model with ``pure'' ballistic deposition (i.e. $p_T \equiv 1$), whereas for $\zeta > \zeta_c$ the behavior is different.  It would be an interesting problem to understand which is the correct scaling of the height function in the latter cases and to determine whether a true scaling limit actually exists. For concreteness, we summarize the above discussion in the following corollary:

\begin{corollary} Let $\alpha \in (0,2)$ as in Assumptions (F). If we fix $\zeta \in (0,1)$ and set $p_T:=T^{-\zeta}$ for each $T > 0$, then the system $(h^{(p_T)}(\cdot,T))_{T > 0}$ exhibits a phase transition at $\zeta_c:=(2-\alpha) \wedge 1$:
	\begin{enumerate}
		\item [i.] if $\zeta < \zeta_c$ then $\frac{1}{a_{T^{2-\zeta}}}h^{(p_T)}(0,T) \overset{d}{\longrightarrow} H$ as $T \to \infty$,
		\item [ii.] if $\zeta > \zeta_c$ then $\big(\frac{1}{a_{T^{2-\zeta}}}h^{(p_T)}(0,T)\big)_{T > T_0}$ is not tight for any fixed $T_0 > 0$. Furthermore, if also $\alpha \in (1,2)$ then $\frac{1}{a_{T^{2-\zeta}}}h^{(p_T)}(0,T) \overset{\P}{\longrightarrow} \infty$ as $T \to \infty$.
	\end{enumerate}
\end{corollary}

As remarked earlier, the proofs of these results rely on an alternative representation of the height function $h$ based on a variational formula involving a specific last passage percolation problem. This alternative representation of $h_T$ will not only throw light on the link between our ballistic
	deposition model and the CLPP model from Section~\ref{sec:cont2}, it will also form the basis for a coupling in which most of our proofs will build upon. In the next section we formally construct the ballistic deposition model, and then in Section~\ref{sec:lpprep} we introduce this alternative LPP representation for the height function $h$. 
In Section~\ref{sec:outline} we give a general
	outline of the proofs of the above theorems. The rest of the 
	paper is devoted to carrying out those proofs.


\section{Construction of the process}\label{sec:const}

We now carry out the formal construction of our ballistic deposition model as a function of a marked Poisson process on $\Z \times \R_+$. This explicit construction will be useful to establish the link between our model and certain last passage percolation models and it will also serve as a basis for the  couplings we shall perform later in Section \ref{coupling} to help us with the proofs.

Let us start the construction by considering a Poisson process $\xi$ on $\Z \times \R_+$ with intensity $n_\Z \otimes \lambda_{\R^+}$, where $n_\Z$ stands for the counting measure on $\Z$ and $\lambda_{\R^+}$ denotes the Lebesgue measure on $\R_+$. We shall treat $\xi$ indistinctively as both a random point measure
\[
\xi:=\sum_{k \in \N} \delta_{(x_k,t_k)}
\] and a random collection of points $\xi=\{ (x_k,t_k) : k \in \N\}$ (corresponding to the atoms of $\xi$ when viewed as a point measure). In particular, expressions of the sort $(x,t) \in \xi$ and $\xi(\{(x,t)\})=1$ will be understood as synonyms (and may both be used indistinctively in the sequel). 
The points in $\xi$ will represent the \textit{falling block events}, i.e. that the point $(x,t)$ belongs to $\xi$ means that a block will fall down on top of site $x$ (and attach to the growing cluster) precisely at time $t$. 

Next, we endow each of the points $(x,t) \in \xi$ with its own independent mark $(\varepsilon(x,t),\eta(x,t))$, where:
\begin{enumerate}
	\item [$\bullet$] $\varepsilon(x,t)$ is a Bernoulli random variable of parameter $p$, which indicates whether the falling block $(x,t)$ will choose to stick to the first point of contact of the growing cluster (if $\varepsilon(x,t)=1$) or not (if $\varepsilon(x,t)=0$);
	\item [$\bullet$] $\eta(x,t)$ represents the random \textit{height} of the falling block $(x,t)$, which is distributed according to $F$ and independent of $\varepsilon(x,t)$. 
\end{enumerate} Formally, we could carry out this marking by introducing a Poisson process $\overline{\xi}$ on $\Z\times \R_+ \times (\{0,1\}\times \R_+)$ with intensity measure $n_\Z \otimes \lambda_{\R_+} \otimes (\text{Ber}(p)\otimes P_F)$, where $P_F$ here denotes the distribution on $\R_+$ induced by $F$, and viewing $\xi$ as the restriction of $\overline{\xi}$ to its first two coordinates. Thus, a point $(x,t,\varepsilon,\eta) \in \overline{\xi}$ would be regarded as a block $(x,t) \in \xi$ carrying the mark $(\varepsilon,\eta)$. However, in the sequel it will be more convenient to think of marks not via $\overline{\xi}$ but rather as described above: as independent decorations $(\varepsilon(x,t),\eta(x,t))$ attached to each $(x,t) \in \xi$. 

Now, we construct our ballistic deposition model $h=(h(x,t) : x \in \Z\,,\,t \geq 0)$ by specifying the evolution of the height functions $t \mapsto h(x,t)$ for each $x \in \Z$. For any such $x$, the value of $h(x,t)$ will be piecewise constant, updating itself only at those times $t \geq 0$ such that $(x,t) \in \xi$. Thus, we fix $h(x,0) \equiv 0$ as the initial condition and then for $(x,t) \in \xi$ we set 
\[
h(x,t)\,:=\, \varepsilon(x,t) R^{(1)}_x(h(\cdot,t^-),\eta(x,t)) + (1-\varepsilon(x,t))R^{(0)}_x(h(\cdot,t^-),\eta(x,t))\,,
\] where $R^{(0)}_x$ and $R^{(1)}_x$ are defined as in \eqref{eq:r0} and \eqref{eq:r1}, respectively. 

It is straightforward to check that $h(x,t)$ is well-defined for all $t \geq 0$ and $x$, and that the process $h=(h(x,t) : x \in \Z\,,\,t \geq 0)$ has infinitesimal generator given by \eqref{eq:gen}. We omit the details. 


\section{Last passage percolation representation}\label{sec:lpprep}

We now give an alternative representation of our height configuration $h$ based on a last passage percolation problem, which will help us better understand the connection with the limit $H$ in Theorem \ref{theo:1}. 
Intuitively, the idea is to evaluate the height $h(x,T)$
at a point $x\in \Z$ and time $T>0$ by exploring the Poisson process $\xi$ of falling blocks backwards in time from $t=T$ to $t=0$.
At first, i.e. for times just below $T$, it is enough to look at the Poisson process restricted to $\lbrace x\rbrace  \times [0,T]$. However, once we encounter a sticky block, we may ``jump'' to the neighboring sites to go and find the highest column.
This backward exploring procedure generates a \emph{path space}
and a \emph{propagation cone} of accessible paths
that we can use to compute $h(x,T)$.
Let us begin formalizing these ideas by introducing some notation.

First, given our Poisson process $\xi$ from the construction of $h$ in Section \ref{sec:const}, let us define
\begin{equation} \label{eq:defs}
\xi^{\text{[st]}}:=\{ (x,t) \in \xi : \varepsilon(x,t) = 1\}.
\end{equation} Put into words, $\xi^{\text{[st]}}$ is simply the collection of all points in $\xi$ which are \textit{sticky}, i.e. which have a Bernoulli mark equal to $1$. These sticky points represent the falling blocks in our deposition dynamics which stick to the growing cluster at the first point of contact. 

Now, given $T_2 > T_1 \geq 0$ and $y \in \Z$, define the set $\mathcal{C}_{[T_1,T_2];y}$ of \textit{compatible paths} as
\begin{equation}\label{def:compatible}
\mathcal{C}_{[T_1,T_2];y}=\bigg\lbrace
\gamma:[T_1,T_2]\to \Z\,\bigg|\,\begin{matrix} \gamma(T_2)=y,s \text{ c\`adl\`ag}, |\gamma(t^-)-\gamma(t)|\leq 1 \text{ for all }t,\\ 
\gamma(t^-)\neq \gamma(t) \text{ for some }t \Longrightarrow (\gamma(t),t) \in \xi^{\text{[st]}} \end{matrix}
\bigg\rbrace.
\end{equation}
That is, $\mathcal{C}_{[T_1,T_2];y}$ is the set of paths on $\Z$ such that they: 
\begin{enumerate}
	\item [i.] end at $y$;
	\item [ii.] have càdlàg trajectories which are piecewise constant and consist only of nearest-neighbor jumps;
	\item [iii.] can jump at time $t$ from one site $x'$ onto one of its nearest neighbors $x$ only if $(x,t)$ is a sticky point from $\xi$.
\end{enumerate}   

Having introduced the set of compatible paths, we now set for $T_2>T_1 \geq 0$
\[
\xi_{[T_1,T_2]}:=\{ (x,t) \in \xi : T_1\leq t \leq T_2 \}
\] and define the set $\cV_{[T_1,T_2];y}$ of \textit{attainable space-time points} as
\begin{equation} \label{def:attainable}
	\cV_{[T_1,T_2];y}:=\big\lbrace
	(x,t)\in \xi_{[T_1,T_2]} : 
		\exists\, \gamma \in \cC_{[T_1,T_2];y}\ \text{such that}\ \gamma(t)=x
	\big\rbrace.
\end{equation}

With these definitions at hand, we have the following last passage percolation type representation for the height $h(y,T_2)$:

\begin{proposition}\label{prop:lpprep} For any fixed $T_2>T_1\geq 0$ and $y \in \Z$, with probability one we have the identity
\begin{equation} \label{eq:lpprep}
h(y,T_2)=\max
\bigg\lbrace
h\big(
\gamma(T_1),T_1
\big)+\!\!\!\!\!\!
\sum_{(\gamma(t),t) \in \xi} \eta(
\gamma(t),t
) :
\gamma\in\cC_{[T_1,T_2];y}
\,\bigg\rbrace.
\end{equation}
\end{proposition}

\begin{remark} This specific representation for the height function $h$ is known and has been used previously in the literature, see e.g. \cite{khanin2010ballistic}. See also \cite{cannizzaro2021brownian} for an analogous representation in the infinite-temperature regime.
\end{remark}

\begin{proof}
Consider the set of attainable sticky points $\cV^{\text{[st]}}_{[T_1,T_2];y}:= \cV_{[T_1,T_2];y} \cap \xi^{\text{[st]}}$. This set is almost surely finite (as a matter of fact, we will show in Lemma~\ref{lemma:control} that the cardinality of $\cV_{[T_1,T_2];y}$ has finite expectation, so that $\cV_{[T_1,T_2];y}$ and all of its subsets are a.s.-finite).
If $\cV^{\text{[st]}}_{[T_1,T_2];y}$ is empty then $\cC_{[T_1,T_2];y}$ consists only of the constant path $s(t) \equiv y$, in which case \eqref{eq:lpprep} is immediately verified. Thus, let us assume that $\cV^{\text{[st]}}_{[T_1,T_2];y}$ is nonempty. In this case, since for fixed $T_2 > T_1 \geq 0$ with probability one no two points in $\cV^{\text{[st]}}_{[T_1,T_2];y}$ have the same time coordinate and, furthermore, all these time coordinates are different from $T_1$ and $T_2$, we may number these points in a time-decreasing fashion: that is, we can write 
\[
\cV^{\text{[st]}}_{[T_1,T_2];y}:=\{(x_1,t_1),\dots,(x_n,t_n)\}
\] where the $t_i$ satisfy $T_2 > t_1 > \dots > t_n > T_1$.	
Observe that, by definition of $\cV^{\text{[st]}}_{[T_1,T_2];y}$, we have $x_1 = y$.

Now, choose some (random) $\delta > 0$ small enough so that $t_1 - \delta > t_2$ if $n > 1$ or $t_1 - \delta > T_1$ if $n=1$. Observe that, by definition of compatible path and the ordering of the $(x_i,t_i)$, the set $\cC_{[t_1-\delta,T_2];y}$ is composed of \textit{exactly} three paths: all of them are of the form
\[
s(t)\,=\,\begin{cases}
	y & \text{ if }t \in [t_1,T_2]\\ \\ y+c & \text{ if }t \in [t_1-\delta,t_1),\end{cases}
\] with $c \in \{-1,0,+1\}$. With this in mind, if we take $\delta$ small enough so that, in addition, no
	point in $\cV_{[T_1,T_2];y}$ has its time coordinate equal to $t_1-\delta$ (which we can do since $\cV_{[T_1,T_2];y}$ is almost surely finite), it is straightforward
	to check that
\[
h(y,T_2)=\max
\bigg\lbrace
h\big(
\gamma(t_1-\delta),t_1-\delta
\big)+\!\!\!\!\!\!
\sum_{(\gamma(t),t) \in \xi} \eta(
\gamma(t),t
) :
\gamma\in\cC_{[t_1-\delta,T_2];y}
\bigg\rbrace.
\] The general claim in \eqref{eq:lpprep} now follows from this by induction on the $t_i$.
\end{proof}

There is yet another way to realize our ballistic deposition process, based on Proposition \ref{prop:lpprep}, which is intimately related with our limit object $H$ in \eqref{eq:H}. We explain this alternative realization next.

To begin, we introduce the following less cumbersome notation 
for the objects we use the most:
\begin{equation}\label{eq:ab}
\cC_T:=\cC_{[0,T];0}\qquad\xi_T:=\xi_{[0,T]}\qquad
\cV_T\:=\cV_{[0,T];0}
\end{equation} and define also $\cV^{(\eta)}_T$ to be the set of points in $\cV_T$ endowed with their $\eta$-marks, i.e.
\[
\cV^{(\eta)}_T:=\{(x,t,\eta(x,t)) : (x,t) \in \cV_T \}.
\]

Let $N_T$ be the number of points in $\cV_T$. Observe that $N_T$ is almost surely finite (as a matter of fact, it has finite expectation, see Lemma~\ref{lemma:control}) and thus we may number the points in $\cV_T=\{(x_1,t_1),\dots,(x_{N_T},t_{N_T})\}$ in such a way that their marks $\eta(x_i,t_i)$ are ordered in decreasing fashion, i.e. 
\[
\eta(x_1,t_1) \geq \eta(x_2,t_2) \geq \dots \geq \eta(x_{N_T},t_{N_T}).
\] Abbreviating $U_i^{(T)}:=(x_i,t_i)$ and $M_i^{(T)}:=\eta(x_i,t_i)$ for each $i=1,\dots,N_T$, we define
\begin{equation}\label{eq:defpia}
\Pi_T:= \sum_{i=1}^{N_T} M_i^{(T)} \delta_{U_i^{(T)}}
\end{equation} with the convention that $\Pi_T \equiv 0$ whenever $N_T \equiv 0$. Observe that, as a~direct consequence of Proposition~\ref{prop:lpprep}, we have for each $T>0$ the following equality in distribution:
\begin{equation}\label{eq:defpi1}
h(0,T)\overset{d}{=} \max_{\gamma \in \cC_T} \Pi_T\big(\text{graph}(\gamma)\big).
\end{equation} The connection between this representation of the height function $h$ given~by the right-hand side of~\eqref{eq:defpi1} and the limit object $H$ defined in~\eqref{eq:H} is now~clear:
on the one hand, for the discrete model we have
\[
\frac{1}{a_{pT^2}}h(0,T)\overset{d}{=} \max_{\gamma \in \cC_T} \widetilde{\Pi}_T\big(\text{graph}(\gamma)\big)\qquad\text{with}\qquad
\widetilde{\Pi}_T=\sum_{i=1}^{N_T} \tfrac{1}{a_{pT^2}} M_i^{(T)}\delta_{U_i^{(T)}}\,
\] while, on the other hand, for the continuous model we have
\[
H\overset{d}{=} \max_{\gamma \in \cL} \Pi\big(\text{graph}(\gamma)\big)\qquad\text{with}\qquad
\Pi=\sum_{i\geq 1} M_i\delta_{U_i}.
\]
Heuristically, if we can couple the $(U_i^{(T)},M_i^{(T)})$ and the $(U_i,M_i)$ so that $\frac{1}{T}U_i^{(T)}\to U_i$ and $\frac{1}{a_{pT^2}}M_t^{(T)}\to M_i$, then the above representation will yield the convergence $\frac{1}{a_{pT^2}}h(0,T)\to H$. Before proceeding to formalize this heuristic,
let us explain better why $a_{pT^2}$ is the appropriate scaling factor.

To this end, given any $N \in \N$ let us recall the quantity $a_{N}:=F^{-1}(1-1/N)$ defined in~\eqref{eq:defq}. If we consider the order statistics 
\[
M_{1,(N)} \geq M_{2,(N)} \geq \dots \geq M_{N,(N)}
\] of an i.i.d. sample of $N$ random variables with distribution function $F$, then $a_{N}$ represents the order of magnitude of their maximum $M_{1,(N)}$. In particular, it is a standard fact from extreme values theory that, for each $k \in \N$, we have as $N \rightarrow \infty$ the convergence in distribution 
\begin{equation}\label{eq:convM}
\frac{1}{a_N}\big(M_{1,(N)},\dots,M_{k ,(N)}
\big)\,\overset{d}{\lra}\, \big(M_1,\dots,M_k\big),
\end{equation} where $M=(M_n)_{n \in \N}$ is the non-homogeneous Poisson process defined in \eqref{eq:defM}. 
Coming back to the random measure $\Pi_T$, we will show later in Lemma \ref{lemma:control} that, for fixed $p \in (0,1]$, in the limit as $T \rightarrow \infty$,
\[
\frac{N_T}{pT^2} \overset{\P}{\longrightarrow} 1
\] which implies that $\frac{a_{N_T}}{a_{pT^2}} \overset{\P}{\longrightarrow} 1$ in the same limit. 
Thus, it follows from~\eqref{eq:convM} that for each 
$k\in\N$, as $T\to\infty$,
\[
\frac{1}{a_{pT^2}}\big(
M_1^{(T)},\dots,M_k^{(T)}
\big)\overset{d}{\lra}\big(M_1,\dots,M_k\big), 
\] which shows why $a_{pT^2}$ is the correct scaling.

\section{General outline of the proofs}\label{sec:outline}

We now outline the general strategy we will use to prove each of our results. This strategy will rely on showing a few auxiliary and more technical results, whose proofs are deferred to subsequent sections. We begin with the proof of Theorem~\ref{theo:0}.

Let $\Pi$ be the random point measure from Section \ref{sec:res}. For $k \in \N$, let us write
\[
\Pi^{\leq k}:=\sum_{i=1}^k M_i \delta_{U_i} \hspace{1cm}\text{ and }\hspace{1cm}\Pi^{\geq k}:= \sum_{i=k}^\infty M_i \delta_{U_i}
\] and define the quantities
\begin{equation}\label{eq:defhrcont}
H_k:= \sup_{\gamma \in \mathcal{L}} \Pi^{\leq k}\big(\text{graph}(\gamma)\big)\hspace{1cm}\text{ and }\hspace{1cm}R_k:=\sup_{\gamma \in \mathcal{L}} \Pi^{\geq k}\big(\text{graph}(\gamma)\big).	
\end{equation}

Then, we have the following result, analogous to \cite[Lemma~3.1]{HamblyMartin07}.

\begin{lemma}\label{lemma:cont1} The quantities $H_k$ and $R_k$ are both measurable for all $k \in \N$. Furthermore, with probability $1$, $R_k < \infty$ for all $k$ and $R_k \rightarrow 0$ as $k \rightarrow \infty$.	
\end{lemma}

In particular, since $H = R_1$, we obtain from Lemma \ref{lemma:cont1} that $H$ is measurable and a.s. finite. Moreover, we have that $H_k \rightarrow H$ almost surely as $k \rightarrow \infty$ since, for all $k \in \N$,
\begin{align}
0 \leq H - H_k &= \sup_{\gamma \in \mathcal{L}} \Pi\big(\text{graph}(\gamma)\big) - \sup_{\gamma \in \mathcal{L}} \Pi^{\leq k}\big(\text{graph}(\gamma)\big) \nonumber\\
& \leq \sup_{\gamma \in \mathcal{L}} \Pi^{\geq k+1}\big(\text{graph}(\gamma)\big) \nonumber\\
& = R_{k+1} \longrightarrow 0. \label{eq:convhhk}
\end{align}

Our next step will be to show that the supremum from the definition of $H$ in \eqref{eq:H} is almost surely attained (by a unique path).

\begin{proposition}\label{prop:cont1} With probability one, there exists a unique path $\gamma^* \in \cL$ such that $H=\Pi(\text{graph}(\gamma^*))$. In particular, almost surely we have
	\[
	H\,:=\,\max_{\gamma\in\cL}\Pi\big(
		\text{graph}(\gamma)\big).
	\]
\end{proposition} 

Theorem \ref{theo:0} now immediately follows from Lemma \ref{lemma:cont1} and Proposition \ref{prop:cont1}.

We next turn to the proof of Theorem \ref{theo:1}. For this purpose, let us consider the quantities analogous to \eqref{eq:defhrcont} but for the discrete model. That is, being $\Pi_T$ the random measure defined in~\eqref{eq:defpia}, for $k \in \N$ and $T > 0$ let us define
\[
\Pi^{\leq k}_T:= \sum_{i=1}^{k \wedge N_T} M_i^{(T)} \delta_{U_i^{(T)}}\hspace{1cm}\text{ and }\hspace{1cm}\Pi^{\geq k}_T:= 
	\sum_{i=k}^{N_T} M_i^{(T)} \delta_{U_i^{(T)}},
\] with the convention that $\Pi^{\leq k}_T=0$ if $N_T=0$ and $\Pi^{\geq k}_T=0$ if $N_T < k$. 
Furthermore, in analogy with the continuous model, we also define
\[
H_k^{(T)}:=\sup_{\gamma \in \cC_T} \Pi_T^{\leq k}(\text{graph}(\gamma))\hspace{1cm}\text{ and }\hspace{1cm}R_k^{(T)}:=\sup_{\gamma \in \cC_T} \Pi_T^{\geq k}(\text{graph}(\gamma))
\] and set
\begin{equation}\label{eq:defrtilde}
\widetilde{H}^{(T)}_k:= \frac{1}{a_{pT^2}}H^{(T)}_k \hspace{1cm}\text{ and }\hspace{1cm}\widetilde{R}^{(T)}_k:= \frac{1}{a_{pT^2}}R^{(T)}_k
\end{equation} for $a_t$ as in \eqref{eq:defq}. Observe that, if we define
\begin{equation}\label{eq:defck}
\mathcal{C}^{(k)}:=\{ A \subseteq \{1,\dots,k\} : \exists\, \gamma \in \mathcal{L} \text{ with } U_i \in \text{graph}(\gamma)  \text{ for all }i \in A\}
\end{equation} together with
\begin{equation}\label{eq:defckt}
\mathcal{C}^{(k)}_T:=\{ A \subseteq \{1,\dots,k \wedge N_T \} : \exists\, \gamma \in \cC_T \text{ with } U_i^{(T)} \in \text{graph}(\gamma)  \text{ for all }i \in A\}
\end{equation} with the convention that $\mathcal{C}^{(k)}_T=\emptyset$ whenever $N_T=0$ then, upon recalling~\eqref{eq:defpi0},
we can rewrite (similarly to \eqref{eq:H})
\begin{equation}\label{eq:altrep}
H_k = \sup_{A \in \mathcal{C}^{(k)}} \Pi(A) \hspace{1cm}\text{ and }\hspace{1cm}H_k^{(T)}= \sup_{A \in \mathcal{C}^{(k)}_T} \Pi_T(A).
\end{equation} Finally, for $i=1,\dots,N_T$ let us define 
\[
\widetilde{M}_i^{(T)}:= \frac{1}{a_{pT^2}}M_i^{(T)}.
\] The following proposition is a key element in the proof of Theorem \ref{theo:1}:

\begin{proposition}
	\label{prop:coupling} For any $p \in (0,1]$, $\delta > 0$ and $k \in \N$ there exists $T_{k,\delta,p}>0$ such that, for each $T \geq T_{k,\delta,p}$, there exists a coupling of the continuous model and the discrete model with sticking parameter $p$ at time $T$ which satisfies the following properties:
	\begin{enumerate}
		\item [(C1)] $\P\left( \sum_{i=1}^{k \wedge N_T} |M_i - \widetilde{M}_i^{(T)}| > \delta\right) \leq \delta$,
		\item [(C2)] $\P\left( \sum_{i=1}^{k \wedge N_T} \|U_i - r_{p} (\tfrac{1}{T}U_i^{(T)})\| > \delta\right) \leq \delta$,
		\item [(C3)] $\P\left(\cC^{(k)}_T \neq \cC^{(k)}\right) \leq \delta$,
	\end{enumerate} where $r_{p}(x,t):=(\tfrac{x}{p},t)$.
\end{proposition}

\begin{remark}\label{rem:diff1} Proposition \ref{prop:coupling} above is the analogue of \cite[Proposition 3.2]{HamblyMartin07}, obtained in the context of heavy-tailed Last Passage Percolation on $\N \times \N$. However, the proof of this result in our setting is more involved than for LPP, mainly for two reasons. On the one hand, the geometric structure of the set $\cC_T$ of admissible paths is more complicated now, since the attainable points in $\cV_T$ are not located on a regular lattice such as $\N \times \N$ anymore, but rather on a random ``Poissonian lattice''. On the other hand, as opposed to \cite{HamblyMartin07}, the~notion of \textit{compatible points} in the discrete and continuous models (i.e. the conditions used to define the sets $\cC^{(k)}$ and $\cC^{(k)}_T$ in \eqref{eq:altrep}) are not equivalent in~our setting, but are rather only asymptotically equivalent as $T \rightarrow \infty$. These two facts will make the construction of the coupling and verification of its properties significantly more difficult, see Section \ref{coupling} for details.
\end{remark}

The other key element in the proof of Theorem \ref{theo:1} is the following result:

\begin{proposition}\label{prop:remdiscrete} Given $\delta > 0$, for all $k \in \N$ large enough (depending on $\delta$) we have
	\[
	\sup_{T > (2k)^{1/\alpha}} \P\big( \widetilde{R}_k^{(T)} > \delta\big) \leq \delta.
	\]	
\end{proposition}

With these two propositions at our disposal, we can now conclude the proof of Theorem \ref{theo:1}. 

\begin{proof}[Proof of Theorem \ref{theo:1}]
For each $T> 0$ define
\[
k_T:= \max\{ k \in \N : T \geq T_{k,\frac{1}{k}, p}\},
\] where $T_{k,\frac{1}{k},p}>0$ is the one given by Proposition \ref{prop:coupling} for $\delta=\frac{1}{k}$. 
It follows from the finiteness of each $T_{k,\frac{1}{k}, p}$ that $k_T \rightarrow \infty$ as $T \rightarrow \infty$. Then, using Lemma \ref{lemma:cont1} together with Propositions \ref{prop:coupling} and \ref{prop:remdiscrete}, we can construct couplings between the discrete and continuous models at each time $T$ such that, as $T \to \infty$,
\begin{equation}\label{eq:C1}
\sum_{i=1}^{k_T \wedge N_T} |M_i - \widetilde{M}_i^{(T)}|\overset{\P}{\longrightarrow} 0
\end{equation}together with
\begin{equation}\label{eq:C2}
R_{k_T} \overset{\P}{\longrightarrow} 0 \hspace{1cm}\widetilde{R}_{k_T}^{(T)} \overset{\P}{\longrightarrow} 0
\end{equation} and 
\begin{equation}\label{eq:C3}
\lim_{T \rightarrow \infty} \P\left(\cC^{(k_T)}_T \neq \cC^{(k_T)}\right) = 0.
\end{equation} Observe that, under such coupling, we have
\[
H \overset{d}{=}R_1 \hspace{1cm}\text{and }\hspace{1cm}\frac{1}{a_{pT^2}}h(0,T) \overset{d}{=} \widetilde{R}^{(T)}_1.
\] Thus, in order to conclude Theorem \ref{theo:1}, it will suffice to show that as $T \rightarrow \infty$,
\begin{equation}\label{eq:convcoupling}
\widetilde{R}^{(T)}_1 \overset{\P}{\rightarrow} R_1.
\end{equation} To this end, notice that
\[
R_1 - \widetilde{R}^{(T)}_1 = (R_1 - H_{k_T}) + (H_{k_T} - \widetilde{H}^{(T)}_{k_T}) + (\widetilde{H}^{(T)}_{k_T}- \widetilde{R}^{(T)}_1).
\] Since $|R_1 - H_{k_T}| \leq R_{k_T}$ and $|\widetilde{H}^{(T)}_{k_T}- \widetilde{R}^{(T)}_1| \leq \widetilde{R}_{k_T}^{(T)}$, by \eqref{eq:C2} we conclude that, in order for us to obtain \eqref{eq:convcoupling}, it will suffice to show that $H_{k_T} - \widetilde{H}^{(T)}_{k_T} \overset{\P}{\longrightarrow} 0$. To this end, observe that by \eqref{eq:altrep}
\[
H_{k_T} = \max_{A \in \cC^{(k_T)}} \sum_{i \in A} M_i \hspace{1cm}\text{ and }\hspace{1cm}\widetilde{H}_{k_T}^{(T)}=\max_{A \in \cC^{(k_T)}_T} \sum_{i \in A} \widetilde{M}_i^{(T)}.
\] In particular, on the event $\{\cC^{(k_T)}_T = \cC^{(k_T)}\}$ we have that
\[
\big|H_{k_T} - \widetilde{H}^{(T)}_{k_T}\big| \leq \sum_{i=1}^{k_T \wedge N_T} |M_i - \widetilde{M}_i^{(T)}|.
\] 
In view of \eqref{eq:C1} and \eqref{eq:C3}, the former inequality implies that $H_{k_T} - \widetilde{H}^{(T)}_{k_T} \overset{\P}{\longrightarrow} 0$ and thus concludes the proof of Theorem \ref{theo:1}.
\end{proof}

Finally, in order to prove Theorem \ref{theo:2}, we can repeat the same strategy~used to prove Theorem \ref{theo:1}. To be successful this time, the only difference is that we need to replace Propositions \ref{prop:coupling} and \ref{prop:remdiscrete} by stronger versions which are ``uniform over $p \geq T^{-\zeta}$''. In the sequel, since we will consider simultaneously multiple discrete models having different sticking parameters, we will write $\P_p$ (or $\E_p$) to indicate that all the quantities associated with the discrete model which appear in the respective probability (or expectation) correspond to the one with sticking parameter $p$. 

The stronger form of Proposition \ref{prop:coupling} we shall need is the following:

\begin{proposition}
	\label{prop:coupling2} For any $\zeta \in (0,1)$, $\delta >0$ and $k \in \N$ there~exists $T_{k,\delta,\zeta}>1$ such that, for each $T \geq T_{k,\delta,\zeta}$ and $p \in [T^{-\zeta},1]$, there exists a coupling between the continuous model and the discrete model of sticking parameter $p$ at time~$T$ in such a way that the following properties hold:
	\begin{enumerate}
		\item [(C1')] $\sup_{p \geq T^{-\zeta}} \P_p\left( \sum_{i=1}^{k \wedge N_T} |M_i - \widetilde{M}_i^{(T)}| > \delta\right) \leq \delta$,
		\item [(C2')] $\sup_{p \geq T^{-\zeta}} \P_p \left( \sum_{i=1}^{k \wedge N_T} \|U_i - r_{p}(\tfrac{1}{T}U_i^{(T)})\| > \delta\right) \leq \delta$,
		\item [(C3')] $\sup_{p \geq T^{-\zeta}} \P_p \left(\cC^{(k)}_T \neq \cC^{(k)}\right) \leq \delta$,
	\end{enumerate} where, with a slight abuse of notation, in the conditions above and henceforth the expression ``$p \geq T^{-\zeta}$'' stands for $p \in [T^{-\zeta},1]$.
\end{proposition}

Similarly, the stronger form of Proposition \ref{prop:remdiscrete} we shall need is the following:

\begin{proposition}\label{prop:remdiscrete2} Given any $\zeta \in (0, (2-\alpha) \wedge 1)$ and $\delta >0$, for all $k \in \N$ large enough (depending on $\zeta$ and $\delta$) we have
	\[
	\sup_{T > (2k)^{1/\alpha}} \left[ \sup_{p \geq T^{-\zeta}} \P_p\big( \widetilde{R}_k^{(T)} > \delta\big)\right] \leq \delta.
	\]	
\end{proposition}

\begin{remark}Proposition \ref{prop:remdiscrete} (and, more generally, Proposition \ref{prop:remdiscrete2}) above is the analogue of \cite[Proposition 3.3]{HamblyMartin07}, shown in the context of heavy-tailed Last Passage Percolation on $\N \times \N$. Even if our approach is inspired by \cite{HamblyMartin07}, the execution of the proof in our setting will have two important differences. On the one hand, the key estimate \cite[Lemma 3.5]{HamblyMartin07} used to prove the result for LPP will require a different proof in our setting, due to the more complicated geometric structure of the set $\cC_T$ of compatible~paths, see also Remark \ref{rem:diff1}. On the other hand, to obtain the stronger form of this result in Proposition~\ref{prop:remdiscrete2} which allows for vanishing sticking parameters, we will have in fact to \textit{refine} the original estimate in \cite{HamblyMartin07}, see Theorem \ref{theo:bbp} and Lemma \ref{lemma:hml} below.
\end{remark}

With these stronger results at hand, we can now prove Theorem \ref{theo:2}:

\begin{proof}[Proof of Theorem \ref{theo:2}] The result follows by mimicking the proof outlined above for Theorem \ref{theo:1}, the only modification being the use of Propositions \ref{prop:coupling2}-\ref{prop:remdiscrete2} instead of \ref{prop:coupling}-\ref{prop:remdiscrete} to show the analogues \eqref{eq:C1}-\eqref{eq:C2}-\eqref{eq:C3} in this context. We omit the details.
\end{proof}

The following sections are devoted to the proofs of all the auxiliary results we stated in this section. Before we begin, we introduce some further notation to be used extensively in the sequel:
\begin{enumerate}
	\item [$\bullet$] For any finite set $A$, we will denote its cardinal by $|A|$.
	\item [$\bullet$] Given $P=(X,T) \in \R^2$, we will denote its space and time coordinates respectively by $x(P)$ and $t(P)$, i.e. $x(P):=X$ and $t(P):=T$. Likewise, for any $B \subseteq \R^2$, $x(B):=\{ x(b) : b \in B\}$ and $t(B):=\{ t(b) : b \in B\}$ will denote the projection of $B$ onto the space and time coordinates.
	\item [$\bullet$] Given $Q=(X,T,\varepsilon,Z) \in \R^2 \times \{0,1\} \times \R_+$, we shall write $x(Q):=X$, $t(Q):=T$, $\varepsilon(Q):=\varepsilon$ and $z(Q):=Z$ to denote each of its coordinates.
	\item [$\bullet$] For any subset $\mathcal{O} \subseteq \xi$ of space-time points, we shall use the subscripts $\mathcal{O}^{(\varepsilon)}$ and $\mathcal{O}^{(\varepsilon,\eta)}$ to refer to the set of points in $\mathcal{O}$ endowed with their $\varepsilon$ and $(\varepsilon,\eta)$ marks, respectively.  
\end{enumerate}

We are now ready to begin with the proofs.

\section{Proof of Theorem \ref{theo:0}}\label{sec:cont}

As mentioned in Section \ref{sec:outline}, in order to obtain Theorem \ref{theo:0} it will suffice to prove Lemma \ref{lemma:cont1} and Proposition \ref{prop:cont1}. Since the proofs of both results are similar to their analogous counterparts in \cite{HamblyMartin07}, we will give an outline of these proofs and refer the reader to \cite{HamblyMartin07} for some of the details. 

\begin{proof}[Proof of Lemma \ref{lemma:cont1}] Observe that, if we define
\begin{equation}\label{def:c2}
\mathcal{C}:=\{ A \subseteq \N : \exists\, \gamma \in \mathcal{L} \text{ with } U_i \in \text{graph}(\gamma)  \text{ for all }i \in A\},
\end{equation} we have the following alternative representation for $R_1=H$:
\[
R_1= \sup_{A\in \cC} \sum_{i \in A} M_i.
\] As a matter of fact, to obtain $R_1$ it suffices to take the supremum over \textit{finite} subsets $A$ in $\cC$. Indeed, either $R_1 < \infty$ and then we can approximate the sum over any infinite set $A$ by that over $A \cap \{1,\dots,k\}$ for some $k$ large enough, or $R_1=\infty$, in which case we may always find sums over finite sets $A$ which are arbitrarily large. In particular, $R_1$ is the supremum of a countable family of measurable random variables, and is hence measurable itself. The same argument applies to establish the measurability of all other $H_k$ and $R_k$.

To prove the remaining parts of Lemma~\ref{lemma:cont1}, we compare our continuous~model to that in \cite{HamblyMartin07}. To this end, let us consider the 135-degree counterclockwise
 rotation $\cR$ about the origin which maps the translated triangle $\Delta+(0,-1)$ to the region 
\[
\cR(\Delta+(0,-1)):=\{ (x,y) \in \R^2 : 0 \leq x \leq \sqrt{2}\,,
0 \leq y \leq \sqrt{2}-x\}
\] corresponding to the lower half of the square $
[0,\sqrt{2}]^2$ obtained when splitting it into two alongside its diagonal $y=\sqrt{2}-x$ (see Figure \ref{fig:delta}) and, in addition, define $\mathcal{T}(x,y):=\tfrac{1}{\sqrt{2}}\mathcal{R}((x,y)+(0,-1))$. Observe that:
\begin{enumerate}
	\item [a)] given two points $v_1,v_2 \in \Delta$, there exists a path in $\cL$ joining $v_1$ and $v_2$ if and only if there exists an increasing path joining $\mathcal{T}(v_1)$ and $\mathcal{T}(v_2)$, see Figure \ref{fig:delta};
	\item [b)] if one considers a Poisson process on $[0,1]^2 \times \R_+$  with~intensity~measure having density $1_{[0,1]^2}(x) \times \alpha m^{-(\alpha+1)}1_{m>0}$, i.e. the continuous model from~\cite{HamblyMartin07}, then its restriction to $\mathcal{T}(\Delta)$, the lower half of the unit square, corresponds to (the image via $\mathcal{T}$~of) the random measure $2^{-1/\alpha}\Pi$. 
\end{enumerate}
\begin{figure}
	\includegraphics[width=\textwidth]{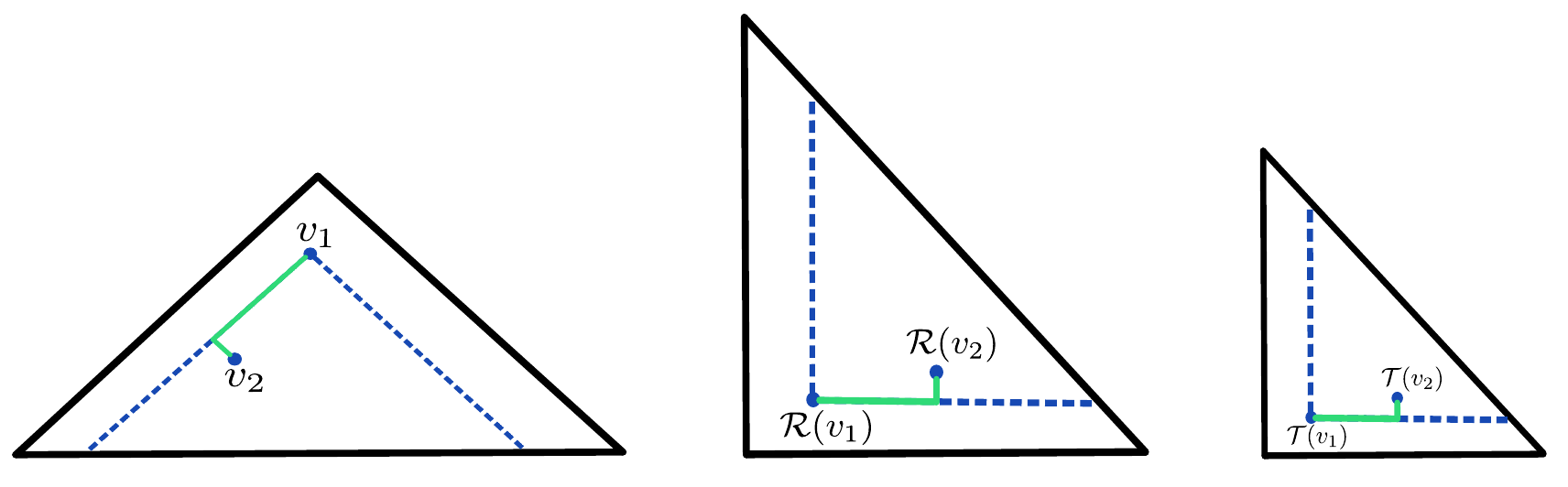}
	\caption{From left to right, the triangles $\Delta$, $\mathcal{R}(\Delta)$ and $\mathcal{T}(\Delta)$, respectively. The points inside the dotted blue triangle with apex $v_1$ in $\Delta$ are precisely those which can be joined with $v_1$ by a path in $\mathcal{L}$, a convenient example of such a path being depicted in green. Observe how this green path is mapped into a increasing path by the scaling-rotation $\mathcal{T}$.}
	\label{fig:delta}
\end{figure} Combining (a) and (b) above (together with the fact that it suffices to take suprema over finite sets $A \in \cC$) yields that $R_k \preceq 2^{1/\alpha} S_k$ for all $k \in \N$, where $\preceq$ stands for stochastic domination and $S_k$ denotes the analogue of $R_k$ corresponding to the continuous model in \cite{HamblyMartin07}. Taking this into consideration, the rest of the proof now follows from \cite[Lemma 3.1]{HamblyMartin07}.
\end{proof}

\begin{proof}[Proof of Proposition \ref{prop:cont1}] The proof is similar to that of \cite[Proposition~4.1]{HamblyMartin07}, we summarize it here for completeness and refer to \cite{HamblyMartin07} for some of the details. 
Recall the definitions of $\cC^{(k)}$ and $H_k$
from~\eqref{eq:defck} and~\eqref{eq:altrep}
and, for each $k \in \N$, define $A_k^* \subseteq \cC^{(k)}$ as the set achieving the maximum in the definition of $H_k$, i.e.
\[
H_k = \Pi( A_k^*)=\sum_{i \in A_k^*} M_i.
\] (Since the height distribution $F$ is assumed to be continuous, the set $A_k^*$ is almost surely unique.) If we view each $A_k^*$ as a random element in $\{0,1\}^\N$ then, since $\{0,1\}^\N$ is compact when endowed with the product topology, it follows~that with probability one the sequence $(A_k^*)_{k \in \N}$ has at least one convergent subsequence (in this topology). If $(A_{k_j}^*)_{j \in \N}$ is any subsequence converging to some set $A^* \subseteq \N$ (notice that the $k_j$'s may be random), then for each $m \in \N$ we have that, for all $j \in \N$ large enough,
\begin{equation}\label{eq:convprod}
A^* \cap \{1,\dots,m\} =  A^*_{k_j} \cap \{1,\dots,m\} 
\end{equation} by definition of convergence in the product topology. Then, it follows that, for all $j$ large enough so that \eqref{eq:convprod} holds,
\[
\left|\sum_{i \in A_{k_j}^*} M_i - \sum_{i \in A^*} M_i \right| \leq \sup_{A \in \cC} \sum_{i \in A\,,\,i\geq m+1} M_i = R_{m+1}
\] which, by Lemma \ref{lemma:cont1} and \eqref{eq:convhhk}, implies that
\begin{equation*} \label{eq:convh}
H = \lim_{j \rightarrow \infty} H_{k_j} = \lim_{j \rightarrow \infty} \sum_{i \in A_{k_j}^*} M_i = \sum_{i \in A^*} M_i.
\end{equation*} Thus, if we manage to show that $A^* \in \cC$, where $\cC$ is as in \eqref{def:c2}, in particular this will imply that the supremum from the definition of $H$ is attained (by whichever path $\gamma^* \in \cL$ collects all the points in $A^*$). This part of the proof differs from that of \cite[Proposition~4.1]{HamblyMartin07}, since our definition of compatible path is not exactly the same as the one in \cite{HamblyMartin07}. To show this, for~each $m \in \N$ let $\gamma^*_m \in \cL$ be a path collecting all points in $A^* \cap \{1,\dots,m\}$. Such a path always exists by \eqref{eq:convprod}, since $A^*_{k_j} \in \cC$ for all $j \in \N$ by definition. Now, since $\mathcal{L}$ is a uniformly bounded and equicontinuous family of paths by definition, by the Arzelà-Ascoli theorem $(\gamma^*_m)_{m \in \N}$ converges uniformly to some path $\gamma^* \in \mathcal{L}$.  Furthermore, since for each $m \in \N$ we have
\[
\{ U_i : i \in A^* \cap \{1,\dots,m\} \} \subseteq \text{graph}(\gamma^*_n)
\] for all $n \geq m$ by choice of the paths $\gamma_n$, by letting $n \rightarrow \infty$ we obtain that 
\[
\{ U_i : i \in A^* \cap \{1,\dots,m\}\} \subseteq \text{graph}(\gamma^*)
\] for all $m \in \N$, which implies that $\{ U_i : i \in A^*\} \subseteq  \text{graph}(\gamma^*)$ and thus $A^* \in \cC$.

Finally, to prove that the maximizing path $\gamma^* \in \mathcal{L}$ is unique, we first establish that the set $A^* \in \cC$ satisfying
\[
H= \sum_{i \in A^*} M_i
\] is a.s.-unique. This can be done as in the proof of \cite[Proposition~4.1]{HamblyMartin07}. Indeed, if there were two maximizing sets $A^*$ then there would exist $k_0 \in \N$ which belongs to only one of them, in which case we would have
\begin{equation}\label{eq:equivh}
\sup_{A \in \cC\,,\,k_0 \in A} \sum_{i \in A} M_i = \sup_{A \in \cC\,,\,k_0 \notin A} \sum_{i \in A} M_i.
\end{equation} Since 
\[
\sup_{A \in \cC\,,\,k_0 \in A} \sum_{i \in A} M_i = M_{k_0} +\sup_{A \in \cC\,,\,k_0 \in A} \sum_{i \in A\,,\,i\neq k_0} M_i, 
\] \eqref{eq:equivh} yields that
\[
M_{k_0} = \sup_{A \in \cC\,,\,k_0 \notin A} \sum_{i \in A} M_i - \sup_{A \in \cC\,,\,k _0\in A} \sum_{i \in A\,,\,i\neq k_0} M_i
\] which, being that both sides of this last equation are independent and $M_{k_0}$ has a continuous distribution, can only occur with zero probability (see \cite{HamblyMartin07} for further details). Taking this into consideration, let us define
\[
U^*:=\{U_i : i \in A^*\} \hspace{1cm} \text{ and }\hspace{1cm}t(U^*):=\{ t(U_i) : i \in A^*\}
\] where $A^*$ is the a.s.-unique set verifying \eqref{eq:equivh}. Note that, by uniqueness of~$A^*$, the graph of any maximizing path $\gamma^* \in \cL$ must contain all the points in $U^*$. In particular, any two maximizing paths must coincide on $t(U^*)$, which is just the collection of times $t \in [0,1]$ in which these points in $U^*$ are collected. By continuity, it then follows that any two such paths must in fact agree on the entire closure $\overline{t(U^*)}$. Now, by adapting the methods in \cite[Section 4]{HamblyMartin07} to our setting (in particular, see the discussion preceding \cite[Theorem 4.4]{HamblyMartin07}) as~explained in the proof of Lemma~\ref{lemma:cont1}, one can show that:
\begin{enumerate}
	\item [i.] $U^*$ is infinite (i.e. any maximizing path collects infinitely many points), 
	\item [ii.] if the closure $\overline{t(U^*)}$ has gaps, i.e. there exist $t_1<t_2 \in \overline{t(U^*)}$ such that $(t_1,t_2) \cap \overline{t(U^*)} = \emptyset$, then any path $\gamma \in \cL$ satisfying that $U^* \subseteq \text{graph}(\gamma)$ must be linear on $(t_1,t_2)$ with slope $\pm 1$, and whether the slope is $+1$ or $-1$ is completely determined by the set $U^*$ (and thus does not depend on the particular choice of maximizing path). 
\end{enumerate} In particular, (ii) implies that any two maximizing paths in $\cL$ must coincide on $[0,1] \setminus \overline{t(U^*)}$. Since we have already argued that they must also coincide on all of $\overline{t(U^*)}$, we conclude that there can only be one maximizing path.
\end{proof}

\section{Proof of Proposition \ref{prop:coupling2}} \label{coupling} 

In this section we construct the coupling between the discrete and continuous models satisfying the properties specified in Proposition \ref{prop:coupling2}. More precisely, for each $T>0$ and $p \in (0,1]$ we shall couple the random variables 
\[
\{ (U_i,M_i) : i \in \N\} \hspace{1cm}\text{ and }\hspace{1cm}\{ (U_i^{(T)},\varepsilon(U_i^{(T)}),M_i^{(T)}) : i=1,\dots,N_T \}
\] in such a way that, for any given $\zeta \in (0,1)$, $k \in \N$ and $\delta > 0$, if $T \geq T_{k,\delta,\zeta}$ then the conditions (C1')-(C2')-(C3') in the statement of Proposition~\ref{prop:coupling2} are all satisfied. Notice that this will immediately prove Proposition \ref{prop:coupling} as well.

To this end, our first step will be to study in more detail the properties of~$\cV_T$, the set of attainable space-time points defined in \eqref{def:attainable} (see also \eqref{eq:ab}). 

\subsection{Interior and boundary of $\cV_T$}\label{sec:boundint}

For fixed $p \in (0,1]$, consider the discrete model with sticking parameter~$p$. Given $T>0$, let $\gamma^{+,T}$ and $\gamma^{-,T}$ respectively denote the rightmost and leftmost paths in $\cC_T$ for this model, i.e. the paths which satisfy, for each $t \in [0,T]$,
\[
\gamma^{+,T}(t)\,=\, \max \{ \gamma(t) : \gamma \in \cC_T \} \hspace{1cm}\text{ and }\hspace{1cm}\gamma^{-,T}(t)\,=\, \min \{ \gamma(t) : \gamma \in \cC_T \},
\] where we have omitted the dependence on $p$ from the notation for simplicity.  
If we consider their time-reversed (càdlàg) versions 
\begin{equation}
	\label{eq:rmlmp}
	s^{+,T}(t):=\gamma^{+,T}((T-t)^- \vee 0) \hspace{1cm}\text{ and }\hspace{1cm}s^{-,T}(t):=\gamma^{-,T}((T-t)^- \vee 0)\,
\end{equation} (the left-hand limit $(T-t)^- \vee 0$ is taken so that both $s^{+,T}$ and $s^{-,T}$ have càdlàg trajectories on $[0,T]$), then it is not hard to check that $s^{+,T}$ and $-s^{-,T}$ are both Poisson processes on $[0,T]$ with jump rate $p$. Note that, if we define the \textit{discrete triangle} $\Delta_T$ generated by the paths $\gamma^{\pm,T}$ as
\[
\Delta_T:=\{ (x,t) \in \Z \times [0,T] : \gamma^{-,T}(t) \leq x \leq \gamma^{+,T}(t)\},
\] then we have (see Figure~\ref{fig:cone})
\[
\cV_T = \xi_T \cap \Delta_T.
\] 
Next, define the \textit{boundary} and \textit{interior} of $\Delta_T$ respectively as
\begin{equation}\label{eq:defboundarytriangle}
\partial \Delta_T := \text{graph}(\gamma^{+,T})\cup\text{graph}(\gamma^{-,T})
\end{equation} and 
\begin{equation} \label{eq:defint}
\Delta^\circ_T:=\Delta_T \setminus \partial \Delta_T = \{ (x,t) \in \Z \times [0,T] : \gamma^{-,T}(t) < x < \gamma^{+,T}(t)\},
\end{equation}
and set
\begin{equation}\label{eq:decomp}
\cV^\circ_T:= \xi_T \cap \Delta^\circ_T \hspace{1cm}\text{ together with }\hspace{1cm}\partial \cV_T:= \xi_T \cap \partial \Delta_T.
\end{equation}
\begin{figure}
	\includegraphics[width=\textwidth]{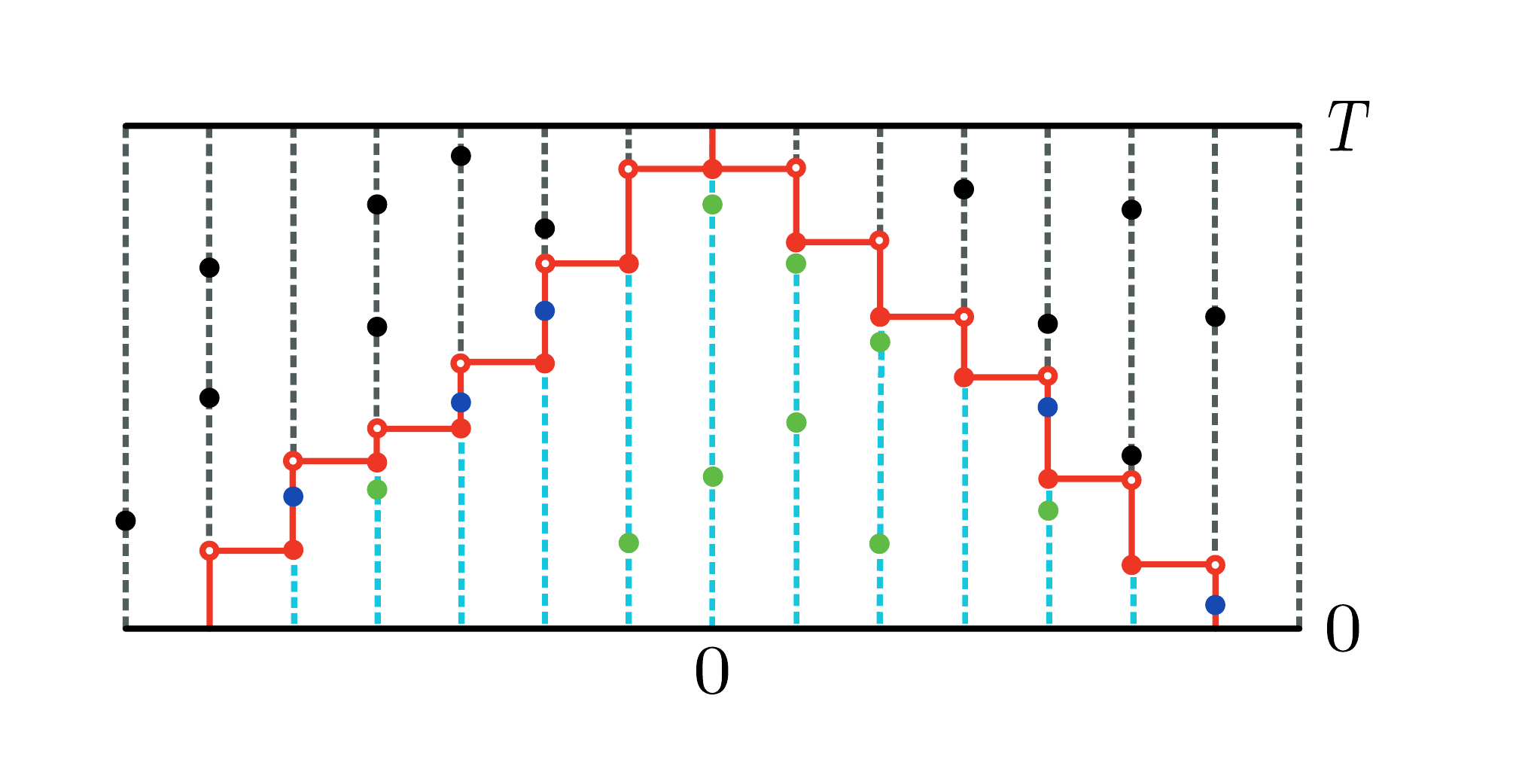}
	\caption{\textit{An illustration of $\cV_T$}. The boundary $\partial \Delta_T$ is colored in red, while the interior $\Delta^\circ_T$ is colored in light blue. The union of the two regions gives $\Delta_T$. The blue and red dots correspond respectively to the non-sticky and sticky points in $\partial \cV_T$, while the green ones to those points in $\cV^\circ_T$. The collection of all red, blue and green dots constitutes $\cV_T$.}
	\label{fig:cone}
\end{figure}
We call the sets $\cV^\circ_T$ and $\partial \cV_T$ the \textit{interior} and \textit{boundary} of $\cV_T$, respectively. In the sequel, we shall use the notation
\[
N^\circ_T:=|\cV^\circ_T| \hspace{1cm}\text{ and }\hspace{1cm}N^\partial_T:= |\partial \cV_T|.
\] By standard properties of the Poisson process, it is not difficult to verify that, conditional on $\partial \cV_T$, $\cV^\circ_T$ is a Poisson process on $\Delta^\circ_T$ whose intensity measure is the restriction of $n_\Z \otimes \lambda_{\R_+}$ to $\Delta^\circ_T$. This is the crucial property on which we will base our coupling. In turn, this characterization yields the following useful asymptotics for the shape of $\Delta_T$ and the quantities $N_T$, $\partial N_T$ and $N^\circ_T$, 
whose proof we delay until Appendix~\ref{sec:proofrelemma} (recall that the notation $\P_p$ and/or $\E_p$ indicates that the model under consideration has sticking parameter $p$):

\begin{lemma}\label{lemma:control} For any $p \in (0,1]$ and $T>0$ we have 
		\begin{equation} \label{eq:expectations}
		\E_p(N^\circ_T)= pT^2 - T + \frac{1-\mathrm{e}^{-pT}}{p}\hspace{1cm}\text{ and }\hspace{1cm}\E_p(N^\partial_T)=2T - \frac{1-\mathrm{e}^{-pT}}{p}.
		\end{equation} Furthermore, for any $\zeta \in (0,1)$ and $\delta>0$ there exist $T_{\delta,\zeta},C_{\delta,\zeta} > 0$ such that, for all $T>T_{\delta,\zeta}$, 
	\begin{equation}
		\label{eq:triangle}
		\sup_{p \geq T^{-\zeta}} \P_p\left( \sup_{0\leq t \leq 1}\left[ \left| \frac{1}{pT}s^{+,T}(tT) - t\right| \vee \left| \frac{1}{pT}s^{-,T}(tT) +t\right|\right]>\delta\right) \leq \mathrm{e}^{-C_{\delta,\zeta} T^{1-\zeta}},
	\end{equation} and
	\begin{equation}
		\label{eq:T2pts}
		\sup_{p \geq T^{-\zeta}} \P_p\left( \left|\frac{N^\circ_T}{pT^2} - 1\right| \vee \left|\frac{N^\partial_T}{2T} - 1\right| \vee \left|\frac{N_T}{pT^2} - 1\right| > \delta \right) \leq \mathrm{e}^{-C_{\delta,\zeta} T^{1-\zeta}}.
	\end{equation}
\end{lemma}

We will use this decomposition of $\cV_T$ into boundary and interior to construct our coupling with the continuous model in the following way:
\begin{enumerate}
	\item [i.] First, sample the boundary $\partial \cV_T$ of $\cV_T$ by constructing a pair of Poisson jump processes with the joint law of $(s^{+,T},-s^{-,T})$, i.e. they agree until their first jump and then behave independently afterwards.
	\item [ii.] Then, conditionally on $\partial \cV_T$, sample the interior $\cV^\circ_T$ and couple it with the continuous model so that it converges to it (in an appropriate way) as $T \rightarrow \infty$.
	\item [iii.] Finally, argue that, since typically $N^\partial_T \ll N^\circ_T$, those space-time points having the $k$ largest block heights will be found (with high probability) always in the interior of $\cV_T$, so that one can disregard points in $\partial \cV_T$ and thus find that the coupling carried out in (ii) verifies the (C')-conditions in the statement of Proposition \ref{prop:coupling2}.  
\end{enumerate}

\subsection{Step 1: sampling $\partial \cV_T$}\label{sec:boundary}

We will begin the construction of our coupling by first sampling the set $\partial \cV_T^{(\varepsilon,\eta)}$ of points in $\partial \cV_T$ endowed with their $\varepsilon$ and $\eta$ marks, i.e.
\[
\partial \cV_T^{(\varepsilon,\eta)}:=\{ (x,t,\varepsilon(x,t),\eta(x,t)) : (x,t) \in \partial \cV_T\}.
\] In order for our coupling to succeed, it shall not be necessary to sample $\partial \cV_T$ in any particular way, any sample will suffice. Therefore, the quickest way to do this would be to consider the Poisson process $\overline{\xi}$ from Section \ref{sec:const} and redo the whole construction of this section to obtain the entire discrete process, in particular yielding a sample of $\cV_T$ and ultimately of $\partial \cV_T$. Another option, which is more technical but does not require constructing the whole process, 
can be briefly summarized as follows:
\begin{enumerate}
	\item [i.] We first sample the time-reversed rightmost/leftmost paths $s^{\pm,T}$ from \eqref{eq:rmlmp}. We do this by sampling $s^{+,T}$ and $-s^{-,T}$ as Poisson jump processes with parameter $p$ which coincide up to their first jump and then evolve independently.
	\item [ii.] We then time-reverse them to obtain the true rightmost/leftmost paths $\gamma^{\pm,T}$ in $\cC_T$.
	\item [iii.] We now construct the sample of $\partial \cV_T$ by appropriately selecting points from the graphs of $\gamma^{\pm,T}$. First, we define the sticky points in our~sample to be exactly those corresponding to a jump of one of the paths $\gamma^{\pm,T}$, i.e. points $(\gamma^{\pm,T}(t),t)$ where $t$ is a discontinuity point of $\gamma^{\pm,T}$. Then, we~add the non-sticky ones by choosing a number $N$ of points uniformly from the curve $\partial \Delta_T$ in \eqref{eq:defboundarytriangle}, where $N$ is an independent random variable with Poisson distribution of parameter $(1-p) \text{length}(\partial \Delta_T)$.
	\item [iv.] Finally, we add independent $\eta$-marks to all the selected points.
\end{enumerate}
Whichever way we wish to proceed, in the end we obtain a sample of $\partial \cV_T^{(\varepsilon,\eta)}$ which, for future reference purposes, we shall denote by $\widehat{\partial \cV}_T^{(\varepsilon,\eta)}$ or simply by $\widehat{\partial \cV}_T$ if we only wish to refer to the space-time points without their marks.

\subsection{Step 2: coupling $\cV^\circ_T$ with the continuous model} \label{sec:couplingstep2}

Our next step is to sample the set $\cV_T^{\circ,(\varepsilon)}$ of points in $\cV_T^{\circ}$  with their $\varepsilon$-marks, i.e.
\[
\cV_T^{\circ,(\varepsilon)}:=\{(x,t,\varepsilon(x,t)) : (x,t) \in \cV^\circ_T\},
\] conditionally on $\partial \cV_T$ and then to couple this set with the sequence $(U_i)_{i \in \N}$ from the continuous model. To this end, we consider a Poisson process 
\[
\Theta = \sum \delta_{(x,t,\e,z)}
\]
on $\R \times [0,1] \times\un \times \R_+$ with intensity measure $\l_{\R}\otimes \l_{[0,1]} \otimes \text{Ber}(p) \otimes \l_{\R_+}$ which is independent of the set $\widehat{\partial \cV}_T^{(\varepsilon,\eta)}$ constructed in the previous subsection. We will interpret the last two coordinates $\varepsilon,z$ in $(x,t,\varepsilon,z)$ as (independent) marks given to the space-time point $(x,t)$. The $z$-marks are not to be confused with the 
$\eta$-marks giving the height of the 
blocks: we use the $z$-marks to induce a random ordering of the points in our sample of $\cV_T^{\circ,(\varepsilon)}$, an ingredient which is necessary to later couple this sample with the $(U_i)_{i \in \N}$.
We shall sometimes write ${\mathbf 1}_{z \in I} \Theta$ or 
${\mathbf 1}_{(x,t)\in D} \Theta$ whenever we wish to restrict the $z$-marks to a certain range $I$ or the space-time points~$(x,t)$ to some region $D$.

Now, for each $T>0$ define $\varphi_T: \R \times [0,1] \rightarrow  \Z \times [0,T]$ the scaling-projection given by the formula
\begin{equation}\label{eq:defvarphi0}
\varphi_{T}(x,t):= (\lfloor Tx +1/2\rfloor,Tt),
\end{equation} and let $\widehat{\xi}_T$  be the collection of scaled-projected space-time points in ${\mathbf 1}_{z<T^2} \Theta$, i.e.
\begin{equation} \label{eq:defc}
\widehat{\xi}_T:=\sum_{(x,t,\varepsilon,z) \in \Theta} \mathbf 1_{z<T^2}\delta_{\varphi_T(x,t)}.
\end{equation} \begin{figure}
\includegraphics[width=\textwidth]{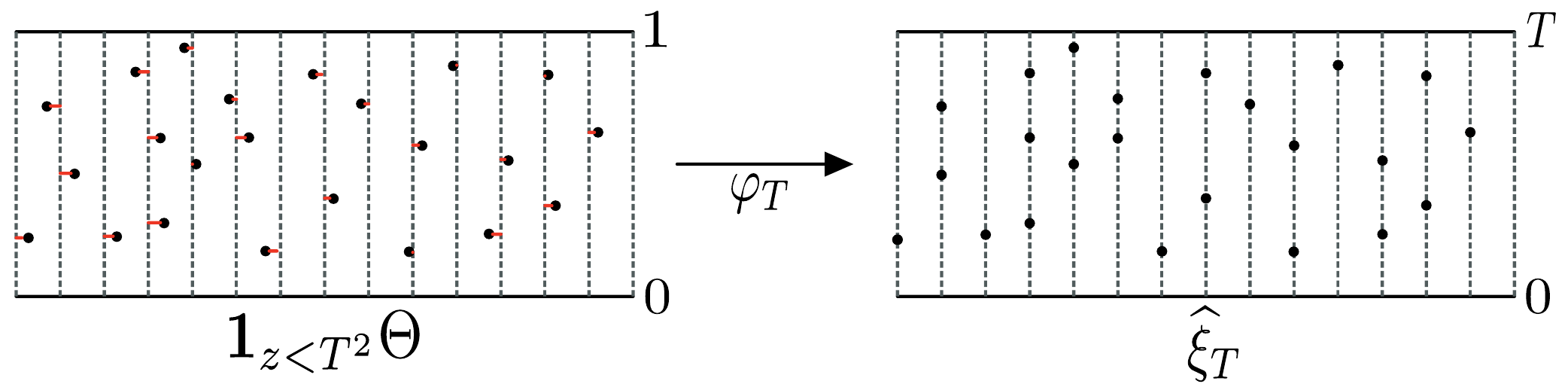}
\caption{\textit{The sample $\widehat{\xi}_T$}. On the left we can see an illustration of $\mathbf{1}_{z<T^2}\Theta$, the set of points in $\Theta$ having a $z$-mark strictly less than $T^2$. On the right, we have the sample $\widehat{\xi}_T$, which is obtained by applying the scaling-projection $\varphi_T$ to the points in the left picture.}
\label{fig:g}
\end{figure}Put into words, $\widehat{\xi}_T$ is the collection of all $(k,t) \in \Z\times \R_+$ such that $\Theta$ has a point at time $t/T$ in the space interval $T^{-1} [k-\tfrac{1}{2}, k+\tfrac{1}{2})$ having a $z$-mark which is smaller than $T^2$, see Figure \ref{fig:g}. It is straightforward to check that $\widehat{\xi}_T$ is a Poisson process on $\Z \times [0,T]$ with intensity $n_\Z \otimes \l_{[0,T]}$, that is, a sample of the set $\xi_T$ in \eqref{eq:ab} (hence the notation $\widehat{\xi}_T$). Therefore, upon~noticing that $\Delta^\circ_T$ is a measurable function of $\partial \cV_T^{(\varepsilon,\eta)}$, i.e.
\begin{equation}
	\label{eq:altercirc}
\Delta^\circ_T = \{ (k,t) \in \Z \times [0,T] : \exists\, t' > t \text{ such that }(k,t') \in  \partial \cV_T \text{ and }\varepsilon(k,t')=1\},
\end{equation} it follows that the collection 
\begin{equation*}\label{eq:constint}
\widehat{\cV}^\circ_T:= \widehat{\xi}_T \cap \Delta^\circ_T,
\end{equation*} where $\Delta^\circ_T$ is defined via \eqref{eq:altercirc} but now using the sample $\widehat{\partial \cV}_T^{(\varepsilon,\eta)}$ constructed in Section \ref{sec:boundary}, is itself a sample of the interior $\cV^{\circ}_T$.

The next step is to add the $\varepsilon$-marks to $\widehat{\cV}^\circ_T$. To this end, we note that, since all points in $\Theta$ have different time coordinates, each point $(k,t) \in \widehat{\xi}_T$ comes from an unique point $Q \in {\mathbf 1}_{z<T^2}\Theta$. Therefore, we may assign marks to each $(k,t) \in \widehat{\xi}_T$ in an unambiguous way by simply allowing it to inherit the marks of its corresponding point $Q \in {\mathbf 1}_{z<T^2}\Theta$, i.e. by defining
\[
\varepsilon(k,t):=\varepsilon(Q) \hspace{1cm}\text{ and }\hspace{1cm}z(k,t):=z(Q).
\]
Furthermore, since $|\widehat{\cV}^\circ_T| < \infty$ because $n_\Z \otimes \lambda_{\R_+}(\Delta^\circ_T) < \infty$ almost surely, we may number the points in $\widehat{\cV}^\circ_T$ by ordering their $z$-marks in increasing fashion, i.e. we may write
\[
\widehat{\cV}^\circ_T=\{ U^{\circ,(T)}_i : i=1,\dots,N^\circ_T\},
\] where $N^\circ_T=|\widehat{\cV}^\circ_T|$ and the $U^{\circ,(T)}_i$ verify that $0< z(U^{\circ,(T)}_1) < \dots < z(U^{\circ,(T)}_{N^\circ_T})$. It follows from the above discussion that the collection
\begin{equation} \label{eq:defupsilon}
\widehat{\cV}_T^{\circ,(\varepsilon)}=\{ (U^{\circ,(T)}_i,\varepsilon(U^{\circ,(T)}_i)) : i=1,\dots,N^\circ_T\}
\end{equation} is a sample of $\cV_T^{\circ,(\varepsilon)}$.

Finally, we couple this set $\widehat{\cV}_T^{\circ,(\varepsilon)}$ together with the space-time points $(U_i)_{i \in \N}$ from the continuous model. To do this, 
we first observe that, by Lemma \ref{lemma:control}, as $T \rightarrow \infty$ the scaled discrete triangle $\frac{1}{T}\Delta_T$ should resemble a (continuous) triangle but with slope $p$, i.e. it should be close to the set
\begin{equation} \label{eq:deftrianglep}
\Delta^{(p)}:=\big\lbrace
(x,t)\in\R^2 : 0\leq t\leq 1\,,\, |x|\leq p(1-t)
\big\rbrace=r_{1/p}(\Delta),
\end{equation} where $r_{1/p}$ is the inverse of the map $r_p$ in the statement of Proposition \ref{prop:coupling}. Hence, since $|\widehat{\cV}_T^{\circ}| \to \infty$ by Lemma \ref{lemma:control} and the locations of points in $\widehat{\cV}_T^{\circ}$ are independent and uniformly chosen from $\Delta^\circ_T$, it follows that as $T \rightarrow \infty$ the law of the collection $\frac{1}{T}\widehat{\cV}_T^{\circ}$ is approximately that of an i.i.d. sequence of uniform random variables on~$\Delta^{(p)}$ (because the boundary of $\Delta^{(p)}$ has measure zero). Having this in mind, the most natural thing to~do is first to couple $\frac{1}{T}\widehat{\cV}_T^{\circ}$ with an i.i.d. sequence $(U_i^{(p)})_{i \in \N}$ with uniform distribution on $\Delta^{(p)}$ (that is, with the space-time points from a ``continuous model on $\Delta^{(p)}$'') in such a way that $\frac{1}{T}\widehat{\cV}_T^{\circ}$ converges to this sequence in an appropriate fashion. We will construct these points $U_i^{(p)}$ from the process $\Theta$. Then, we will scale the variables $U_i^{(p)}$ via $r_p$ to obtain a sample of the space-time points $(U_i)_{i \in \N}$ from our original continuous model, thus producing the desired coupling.

To be more precise, let us consider the collection of marked points 
\[
\mathcal{U}^{(p)}:=\{ (x,t,\varepsilon,z) \in \Theta : (x,t) \in \Delta^{(p)}\}.
\] Since $\Delta^{(p)}$ has a finite Lebesgue measure, for each $K > 0$ the set $\mathcal{U}^{(p)}$ contains only finitely many points with a $z$-mark which belongs to $[0,K]$. In particular, the collection of $z$-marks of points in $\mathcal{U}^{(p)}$ is discrete and thus we may number the points in $\mathcal{U}^{(p)}$ by ordering their $z$-marks in increasing fashion, i.e. we may write $\mathcal{U}^{(p)}$ as
\[
\mathcal{U}^{(p)}=\{ (U_i^{(p)},\varepsilon(U_i^{(p)}),z(U_i^{(p)})) : i \in \N\}
\] where $0< z(U_1^{(p)}) < z(U_2^{(p)}) < \dots$. It is not hard to show that the sequence $(U_i^{(p)})_{i \in \N}$ is i.i.d. with uniform distribution on the triangle $\Delta^{(p)}$. 
Furthermore, it is the natural choice to couple with $\frac{1}{T}\widehat{\cV}^\circ_T$ since
\begin{align*}
\frac{1}{T}\widehat{\cV}^\circ_T&=\{\tfrac{1}{T} \varphi_T(x,t) : \varphi_T(x,t) \in \widehat{\xi}_T \cap \Delta^\circ_T\}\\& = \{\tfrac{1}{T} \varphi_T(x,t) : \tfrac{1}{T}\varphi_T(x,t) \in \tfrac{1}{T}\widehat{\xi}_T \cap \tfrac{1}{T}\Delta^\circ_T\}\\
& \approx \{ (x,t) : (x,t,\varepsilon,z) \in \Theta \text{ for some }(\varepsilon,z)\,,\, (x,t) \in \Delta^{(p)}\} = (U_i^{(p)})_{i \in \N},
\end{align*} where the approximation $\approx$ on the third line holds true as $T \rightarrow \infty$ because under this limit we have  $\frac{1}{T}\varphi_T(x,t) \approx (x,t)$, $\mathbf{1}_{z < T^2} \to 1$ and also $\frac{1}{T}\Delta^\circ_T \approx \Delta^{(p)}$ as argued above (we will give precise details about this approximation in Section \ref{sec:condc2}). 
In particular, the sequence 
\[
\widehat{\mathcal{U}}:=(\widehat{U}_i)_{i \in \N}
\] where, for each $i \in \N$, we define
\begin{equation} \label{eq:deflambda}
\widehat{U}_i:=r_p(U_i^{(p)}),
\end{equation} is an i.i.d. sequence of random variables with uniform distribution on~$\Delta$ and therefore has the same law as the sequence $(U_i)_{i \in \N}$ from the continuous~model.
Hence, the pair $(\widehat{\cV}^{(\varepsilon)}_T,\widehat{\mathcal{U}})$ constitutes the desired coupling. 

\begin{remark}
This coupling has been specifically designed so that it satisfies 
\begin{equation}\label{eq:convc2}
r_p\big(\tfrac{1}{T} U_i^{\circ,(T)}\big) \overset{\P}{\longrightarrow} \widehat{U}_i
\end{equation} for each $i \in \N$, see Section \ref{sec:condc2} below.\footnote{Notice that the random variable $U_i^{\circ,(T)}$ is only well-defined if there are enough points in $\widehat{\cV}^\circ_T$, i.e. on the event that $N^\circ_T > i$, which does not have full probability. However, since 
$\P(N^\circ_T > i)\to 1$ because $N^\circ_T \to \infty$ in probability, we can still make sense of \eqref{eq:convc2}.} This is the crucial property which will allow us to show conditions (C2') and (C3') later on.
\end{remark}

\subsection{Step 3: coupling the block heights}\label{sec:blockheights}

As pointed out in Section \ref{sec:lpprep},  if we consider the order statistics 
\[
M_{1,(N)} \geq M_{2,(N)} \geq \dots \geq M_{N,(N)}
\] of an i.i.d. sample of $N$ random variables with distribution function $F$ then, as $N \rightarrow \infty$, we have the convergence in distribution 
\[
	\big(a_N^{-1}M_{1,(N)},\dots,a_N^{-1}M_{N ,(N)} 
	\big)\overset{d}{\lra} M,
\] where $M$ denotes the non-homogeneous Poisson process on $\R_+$ defined in~\eqref{eq:defM}. Therefore, by eventually enlarging our probability space if necessary, we may assume that in it there exist an infinite sequence $(M_n)_{n \in \N}$ and for each $N \in \N$ a finite sequence $(M^{\circ}_{i,(N)})_{i=1,\dots,N}$ such that the following are satisfied:
\begin{enumerate}
	\item [M0.] Both $(M_n)_{n \in \N}$ and the finite sequences $(M^{\circ}_{i,(N)})_{i=1,\dots,N}$ for each $N \in \N$ are independent of the samples $\widehat{\partial \cV}^{(\varepsilon,\eta)}_T$, $\widehat{\cV}^{\circ,(\varepsilon)}_T$ and $\widehat{\mathcal{U}}$,
	\item [M1.] $(M_i)_{i \in \N}\overset{d}{=}M$,
	\item [M2.] $(M^{\circ}_{i,(N)})_{i \in \N}\overset{d}{=}(M_{1,(N)},\dots,M_{N,(N)})$ for each $N \in \N$,
	\item [M3.] $a_N^{-1}M^\circ_{i,(N)} \overset{as}{\lra} M_i$ as $N \rightarrow \infty$ for each $i \in \N$. 
\end{enumerate} 

Thus, if we define
\begin{equation}\label{eq:defmint}
\mathfrak{C}:=\{ (\widehat{U}_i, M_i) : i \in \N\},
\end{equation} for $\widehat{U}_i$ as defined in \eqref{eq:deflambda}, then $\mathfrak{C}$ is a sample of the continuous model $(U_i,M_i)_{i \in \N}$. On the other hand, if we recall \eqref{eq:defupsilon} and for $i=1,\dots,N^\circ_T$ define 
\[
M^{\circ,(T)}_i:= M^\circ_{i,(N^\circ_T)},
\] then the collection
\[
\widehat{\cV}^{\circ,(\varepsilon,\eta)}_T:=\{ (U^{\circ,(T)}_i,\varepsilon(U^{\circ,(T)}_i),M^{\circ,(T)}_i) : i=1,\dots,N^\circ_T\}
\] is a sample of the set
\[
\cV^{\circ,(\varepsilon,\eta)}_T:=\{ (x,t,\varepsilon(x,t),\eta(x,t)) : (x,t) \in \cV^\circ_T\}.
\] 		
Finally, recalling also \eqref{eq:decomp}, it then follows that the collection
\[
\mathfrak{D}_T:= \widehat{\partial \cV}^{(\varepsilon,\eta)}_T \cup \widehat{\cV}^{\circ,(\varepsilon,\eta)}_T
\] is a sample of the discrete model $\{ (U^{(T)}_i,\varepsilon(U^{(T)}_i),M^{(T)}_i) : i=1,\dots,N_T\}$ and thus the pair $(\mathfrak{C},\mathfrak{D}_T)$ constitutes the desired coupling between the continuous and discrete models. 

\begin{remark} This sample has been specifically designed so that it satisfies 
\[
\frac{1}{a_{pT^2}}M^{\circ,(T)}_i \overset{\P}{\longrightarrow} M_i
\] for each $i \in \N$, see Lemma \ref{lemma:c1.a} below for details. Notice that this convergence is for $\eta$-marks associated with space-time points in $\cV^\circ_T$, and \textit{not} for all points in $\widehat{\cV}_T$ (as, perhaps, one would expect). The reason we have chosen to proceed in this way is because for our purposes we shall require not only the $\eta$-marks to converge but also their associated space-time positions and, in light of our previous construction (see~\eqref{eq:convc2}), this will only hold for points in the interior. In the end, since we will show that points in the boundary can be disregarded, this will not make any difference.
\end{remark}

\subsection{Step 4: verifying the (C')-conditions}

We now verify that the coupling $(\mathfrak{C},\mathfrak{D}_T)$ constructed in the previous sections satisfies all the (C')-conditions appearing in the statement of Proposition \ref{prop:coupling2}. To simplify the notation, in the sequel we shall drop the ``hat-notation'' used before to refer to our coupled samples and instead write
\[
\mathfrak{C}=\{(U_i,M_i) : i \in \N\}
\] where $U_i$ denotes the space-time position of the $i$-th largest weight $M_i$ and, similarly,
\[
\mathfrak{D}_T=\{ (U_i^{(T)},\varepsilon(U_i^{(T)}),M_i^{(T)}) : i=1,\dots,N_T\}
\] where $U_i^{(T)}$ will denote the space-time position in the coupled discrete model of the block with the $i$-th largest height $M_i^{(T)}$ among all those in $\cV_T$. Finally, recall that we write $U_i^{\circ,(T)}$ to denote the space-time position of the block with the $i$-th largest height $M_i^{\circ,(T)}$ among all those in (our coupled version of) $\cV^\circ_T$.
We remind the reader that, as a general rule,  we use $T$ (as a super/subscript) to~denote objects associated with the discrete model, whereas objects denoted without $T$ will correspond to the continuous one.

Let us now verify all three (C')-conditions.  
To keep the exposition as simple as possible, in the next three subsections we will try to convey the main ideas leading to the verification of each (C')-condition, deferring the proofs of some of the more technical aspects to the Appendix~\ref{sec:techapp}.

\subsubsection{Condition (C1')}\label{sec:condc1}

Our first step is to show that it suffices to consider the case in which there are enough points in the interior of $\cV_T$. Indeed, for any $T>1$ and $p \in [T^{-\zeta},1]$, by the union bound we have that
\begin{align}
\P_p\left( \sum_{i=1}^{k \wedge N_T} |M_i - \widetilde{M}_i^{(T)}| > \delta\right) \leq \P_p&( N^\circ_T < k)\nonumber\\ 
&+ \P_p\left( \sum_{i=1}^{k} |M_i - \widetilde{M}_i^{(T)}| > \delta\,,\, N^\circ_T \geq k\right) \label{eq:c1.1}.
\end{align} Since $pT^2 \geq T^{2-\zeta} \rightarrow \infty$ as $T \rightarrow \infty$, by \eqref{eq:T2pts} we have that 
\begin{equation} \label{eq:boundnt}
\lim_{T \rightarrow \infty} \left[\sup_{p \geq T^{-\zeta}} \P_p (N^\circ_T < k)\right] = 0,
\end{equation} so that it will suffice to estimate the last probability in \eqref{eq:c1.1}. To this end, consider the event 
\begin{equation} \label{eq:defomega1}
\Omega_{k}:=\{N^\circ_T\geq k\}\cap \{ M_i^{(T)} = M_i^{\circ,(T)} \text{ for all }i=1,\dots,k\}
\end{equation} i.e. $\Omega_k$ is the event that the $k$ largest heights in $\cV_T$ all correspond to points in the interior $\cV^\circ_T$. Then, by the union bound again, 
the probability in \eqref{eq:c1.1} can be bounded from above by 
\begin{equation} \label{eq:boundc1}
\P_p(\Omega_k^c)
+ \P_p\left( \sum_{i=1}^{k} |M_i - \widetilde{M}_i^{\circ,(T)}| > \delta\,,\,N^\circ_T \geq k\right)
\end{equation} where, recalling \eqref{eq:defmint}, for $i=1,\dots,N^\circ_T$ we set 
\begin{equation}\label{def:mtildeint}
\widetilde{M}_i^{\circ,(T)}:= \frac{1}{a_{pT^2}}M_i^{\circ,(T)}=\frac{1}{a_{pT^2}}M^\circ_{i,(N^\circ_T)}.
\end{equation}
Now, on the one hand, since $\frac{a_{N^\circ_T}}{a_{pT^2}} \overset{\P}{\longrightarrow} 1$ uniformly over $p \in [T^{-\zeta},1]$ by~\eqref{eq:T2pts}, it follows from property~(M3) from the construction in Section \ref{sec:blockheights} that:

\begin{lemma} \label{lemma:c1.a} For $\widetilde{M}_i^{\circ,(T)}$ as defined in \eqref{def:mtildeint}, 
\begin{equation} \label{eq:c1b}
		\lim_{T \rightarrow \infty} \left[\sup_{p \geq T^{-\zeta}} \P_p\left( \sum_{i=1}^{k} |M_i - \widetilde{M}_i^{\circ,(T)}| > \delta\,,\,N^\circ_T \geq k\right)\right]=0.
\end{equation}
\end{lemma}

On~the other hand, since the space-time position associated with $M_i^{(T)}$, the $i$-th largest height from the discrete model, is uniformly distributed among all points in $\cV_T$, by \eqref{eq:T2pts} it follows that the probability that $M_i^{(T)}$ is the height corresponding to a point in $\partial \cV_T$ (given $N^\partial_T$ and $N_T$) is \textit{roughly}
\[
\frac{N^\partial_T}{N_T} \approx \frac{2T}{pT^2}=\frac{2}{pT} \leq \frac{1}{T^{1-\zeta}} \longrightarrow 0
\] uniformly over $p \in [T^{-\zeta},1]$ for each $i=1,\dots,k$, so that by the union bound we obtain:

\begin{lemma}\label{lemma:c1.b} For $\Omega_k$ as defined in \eqref{eq:defomega1}, 
	\begin{equation}
		\label{eq:boundomega}
		\lim_{T \rightarrow \infty} \left[\inf_{p \geq T^{-\zeta}}  \P_p(\Omega_k)\right] = 1.
	\end{equation}
\end{lemma}
Details of the proofs of both lemmas can be found in Appendix~\ref{sec:techappc1}. Finally, in~light of the bound in \eqref{eq:boundc1},  \eqref{eq:c1b} and \eqref{eq:boundomega} combined immediately yield (C1').


\subsubsection{Condition (C2')}\label{sec:condc2}

Given $T>1$ and $p \in [T^{-\zeta},1]$, by the union bound and the construction of $\mathfrak{C}$ in Section \ref{sec:couplingstep2}, we can bound the probability
\[
\P_p \left( \sum_{i=1}^{k \wedge N_T} \|U_i - r_{p}(\tfrac{1}{T}U_i^{(T)})\| > \delta\right)
\] from above by 
\[
\P_p(\Omega_k^c)+\P_p \left( \sum_{i=1}^{k} \|r_p(U^{(p)}_i -\tfrac{1}{T}U_i^{\circ,(T)})\| > \delta\,,\,N^\circ_T \geq k\right).
\] 
In view of~\eqref{eq:boundomega},
to establish (C2'), it will suffice to show that
\begin{equation}\label{eq:suff1}
\lim_{T \rightarrow \infty} \left[ \sup_{p \geq T^{-\zeta}} \P_p \left( \sum_{i=1}^{k} \|r_p(U^{(p)}_i -\tfrac{1}{T}U_i^{\circ,(T)})\| > \delta\,,\,N^\circ_T \geq k\right)\right]=0,
\end{equation} 
i.e. it is enough to assume that the $k$ largest heights from the discrete model correspond to points in the interior $\cV^\circ_T$. 

In order to show \eqref{eq:suff1}, we first observe that the term $\|r_p(U^{(p)}_i -\tfrac{1}{T}U_i^{\circ,(T)})\|$ is small whenever $U_i^{\circ,(T)}=\varphi_T(U^{(p)}_i)$ with $\varphi_T$ the scaling-projection from \eqref{eq:defvarphi0}, i.e. whenever the space-time position of the $i$-th largest (interior) height in the discrete model coincides with the scaled position of the $i$-th largest weight from the continuous one. Indeed, by definition of $\varphi_T$, 
\begin{equation} \label{eq:boundrp}
	\|r_p(U^{(p)}_i -\tfrac{1}{T}U_i^{\circ,(T)})\| = \frac{1}{p}\left|x(U^{(p)}_i) - \frac{1}{T}\lfloor Tx(U^{(p)}_i) + 1/2\rfloor\right| \leq \frac{1}{2pT} \leq T^{\zeta -1}
\end{equation} which shows it is small since $T^{\zeta-1} \rightarrow 0$ as $T \rightarrow \infty$. Nevertheless, it may not always be the case that $U_i^{\circ,(T)}=\varphi_T(U^{(p)}_i)$ for all $i=1,\dots,k$, there are issues for points which are close to the boundaries: points belonging to the $k$ largest weights in the continuous model may fall outside $\Delta_T$ via the mapping $\varphi_T$ and thus not correspond to any $U_i^{\circ,(T)}$, while points outside the continuous triangle $\Delta$ may fall inside the discrete triangle $\Delta_T$ via $\varphi_T$ and thus become part of the $k$ largest heights in the discrete model. The proof of \eqref{eq:suff1} amounts to showing that these undesirable boundary effects occur with vanishing probability as $T \to \infty$. To carry out the proof, we first introduce the favorable event $\Omega'_k$ in which none of these boundary effects occur, i.e.
\begin{equation} \label{eq:defomega2}
\Omega'_k:=\{ N^\circ_T \geq k\} \cap \{ U_i^{\circ,(T)}=\varphi_T(U^{(p)}_i) \text{ for all }i=1,\dots,k\}.
\end{equation} In light of \eqref{eq:boundrp},
since $T^{\zeta-1} \rightarrow 0$ as $T \rightarrow \infty$ we conclude that for all $T$ large enough (depending only on $\zeta$ and $\delta$) the event in \eqref{eq:suff1} cannot occur on $\Omega'_k$. 
Now, let $\mathcal{P}_{T,\delta'}$ be the complement of the event in \eqref{eq:triangle} with $\delta'>0$ in place of $\delta$.
By the estimate in \eqref{eq:triangle}, the statement in \eqref{eq:suff1} will follow once we show that the event $\Omega'_k$ occurs with high probability whenever the discrete triangle $\frac{1}{T}\Delta_T$ is close enough to $\Delta^{(p)}$ (the $p$-slope triangle from \eqref{eq:deftrianglep}), i.e. once we show that
\begin{equation}
	\label{eq:suff2}
\lim_{\delta' \rightarrow 0^+} \limsup_{T \rightarrow \infty} \left[\sup_{p \geq T^{-\zeta}} \P_p((\Omega'_k)^c \cap \mathcal{P}_{T,\delta'} )\right]=0.
\end{equation} 

In order to check \eqref{eq:suff2}, we first note that, since $U_i^{\circ,(T)} \in \Delta^\circ_T$ and $U^{(p)}_i~\in~\Delta^{(p)}$ for each $i=1,\dots,N^\circ_T$ by construction of our coupling, for any such $i$ we have the following implication:
\[
U_i^{\circ,(T)}=\varphi_T(U^{(p)}_i) \Longrightarrow \varphi_T(U^{(p)}_i) \in \Delta^\circ_T \text{ and }U_i^{\circ,(T)} \in \varphi_T(\Delta^{(p)}). 
\] In other words, if the space-time positions  $U_i^{\circ,(T)}$ and $\varphi_T(U^{(p)}_i)$ agree then they must lie in the intersection between $\Delta^\circ_T$ and (the $\varphi_T$-scaling of) $\Delta^{(p)}$. 
In particular, we have the inclusion
\[
\Omega'_k \subseteq \{ N^\circ_T \geq k\} \cap \left[\bigcap_{i=1}^k \left\{\varphi_T(U^{(p)}_i), U^{\circ,(T)}_i \in [ \Delta^\circ_T \cap \varphi_T(\Delta^{(p)})]\right\}\right].
\] 
A moment's thought reveals that this inclusion above is in fact an equality, i.e. $U_i^{\circ,(T)}=\varphi_T(U^{(p)}_i) \text{ for all }i=1,\dots,k$ if and only if both the $\varphi_T(U^{(p)}_i)$'s and the $U^{\circ,(T)}_i$'s all fall in the intersection of $\Delta^\circ_T$ and (the $\varphi_T$-scaling of)~$\Delta^{(p)}$. 
Therefore, by \eqref{eq:boundnt} and the union bound, in order to obtain \eqref{eq:suff2} it will suffice to prove the following lemma:

\begin{lemma} \label{lemma:c2.1} If $\mathcal{P}_{T,\delta'}$ denotes the complement of the event appearing in \eqref{eq:triangle} (with $\delta'>0$ in~place of~$\delta$) then, for each $i=1,\dots,k$, 
\begin{equation} \label{eq:suff3}
\lim_{\delta' \rightarrow 0^+}\limsup_{T \rightarrow \infty} \left[\sup_{p\geq T^{-\zeta}} \P_p\left(\{ \varphi_T(U^{(p)}_i) \notin \Delta^\circ_T\} \cap \mathcal{P}_{T,\delta'} \right)\right]=0
\end{equation} and
\begin{equation} \label{eq:suff4}
\lim_{\delta' \rightarrow 0^+}\limsup_{T \rightarrow \infty} \left[\sup_{p \geq T^{-\zeta}} \P_p\left(\{N^\circ_T\geq k\} \cap \{U^{\circ,(T)}_i \notin \varphi_T(\Delta^{(p)})\} \cap \mathcal{P}_{T,\delta'}\right)\right]=0.
\end{equation}
\end{lemma}

The heuristics behind the proof of Lemma \ref{lemma:c2.1} are quite simple. Indeed, since $\varphi_T(\frac{x}{T},\frac{t}{T}) \approx (x,t)$ (on the scale of $T$) for all $T$ sufficiently large, whenever $\frac{1}{T}\Delta_T \approx \Delta^{(p)}$ we see that $\Delta_T \approx \varphi_T(\Delta^{(p)})$ as~well, so that the intersection $\Delta^\circ_T \cap \varphi_T(\Delta^{(p)})$ will cover most of both $\Delta^\circ_T $ and $\varphi_T(\Delta^{(p)})$, therefore making the events $\{  \varphi_T(U^{(p)}_i) \notin \Delta^\circ_T\}$ and $\{U^{\circ,(T)}_i \notin \varphi_T(\Delta^{(p)})\}$ both extremely unlikely. The full proof of Lemma~\ref{lemma:c2.1} is given in Appendix~\ref{sec:techappc2}. 
In conclusion, this result implies \eqref{eq:suff2} and, as argued above, from this (C2') now follows.

\subsubsection{Condition (C3')}\label{sec:condc3}

We first observe that, by our work in the preceding two subsections, to show condition (C3') we may assume that the $k$ largest block heights in the discrete model lie in $\cV^\circ_T$ and that the space-time points to which they correspond are coupled with those of the $k$ largest weights in the continuous model. Indeed, if we set $\Omega''_k:= \Omega_k \cap \Omega'_k$, where $\Omega_k$ and $\Omega'_k$ are as in \eqref{eq:defomega1} and \eqref{eq:defomega2} respectively, then, by \eqref{eq:boundomega}, \eqref{eq:triangle} and \eqref{eq:suff2}, we see that in order to establish condition (C3') it will be enough to show that
\begin{equation} \label{eq:c3suff1}
\lim_{T \rightarrow \infty} \left[ \sup_{p \geq T^{-\zeta}} \P_p(\{ \cC^{(k)} \nsubseteq \cC^{(k)}_T \} \cap \Omega''_k) \right]= 0
\end{equation} together with
\begin{equation} \label{eq:c3suff2}
	\lim_{T \rightarrow \infty} \left[ \sup_{p \geq T^{-\zeta}} \P_p( \{ \cC^{(k)}_T\nsubseteq \cC^{(k)}\} \cap \Omega''_k ) \right]= 0. 
\end{equation} 
We prove only~\eqref{eq:c3suff1}, as the argument for \eqref{eq:c3suff2} is very similar.
Before jumping to the proof,
let us explain the main challenges involved.
Suppose that two points $U_i$ and $U_j$
in the continuous triangle $\Delta$ are such that the 
segment $\overline{U_i U_j}$
has a slope with absolute value close to 1 (where by slope we mean $\frac{x(U_i)-x(U_j)}{t(U_i)-t(U_j)}$, i.e. the slope induced by our unconventional definition of graph and not the usual geometric slope in Euclidean coordinates which is its inverse). If $\overline{U_i U_j}$ has slope close to $1$ (the case when the slope is close to $-1$ is analogous) then two things may happen:
\begin{itemize}
\item If the slope is slightly less than 1, then there exists a $1$-Lipschitz path in the continuous model joining $U_i$ and $U_j$ but there may fail to be one in the discrete model which joins $U_i^{(T)}$ and $U_j^{(T)}$.
Note that, in particular, the scaling-projection $\varphi_T$
may be such that the segment 
$\overline{U_i^{(T)}U_j^{(T)}}$
has a slope larger than 1.
\item If the slope is slightly larger than 1,
then there is no $1$-Lipschitz path in the continuous model
which joins $U_i$ and $U_j$,
but there may still be one in the discrete model joining $U_i^{(T)}$ and $U_j^{(T)}$.
Note that, in particular, the scaling-projection
$\varphi_T$ may be such that the segment
$\overline{U^{(T)}_iU^{(T)}_j}$
has a slope smaller than 1.
\end{itemize}
Thus,
the first challenge is to show that the above situations
happen with low probability simultaneously 
for all pairs of points in $U_1,\dots,U_k$.
Intuitively, 
one would like to do this by arguing that 
the situations described above are analogous 
to what happens with the boundaries of the 
continuous/discrete triangles $\Delta$ and $\Delta_T$.
This is where the second challenge arises:
while the boundary of the discrete triangle $\Delta_T$ admits a simple representation in terms
of a Poisson process and this allows us to compare it with the boundary of $\Delta$ via \eqref{eq:triangle},
there is no such representation in the current setting
due to the fact that we are conditioning on having at least $k$ points in the discrete model and that, by definition, any compatible path which joins any two of these $k$ points is forced to remain inside the triangle $\Delta_T$.

To carry out the proof of \eqref{eq:c3suff1}, 
let us introduce some further notation.
Given $(n,t) \in \Z \times \R_+$, define
the set of compatible paths (in the discrete model)
ending at $(n,t)$ as
\[
\cC_T(n,t):= \bigg\lbrace
\gamma:[0,t]\to \Z\,\bigg|\,\begin{matrix} \gamma(t)=n,\, \gamma\text{ c\`adl\`ag},\, |\gamma(u^-)-\gamma(u)|\leq 1 \text{ for all }u,\\ 
	\gamma(u^-)\neq \gamma(u) \text{ for some }u \Longrightarrow (\gamma(u),u) \in
\cV^{[\text{st}]}_T \end{matrix}
\bigg\rbrace,
\] 
where $\cV^{[\text{st}]}_T:=\{ U_i^{(T)}  :\varepsilon(U_i^{(T)})=1\}$. That is, $\cC_T(n,t)$ is analogous to the set $\cC_{[0,t];n}$ from \eqref{def:compatible}, but where now the space-time points used to decide when a compatible~path can jump are those belonging to the coupled discrete model at time $T$. 
In particular,  $\cC_T(0,T)$ coincides with the set of paths $\cC_T$ in~\eqref{eq:ab}. 
Then, by analogy with the case $(n,t)=(0,T)$ considered in Section~\ref{sec:boundint}, for arbitrary $(n,t) \in \Z \times \R_+$ we can define $\gamma^{+,T}_{(n,t)}$ and $\gamma^{-,T}_{(n,t)}$ respectively as the rightmost~and leftmost paths in $\cC_T(n,t)$ and consider their time-reversed versions $s^{+,T}_{(n,t)}$ and $s^{-,T}_{(n,t)}$, as well as the discrete triangle $\Delta_T(n,t)$ generated by them, see Figure~\ref{fig:subtriangle} for an illustration. 
We will call $(n,t)$ the \textit{apex} of~$\Delta_T(n,t)$. Notice that, with this notation, $\Delta_T(0,T)=\Delta_T$. 
Finally, for $(x,t) \in \R \times \R_+$, define the continuous triangle $\Delta(x,t)$ with apex $(x,t)$ as
\[
\Delta(x,t):= \{ (x',t') : 0 \leq t' \leq t\,,\, |x'-x| \leq t-t'\}
\] and, for each $\rho \in [0,t]$, 
let $\Delta((x,t);\rho)$ be the $\rho$-\textit{interior} of $\Delta(x,t)$ defined as 
$\Delta((x,t);\rho):= \Delta(x,t-\rho)$ (see Figure \ref{fig:subtriangle} for an illustration). 
Observe that, with this notation, we have that $\Delta(0,1)=\Delta$ and, moreover, that $\Delta(x,t;\rho)$ is the subset of $\Delta(x,t)$ consisting of the points
at an $\ell^1$-distance greater or equal than $\rho$ from the boundary (and hence the name $\rho$-interior). For convenience, let us also set $\Delta(x,t):=\emptyset$ whenever $t < 0$.
 \begin{figure}
	\includegraphics[width=\textwidth]{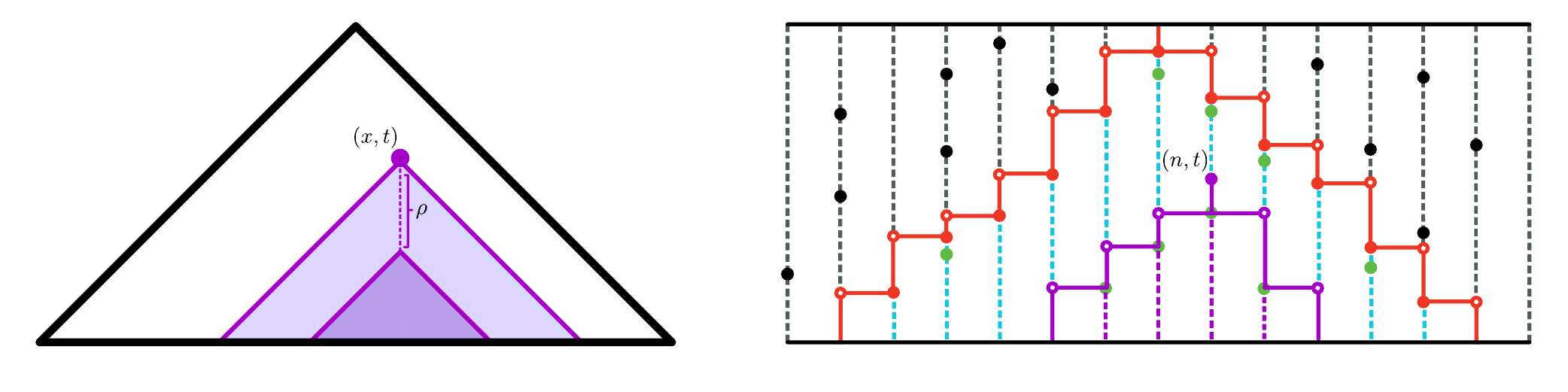}
	\caption{\textit{The triangles $\Delta(x,t)$ and $\Delta_T(n,t)$}. The picture on the left shows three triangles: the white outermost one represents $\Delta$, the middle one colored in light purple is $\Delta(x,t)$, while the smaller dark purple one is its $\rho$-interior $\Delta(x,t;\rho)$. On the right, we can see an illustration of $\Delta_T$, the region enclosed by the red paths and the $x$-axis, with the triangle $\Delta_T(n,t)$ colored in purple inside (for simplicity, we assume for the picture that $p=1$ so that all points are sticky). The paths $\gamma^{\pm,T}_{(n,t)}$ forming the boundary of $\Delta_T(n,t)$ are depicted as solid purple lines (while the interior is illustrated by dashed purple lines).}
	\label{fig:subtriangle}
\end{figure}

Before we embark on the proof of \eqref{eq:c3suff1}, let us make a few preliminary remarks. Recall the definition of the sets $\cC^{(k)}$ and $\cC^{(k)}_T$ from \eqref{eq:defck}-\eqref{eq:defckt} and observe that:
\begin{enumerate}
	\item [$\bullet$] For any $(x,t) \in \Delta$, there always exists  $\gamma \in \mathcal{L}$ such $(x,t) \in \text{graph}(\gamma)$. In particular, $\{i\} \in \cC^{(k)}$ for any $i=1,\dots,k$.
	\item [$\bullet$] For any $(n,t) \in \Delta_T$, there always exists $\gamma \in \cC_T$ such $(n,t) \in \text{graph}(\gamma)$. In particular, $\{i\} \in \cC^{(k)}_T$ for any $i=1,\dots,k$.
	\item [$\bullet$] Given $(x,t),(x',t') \in \Delta$ with $t' \leq t$, there exists $\gamma \in \mathcal{L}$ such that $\{(x,t),(x',t')\} \subseteq \text{graph}(\gamma)$ if and only if $(x',t') \in \Delta(x,t)$.
	\item [$\bullet$] Given $(n,t),(n',t') \in \Delta$ with $t' \leq t$, there exists $\gamma \in \cC_T$ such that $\{(n,t),(n',t')\} \subseteq \text{graph}(\gamma)$ if and only if $(n',t') \in \Delta_T(n,t)$.
\end{enumerate}

In light of these observations, it follows that on the event $\{ \cC^{(k)} \nsubseteq \cC^{(k)}_T \} \cap \Omega''_k$ there must exist $i\neq j \in \{1,\dots,k\}$ such that
\begin{equation}\label{eq:c3inc}
\begin{array}{ccc}
U_j \in \Delta(U_i)&\text{ and }&U^{\circ,(T)}_j \notin \Delta_T(U^{\circ,(T)}_i),
\end{array} 
\end{equation} so that, to obtain \eqref{eq:c3suff1}, it will suffice to show that, for each $i\neq j \in \{1,\dots,k\}$,
\begin{equation} \label{eq:c3inc2}
\lim_{T \rightarrow \infty} \left[ \sup_{p \geq T^{-\zeta}} \P_p ( \{U_j \in \Delta(U_i)\,,\,U^{\circ,(T)}_j \notin \Delta_T (U^{\circ,(T)}_i)\} \cap \Omega''_k) \right]= 0.
\end{equation} 
Before going into the details, we explain briefly the strategy for the proof of \eqref{eq:c3inc2}. To do this, we need to compare discrete triangles
	$\Delta_T(n,t)$ with their continuous counterparts taking into account the space-time scaling
	involved. To that end, we introduce the mapping $r_{p,T}(n,t):=r_{p}(\frac{1}{T}(n,t))=(\frac{1}{pT},\frac{t}{T})$. The strategy to obtain \eqref{eq:c3inc2} can now be summarized as follows. By analogy with the case $\Delta_T=\Delta_T(0,T)$ treated in Lemma \ref{lemma:control}, if $T$~is~large then for any $(n,t) \in \Delta_T$ one expects the (scaled) discrete triangle $r_{p,T}(\Delta_T(n,t))$ to resemble the continuous triangle $\Delta(r_{p,T}(n,t))$. 
Furthermore, on $\Omega''_k$ we have $r_{p,T}(U^{\circ,(T)}_i) \rightarrow U_i$ as~$T \rightarrow \infty$ by \eqref{eq:boundrp}, 
so that on $\Omega''_k$ 
we should have $r_{p,T}(\Delta_T(U_i^{\circ,(T)})) \approx \Delta(U_i)$ for all $T$ large enough. In particular, for any fixed $\rho \in (0,1)$, the discrete triangle $r_{p,T}(\Delta_T(U_i^{\circ,(T)}))$ should contain the $\rho$-interior $\Delta(U_i;\rho)$ if $T$ is sufficiently large, see Figure \ref{fig:inclusionF}.
The former is an important event, which for future reference we shall denote
by $\mathcal{I}^{(T)}_{i,\rho}$, i.e.
\[
\mathcal{I}^{(T)}_{i,\rho}:=\{ \Delta(U_i; \rho) \subseteq r_{p,T}(\Delta_T(U_i^{\circ,(T)}))\}.
\]
With this in mind we see that,
in order to obtain \eqref{eq:c3inc2}, by the union bound it will be enough to show that: 
\begin{enumerate}
	\item [i.] If $\rho$ is taken sufficiently small, on the event $\Omega''_k\cap\mathcal{I}^{(T)}_{i,\rho}$ 
the condition in \eqref{eq:c3inc} essentially cannot occur.
	\item [ii.] The probability of the event $(\mathcal{I}^{(T)}_{i,\rho})^c$ vanishes as $T \rightarrow \infty$.
\end{enumerate}

Step (i) in this strategy is contained in the following lemma:

\begin{lemma}\label{lemma:c3.1} For each $\zeta \in [0,1)$ and $i \neq j \in \{1,\dots,k\}$,
	\begin{equation} \label{eq:c3.1}
	\lim_{\rho \searrow 0} \limsup_{T \rightarrow \infty} \left[ \sup_{p \geq T^{-\zeta}} \P_p ( \{U_j \in \Delta(U_i)\,,\,U^{\circ,(T)}_j \notin \Delta_T (U^{\circ,(T)}_i)\} \cap \Omega''_k \cap \mathcal{I}^{(T)}_{i,\rho}) \right]=0,
	\end{equation} where $\mathcal{I}^{(T)}_{i,\rho}:=\{ \Delta(U_i; \rho) \subseteq r_{p,T}(\Delta_T(U_i^{\circ,(T)}))\}$.
\end{lemma}

The full proof of the Lemma~\ref{lemma:c3.1}
is given in Appendix~\ref{sec:techappc3}, we include here a shorter explanation. Notice that, with overwhelming probability as $\rho \searrow 0$, on the event that $U_j \in \Delta(U_i)$ we will have that in fact $U_j \in \Delta(U_i;2\rho)$. Thus, in this case, on the event $\mathcal{I}^{(T)}_{i,\rho}$ not only $U_j$ will belong to $r_{p,T}(\Delta_T(U_i^{\circ,(T)}))$ but also all points sufficiently close to $U_j$. By \eqref{eq:boundrp}, on $\Omega''_k$ this will also include $r_{p,T}(U^{\circ,(T)}_j)$ for all large enough $T$, which implies that $U^{\circ,(T)}_j \in \Delta_T(U_i^{\circ,(T)})$. This contradicts the conditions stated in the event from \eqref{eq:c3.1} and thus shows that its probability must vanish as $T \rightarrow \infty$ and $\rho \searrow 0$. 

Let us now turn to the proof of (ii).
In agreement with our heuristics, we would like to apply Lemma \ref{lemma:control} in this context. However, before we can do so, we must first translate the event $\mathcal{I}^{(T)}_{i,\rho}$ into the language of Lemma~\ref{lemma:control}. To do this, let us write $(n_i,t_i):=U^{\circ,(T)}_i$ momentarily to simplify the notation. Then, observe that the inclusion $\Delta(U_i;\rho) \subseteq r_{p,T}(\Delta_T(n_i,t_i))$ will be guaranteed if the following two conditions occur:
\begin{enumerate}
	\item[I1.] The apexes of the triangles $r_{p,T}(\Delta_T(n_i,t_i))$ and $\Delta(U_i)$ are sufficiently close to each other.
	\item[I2.] 
	The ``slopes'' of the paths generating the boundary of $r_{p,T}\big(
	\Delta_T(n_i,t_i) \big)$ are not much smaller (in absolute value) than $1$, where 1 corresponds to the absolute value of the slopes of the paths which generate the boundary of $\Delta(U_i)$.
\end{enumerate} 
More precisely, 
if for each $\rho>0$ we consider the events
\begin{equation}\label{eq:normalizedpaths}
	\mathcal{P}^{+,(T)}_{i,\rho}:=\left\{ \inf_{0 \leq u \leq 1}\left[ \frac{1}{pt_i}(s^{+,T}_{(n_i,t_i)}(ut_i)-n_i) - \left(1-\frac{\rho}{4}\right)u\right] > -\frac{\rho}{4}\right\}
\end{equation} and 
\[
\mathcal{P}^{-,(T)}_{i,\rho}:=\left\{ \inf_{0 \leq u \leq 1} \left[ \frac{1}{pt_i}(n_i-s^{-,T}_{(n_i,t_i)}(ut_i)) - \left(1-\frac{\rho}{4}\right)u\right] > -\frac{\rho}{4}\right\}
\]then, since (I1) immediately holds on $\Omega''_k$ for all $T$ sufficiently large by \eqref{eq:boundrp}, one can check that for all $T$ large enough we have that
\begin{equation}\label{eq:c3inc0}
	\Omega''_k \cap (\mathcal{P}^{+,(T)}_{i,\rho} \cap \mathcal{P}^{-,(T)}_{i,\rho}) \subseteq \Omega''_k \cap \mathcal{I}^{(T)}_{i,\rho}. 
\end{equation}Indeed, if $t(U_i) < \rho$ then we have $\Delta(U_i;\rho)=\emptyset$ and there is nothing to prove. On the other hand, if $t(U_i) \geq \rho$ then on $\Omega''_k \cap (\mathcal{P}^{+,(T)}_{i,\rho} \cap \mathcal{P}^{-,(T)}_{i,\rho})$ we have that $T \geq t_i=t(U_i)T \geq \rho T$ and also that
\[
\frac{1}{pT}(s^{-,T}_{(n_i,t_i)}(vT)-n_i)< -v + \frac{\rho}{2} \hspace{1cm}\text{ and }\hspace{1cm}v - \frac{\rho}{2}  < \frac{1}{pT}(s^{+,T}_{(n_i,t_i)}(vT)-n_i) 
\] for all $v \in [0,t_i/T]$ which, since $|s^{+,T}_{(n_i,t_i)}(t)-\gamma^{+,T}_{(n_i,t_i)}(t_i-t)|\leq 1$ for all $t \in [0,t_i]$ (and the same holds for $s^{-,T}_{(n_i,t_i)}$ and $\gamma^{-,T}_{(n_i,t_i)}$), implies the inequalities
\begin{equation}\label{eq:compslopes1}
	\frac{1}{pT}(\gamma^{-,T}_{(n_i,t_i)}(vT)-n_i)<-\left(\frac{t_i}{T}-v\right) + \frac{\rho}{2} + \frac{1}{pT}  
\end{equation} and
\begin{equation}\label{eq:comslopes2}
	\left(\frac{t_i}{T}-v\right) - \frac{\rho}{2} -\frac{1}{pT}  < \frac{1}{pT}(\gamma^{+,T}_{(n_i,t_i)}(vT)-n_i) 
\end{equation} for all $v \in [0,t_i/T]$. From here, a straightforward computation using \eqref{eq:boundrp} then shows that, if $p\geq T^{-\zeta}$ and $T$ is sufficiently large so as to have $T^{\zeta-1} < \frac{\rho}{4}$, the inclusion $\Delta(U_i;\rho) \subseteq r_{p,T}(\Delta_T(n_i,t_i))$ holds and therefore \eqref{eq:c3inc0}  now follows. See Figure \ref{fig:inclusionF} for an illustration.

\begin{figure}
	\includegraphics[width=\textwidth]{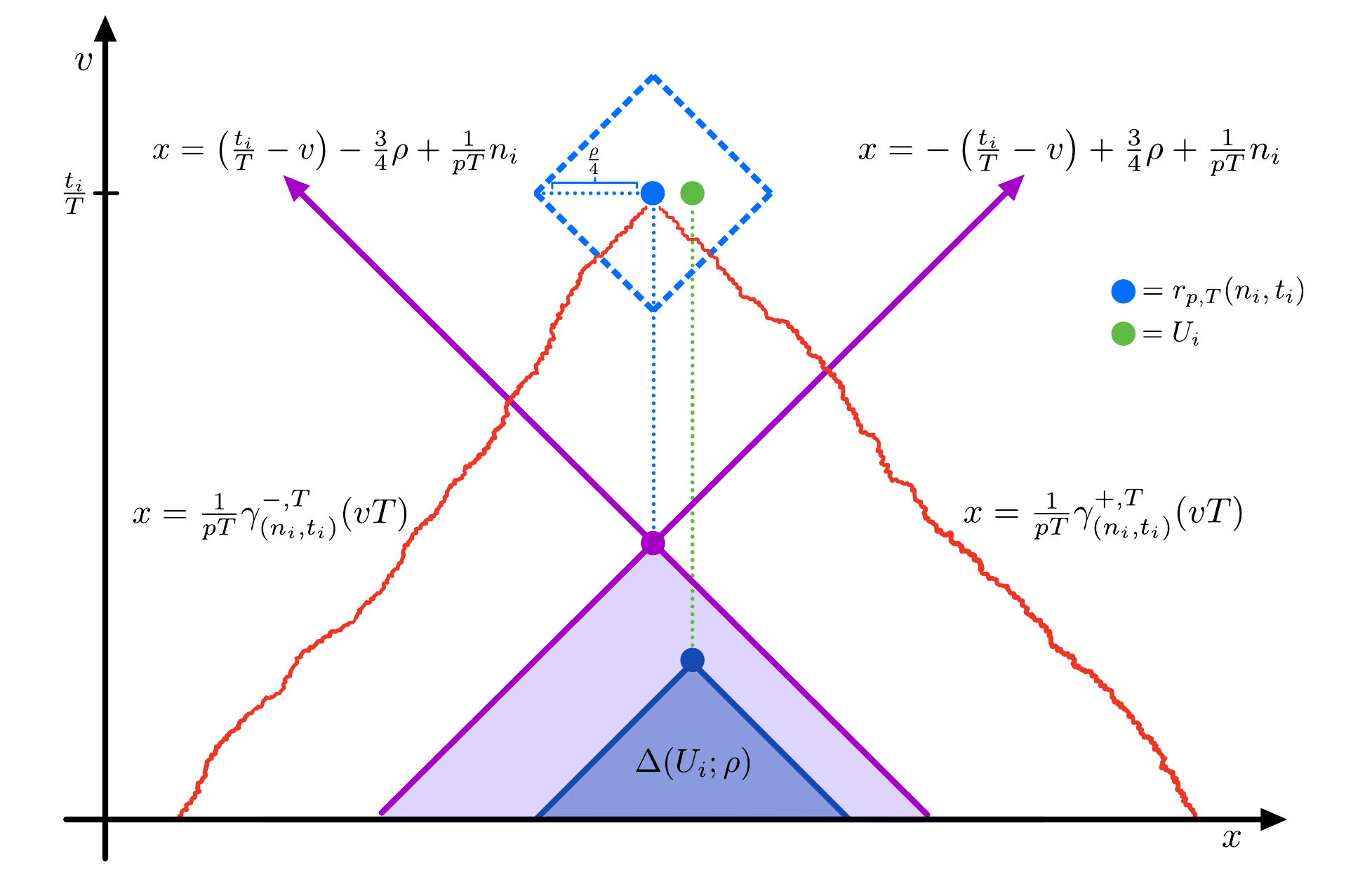}
	\caption{\textit{The inclusion $\Delta(U_i;\rho) \subseteq r_{p,T}(\Delta_T(n_i,t_i))$}. If $T^{\zeta-1} < \frac{\rho}{4}$, whenever the events in \eqref{eq:compslopes1}-\eqref{eq:comslopes2} occur we have that $r_{p,T}(\Delta_T(n_i,t_i))$, the region enclosed by the red paths, contains the purple triangle shown in the picture, which is in fact exactly $\Delta(r_{p,T}(n_i,t_i);\frac{3}{4}\rho)$. In particular, if $\|r_{p,T}(n_i,t_i) - U_i\| < \frac{\rho}{4}$ then the blue triangle $\Delta(U_i;\rho)$ is contained in the purple one $\Delta(r_{p,T}(n_i,t_i);\frac{3}{4}\rho)$, so that the inclusion $\Delta(U_i;\rho) \subseteq r_{p,T}(\Delta_T(n_i,t_i))$ holds.}
	\label{fig:inclusionF}
\end{figure}

In particular, in light of Lemma \ref{lemma:c3.1} and \eqref{eq:c3inc0}, we obtain that \eqref{eq:c3inc2} will follow from the union bound once we prove that, for each $i=1,\dots,k$ and $\rho \in (0,1)$,
\begin{equation} \label{eq:c3suff4}
	\lim_{T \rightarrow \infty} \left[\sup_{p \geq T^{-\zeta}} \P_p\big(\Omega''_{k} \cap (\mathcal{P}^{+,(T)}_{i,\rho})^c\big) \right] = 0	
\end{equation} and
\begin{equation}\label{eq:c3suff4b}
	\lim_{T \rightarrow \infty} \left[\sup_{p \geq T^{-\zeta}} \P_p\big(\Omega''_{k} \cap (\mathcal{P}^{-,(T)}_{i,\rho})^c\big) \right] = 0.	
\end{equation} We will prove only \eqref{eq:c3suff4}, as the argument for \eqref{eq:c3suff4b} is completely analogous.

Given the similarity between $(\mathcal{P}^{+,(T)}_{i,\rho})^c$ and the lower bound found in the event in \eqref{eq:triangle}, one would be tempted to use the exponential decay from Lemma \ref{lemma:control} to show \eqref{eq:c3suff4}. The main issue is that on the event $\{N^\circ_T \geq k\}$ the path~$s_{(n_i,t_i)}^{+,T}$ may not act exactly as a Poisson jump process: indeed, the value of $N^\circ_T$ will influence the number of jumps that this path can make. However, since we are only conditioning on the event $\{N^\circ_T \geq k\}$ which will occur with probability tending to one as $T \rightarrow \infty$, this will only introduce a small dependence which we can do away with. Indeed, the idea will be to split the interior $\cV^\circ_T$ into two sets $\cV^{\circ,\,\text{small}}_T$ and $\cV^{\circ,\,\text{big}}_T$, which will be independent given $\partial \cV_T$, in such~a~way that, with overwhelming probability as $T \rightarrow \infty$, the following occurs:
\begin{enumerate}
	\item [a)] On the event $\{N^\circ_T \geq k\}$, the space-time positions of the $k$ largest heights in the interior belong to $\cV^{\circ,\,\text{small}}_T$, i.e. $U_1^{\circ,(T)},\dots,U_k^{\circ,(T)} \in \cV_T^{\circ,\,\text{small}}$.
\item [b)] For each $i=1,\dots,N^\circ_T$, the rightmost path in $\cC_T(n_i,t_i)$ does not depend on any of the points in $\cV_T^{\circ,\,\text{small}}$, except perhaps for $(n_i,t_i)$. (Notice that, in principle, $s^{+,T}_{(n_i,t_i)}$ should depend on \textit{all} of $\cV_T=\cV_T^{\circ,\,\text{small}} \cup \cV_T^{\circ,\,\text{big}} \cup \partial \cV_T$.)
\end{enumerate} The two properties (a) and (b) above will allow us to ``decouple'' the behavior of the rightmost path $s^{+,T}_{(n_i,t_i)}$ in $\cC_T(n_i,t_i)$ from the information contained in the event $\{N^\circ_T \geq k\}$, thus freeing us from this small dependence stated earlier. 

We will construct the sets $\cV^{\circ,\,\text{small}}_T$ and $\cV^{\circ,\,\text{big}}_T$ as colorings of $\cV^\circ_T$. Recall that, given $p \in (0,1)$, a random subset $A$ of a countable set $X$ is a $p$-\textit{coloring} of~$X$ if $A$ is obtained from $X$ by letting each $x \in X$ belong to $A$ independently with probability $p$. Moreover, recall also that, according to our construction of $\cV^\circ_T$ in Section \ref{sec:couplingstep2}, each space-time~point in the interior $U_i^{\circ,(T)}$ has an associated $z$-mark in $\R_+$ denoted by $z(U_i^{\circ,(T)})$. Then, for a fixed parameter $\beta \in (\zeta,1)$, we define $\cV^{\circ,\,\text{small}}_T$ and~$\cV^{\circ,\,\text{big}}_T$ as
\[
\cV^{\circ,\,\text{small}}_T:=\{ U_i^{\circ,(T)} \in \cV^\circ_T : z(U_i^{\circ,(T)}) < T^\beta\}
\] and
\[
\cV^{\circ,\,\text{big}}_T:=\cV^{\circ}_T-\cV^{\circ,\,\text{small}}_T=\{ U_i^{\circ,(T)} \in \cV^\circ_T : T^\beta \leq z(U_i^{\circ,(T)}) < T^2\}.
\] Since the $z$-marks of points in $\cV^\circ_T$ are i.i.d. with uniform distribution on~$[0,T^2)$ (although we later order these in increasing fashion for labeling purposes), it~follows that $\cV^{\circ,\,\text{small}}_T$ is a $T^{\beta-2}$-coloring of $\cV^\circ_T$ (and, likewise, that $\cV^{\circ,\,\text{big}}_T$ is a $(1-T^{\beta-2})$-coloring of $\cV^\circ_T$). In particular, since $\cV^\circ_T$ is a Poisson process given~$\partial \cV_T$, the sets $\cV^{\circ,\,\text{small}}_T$ and $\cV^{\circ,\,\text{big}}_T$ are also independent given~$\partial \cV_T$.

To continue, let us verify properties (a) and (b) above for $\cV^{\circ,\,\text{small}}_T$ and $\cV^{\circ,\,\text{big}}_T$. If we write $N^{\circ,\,\text{small}}_T:=|\cV^{\circ,\,\text{small}}_T|$ then, since $\cV^{\circ,\,\text{small}}_T$ is a $T^{\beta-2}$-coloring of~$\cV^{\circ}_T$ and $pT^\beta \rightarrow \infty$ as $T \rightarrow \infty$ uniformly over all $p \in [T^{-\zeta},1]$, from an estimate of the form in \eqref{eq:T2pts} we obtain that for each $i \in \N$,
\begin{equation} \label{eq:ncirc1}
\lim_{T \rightarrow \infty}\left[\sup_{p\geq T^{-\zeta}}  \P_p( N^{\circ,\,\text{small}}_T < i )\right] = 0
\end{equation} which, since $\cV_T^{\circ,\,\text{small}}$ contains the points with the lowest $z$-marks, implies (a). To see (b), let us define for $(n,t) \in \Z \times \R_+$ the set of paths
\begin{equation} \label{eq:defcihat}
\widehat{\cC}_T(n,t):= \bigg\lbrace
\gamma \in \cC_T(n,t) : \gamma(u^-)\neq \gamma(u) \Longrightarrow (\gamma(u),u) \in \cV^{[\text{st}]}_T \setminus \cV_T^{\circ,\,\text{small}}
\bigg\rbrace
\end{equation} for $\cV^{[\text{st}]}_T :=\{ U^{(T)}_i \in \cV_T :  \varepsilon(U^{(T)}_i) =1\}$. That is, $\widehat{\cC}_T(n,t)$ is the set of paths in $\cC_T(n,t)$ which are allowed to jump only when they encounter a sticky point which belongs to $\cV^{\circ,\,\text{big}}_T \cup \partial \cV_T$, fully disregarding all sticky points in $\cV_T^{\circ,\,\text{small}}$. Then, property (b) can be formalized as saying that, for any $i=1,\dots,N^\circ_T$, with overwhelming probability as $T \rightarrow \infty$ the rightmost path in $\widehat{\cC}_T(U_i^{\circ,(T)})$ coincides with the one in $\cC_T(U_i^{\circ,(T)})$. More precisely, if we let $\widehat{s}^{\,+,T}_{(n,t)}$ denote the (time-reversed) rightmost path in $\widehat{\cC}_T(n,t)$ and we write $(n_i,t_i):=U^{\circ,(T)}_i$ as before, then we have the following
equivalent formulation of property~(b):

\begin{lemma}\label{lemma:bigpathscoincide} For each $\zeta \in [0,1)$ and $i=1,\dots,k$,
	\begin{equation} \label{eq:equality}
		\lim_{T \rightarrow \infty} \left[ \sup_{p\geq T^{-\zeta}} \P_p( N^\circ_T \geq i\,,\, s^{+,T}_{(n_i,t_i)} \neq \widehat{s}^{\,+,T}_{(n_i,t_i)})\right] =0.
	\end{equation}
\end{lemma}

The details of the proof can be found in the Appendix~\ref{sec:techappc3}.
The main idea behind the proof of Lemma \ref{lemma:bigpathscoincide} can be summarized as follows. The only way in which the rightmost paths $s^{+,T}_{(n_i,t_i)}$ and $\widehat{s}^{\,+,T}_{(n_i,t_i)}$ can differ is if both of them reach simultaneously a sticky point belonging to  $\cV^{\circ,\,\text{small}}_T$. However, since $\cV^{\circ,\,\text{small}}_T$ is relatively small in size (at least compared with the total size of~$\cV_T$), the points in $\cV^{\circ,\,\text{small}}_T$ will be very scarce and therefore the probability that the former occurs will~be very small for large enough~$T$. Indeed, since $\widehat{s}^{\,+,T}_{(n_i,t_i)}$ has length at most $T$, 
by standard properties of Poisson processes,
the probability that $\widehat{s}^{\,+,T}_{(n_i,t_i)}$ ever reaches a point in $\cV^{\circ,\,\text{small}}_T$
can be seen to be (roughly) at most $T^{\beta-2} \cdot T= T^{\beta -1}$
which goes to~$0$ as~$T \rightarrow \infty$ . 

Having shown properties (a) and (b), let us continue with the proof of \eqref{eq:c3suff4}. In light of Lemma \ref{lemma:bigpathscoincide}, it will be enough to prove a version of \eqref{eq:c3suff4} in which the rightmost path $s^{+,T}_{(n_i,t_i)}$ in the event $\mathcal{P}^{+,(T)}_{i,\rho}$ is replaced by the new one $\widehat{s}^{\,+,T}_{(n_i,t_i)}$. More precisely, by Lemma \ref{lemma:bigpathscoincide} and \eqref{eq:ncirc1}, to obtain \eqref{eq:c3suff4} (and thus conclude the proof), it will suffice to show that for each $\rho \in (0,1)$ and $i=1,\dots,k$
\[
\lim_{T \rightarrow \infty} \left[ \sup_{p \geq T^{-\zeta}} \P_p\big(\{N^{\circ,\,\text{small}}_T \geq i\} \cap  \big(\widehat{\mathcal{P}}^{+,(T)}_{i,\rho}\big)^c\big)\right]=0,
\] where $\widehat{\mathcal{P}}^{+,(T)}_{i,\rho}$ is defined as in \eqref{eq:normalizedpaths} but using the path $\widehat{s}^{\,+,T}_{(n_i,t_i)}$ instead of $s^{+,T}_{(n_i,t_i)}$. The advantage of performing this switch is that, as we will see, the events $\{N^{\circ,\,\text{small}}_T \geq i\}$ and $\big(\widehat{\mathcal{P}}^{+,(T)}_{i,\rho}\big)^c$ can be ``decoupled'' by conditioning~on~$\partial \cV_T$, thus solving the small dependence issue mentioned in the discussion following \eqref{eq:c3suff4}. However, as per our original plan, the idea now is still to use \eqref{eq:triangle} to control the probability of $\big(\widehat{\mathcal{P}}^{+,(T)}_{i,\rho}\big)^c$ and to do this we first need to make sure that $\widehat{s}^{\,+,T}_{(n_i,t_i)}$ is indeed a Poisson process. 

At this point is where we arrive at our  final obstacle: since $\cV_T^{\circ,\,\text{big}}$ is a ${(1-T^{\beta-2})}$-coloring of $\cV_T^\circ$ it follows that, conditionally on  $\cV_T^{\circ,\,\text{small}}$ and~$\partial \cV_T$, on the event $\{ N_T^{\circ,\,\text{small}}\geq i\}$ the path $\widehat{s}^{\,+,T}_{(n_i,t_i)}-n_i$ is indeed distributed as a Poisson jump process on the time interval $[0,t_i]$ (of parameter $p(1-T^{\beta-2})$), but only until \textit{the first time in which $\widehat{s}^{\,+,T}_{(n_i,t_i)}$ hits $\partial \Delta_T$}, the boundary of the ``outer'' triangle~$\Delta_T$. 
By definition of $\widehat{\cC}_T(n_i,t_i)$ (see \eqref{eq:defcihat}), once $\widehat{s}^{\,+,T}_{(n_i,t_i)}$ hits $\partial \Delta_T$ it will agree with the rightmost path $s^{+,T}$ in $\cC_T=\cC_T(0,1)$ from then onwards, making it no longer random (given $\partial \cV_T$). Therefore, we cannot directly apply \eqref{eq:triangle} to the path $\widehat{s}^{\,+,T}_{(n_i,t_i)}-n_i$. Fortunately, this will not be a serious problem for us because, as we shall later see, the event that this path eventually hits the boundary $\partial \Delta_T$ is extremely unlikely.		   

To deal with this small technical issue, we proceed in the following manner. That $\widehat{s}^{\,+,T}_{(n_i,t_i)}-n_i$ is a Poisson process until the hitting time of $\partial \Delta_T$ implies~that we can couple it with a \textit{true} Poisson process $\sigma^{+,T}$ with parameter $p(1-T^{\beta-2})$, which is independent of both $\cV_T^{\circ,\,\text{small}}$ and $\partial \cV_T$, in such a way that both $\widehat{s}^{\,+,T}_{(n_i,t_i)}$ and $n_i+\sigma^{+,T}$ agree until the first~time in which they hit the boundary $\partial \Delta_T$ (we omit the details of such a coupling since it is straightforward). But, since $(n_i,t_i)$ will lie far away from $\partial \Delta_T$ with overwhelming probability as $T \rightarrow \infty$, there are essentially two ways in which $n_i+\sigma^{+,T}$ can hit the boundary: (i) if it makes an unusually high number of jumps (so that its ``slope'' exceeds the typical value $p(1-T^{\beta-2})$) or (ii) if the triangle $\Delta_T$ is unusually narrow, i.e. the rightmost path $s^{+,T}$ makes an unusually low number of jumps (so that its ``slope'' is below the typical value $p$). Both these large deviations events can be handled using \eqref{eq:triangle}, thus yielding:

\begin{lemma}\label{lemma:couplepoisson} For each $\zeta \in [0,1)$ and $i=1,\dots,k$, 
\begin{equation} \label{eq:equality2}
	\lim_{T \rightarrow +\infty} \left[\inf_{p \geq T^{-\zeta}} \P_p\left( N^{\circ,\,\text{small}}_T \geq i\,,\,  \widehat{s}^{\,+,T}_{(n_i,t_i)}= n_i+ \sigma^{+,T} \text{ on }[0,t_i]\right)\right]=1.
\end{equation}
\end{lemma}

As we claimed above, Lemma \ref{lemma:couplepoisson} follows from a standard application of \eqref{eq:triangle}. However, since some steps of the proof require a bit of care, for convenience of the reader we include all details in Appendix~\ref{sec:techappc3}.

With Lemma \ref{lemma:couplepoisson} at our disposal, we are ready to conclude the proof of~(C3'). Indeed, by combining \eqref{eq:ncirc1},  \eqref{eq:equality} and \eqref{eq:equality2}, we obtain that, in order to establish \eqref{eq:c3suff4}, it will suffice to prove that, for each $\rho \in (0,1)$ and $i \in \N$, 
\begin{equation} \label{eq:finalbound}
\lim_{T \rightarrow \infty} \left[ \sup_{p \geq T^{-\zeta}} \P_p(\{ N^{\circ,\,\text{small}}_T \geq i\} \cap \mathcal{Q}^{+,(T)}_{i,\rho})\right]=0,
\end{equation} where 
\[
\mathcal{Q}^{+,(T)}_{i,\rho}:=\left\{ \inf_{0 \leq u \leq 1} \left[ \frac{1}{pt_i}\sigma^{+,T}(ut_i) - \left(1-\frac{\rho}{4}\right)u\right] \leq \frac{\rho}{4}\right\}.
\] But, since $\sigma^{+,T}$ is a Poisson jump processes with parameter $p(1-T^{\beta-2})$, \eqref{eq:finalbound} follows from \eqref{eq:triangle} by conditioning on $\cV_T^{\circ,\,\text{small}}$ in a similar fashion to \eqref{eq:equality2} above. We omit the details (but refer the reader to  the proof of Lemma \ref{lemma:couplepoisson}). This yields \eqref{eq:c3suff4} and thus concludes the proof of Proposition \ref{prop:coupling2}.

\section{Proof of Proposition \ref{prop:remdiscrete2}}

In this section we carry out the proof of Proposition \ref{prop:remdiscrete2}. Observe that this will immediately prove Proposition \ref{prop:remdiscrete} as well. The proof will require a few preliminary results, which we shall cover in the following two subsections. The derivation of Proposition~\ref{prop:remdiscrete2} using these results will then be carried out in Subsection \ref{sec:derivation}. 

\subsection{Bernoulli ballistic deposition}\label{sec:bbd}

Our first step in the proof of Proposition~\ref{prop:remdiscrete2} will be to study an auxiliary ballistic deposition~model, which we call Bernoulli ballistic deposition (BBD). This model is exactly the same as the original one introduced in Section \ref{sec:desc}, with the only difference that now the heights of the falling blocks are i.i.d. Bernoulli random variables of some parameter $\sigma \in (0,1]$. Thus, a ``flat'' block of height $0$ can only impact the size of a given column whenever it sticks to one of its adjacent columns which is higher.

Let us fix $\s,p\in (0,1]$ and denote by $h_T(x)$ the height of the growing cluster of blocks at time $T>0$ above a given site $x \in \Z$. Then, once again we have the last passage percolation representation of the height $h_T$ from Section \ref{sec:lpprep}, i.e. if $h_0 \equiv 0$ then with probability one,
\begin{equation} \label{eq:lpprep2}
	h_T(0)=\max\bigg\lbrace
	\sum_{(\gamma(t),t) \in \xi} \eta(
	\gamma(t),t
	) :
	\gamma\in\cC_{T}
	\bigg\rbrace,
\end{equation} where the $\xi$ and $\cC_T$ are the same before, but now the marks $\eta(x,t)$ are i.i.d. Bernoulli random variables of parameter $\sigma$.

Throughout this subsection, since we will work with fixed values of $\sigma$ and $p$, we shall omit the dependence on these parameters from the notation unless explicitly stated otherwise. The main objective of this subsection is to prove the following result:
\begin{theorem}\label{theo:bbp} There exists a constant $C>0$ such that, for any $\zeta \in [0,1]$, $\sigma\in (0,1]$, $T \geq 1$ and $p\in [T^{-\zeta},1]$, 
	\begin{equation} \label{eq:hml}
		\E\big(
		h_T(0)
		\big)\leq C\big(\sigma p T^2\big)^{\tfrac{1}{2-\zeta}},
	\end{equation} where $h_T$ is the height function for the BBD model with sticking parameter~$p$, block height distribution $\text{Bernoulli}(\sigma)$ and initial configuration $h_0 \equiv 0$.  
\end{theorem}

In order to ease the notation of this subsection we write
\[
\Gamma:=(\sigma p T^2)^{\tfrac{1}{2-\zeta}}
\]

Let us notice that for the case $\sigma pT^2 \leq 1$, the bound in \eqref{eq:hml} is straightforward. Indeed, by \eqref{eq:lpprep2} we can bound $h_T(0)$ by the number of blocks with height $1$ in $\cV_T$ which, by \eqref{eq:expectations}, gives the bound
\begin{equation}\label{eq:bbd1}
	\E(h_T(0)) = \sigma \E(N_T) = \sigma(pT^2+T) \leq 
	2\sigma pT^2 \leq 2(\sigma pT^2)^{\tfrac{1}{2-\zeta}},
\end{equation} where the first inequality follows from the fact that 
$pT^2\geq T^{2-\zeta}\geq T$ if $\zeta \leq 1$, whereas the second inequality is a consequence of the assumption $\sigma p T^2\leq 1$ and the fact that $2-\zeta \geq 1$. Hence, for the remainder of the~section we shall focus on the case 
$\sigma pT^2 > 1$. 

There are essentially two ways for a path to collect a large number of positive Bernoulli marks. Either the path finds regions where there are many positive Bernoulli marks and remains localized there, or it manages to travel across many different regions, collecting all the positive Bernoulli marks it can find. Whenever 
$\sigma pT^2 > 1$, the good concentration properties of the Bernoulli distribution rule out the former option (see Lemma \ref{smallk}), 
while the regularity of the Poisson process makes the latter highly unlikely (see Lemma \ref{largek} and keep in mind that, for $T$ large enough, compatible paths should be~essentially $p$-Lipschitz, see Section \ref{sec:condc3} for details). Therefore, with high probability any compatible path will not be able to collect a large number of positive Bernoulli~marks, a fact which in~turn will yield the bound in \eqref{eq:hml}.

To formalize this intuition and prove Theorem \ref{theo:bbp} in this case, we introduce some new notation and present a few useful lemmas. 
For fixed $\sigma \in (0,1]$ 
and $p\in [T^{-\zeta}, 1]$, set
\[
\ell_h:=\bigg\lfloor
\frac{\Gamma}{\sigma T}
\bigg\rfloor\,,\qquad
\ell_v:=\frac{T}{2\Gamma},
\]
and for $(i,j)\in\Z^2$ define the mesoscopic box $C_{ij}\subset \R^2$ as
\[
C_{i,j}
:=
(\ell_h i, \ell_v j)
+[-\ell_h/2,\ell_h/2)
\times [0,\ell_v].
\]
We shall call $i$ the \textit{column index} and $j$ the \textit{row index} of $C_{i,j}$, respectively. Notice that the dimensions of any such box are 
$\ell_h \times \ell_v$. We point out a few important consequences of this choice of $\ell_v$ and $\ell_h$ in the following lemma:

\begin{lemma}\label{lemma:dimrk} 
For $\ell_h$ and $\ell_v$ defined as above, we have:
\begin{enumerate}
\item $\ell_h \geq 1$;
\item the mean number of blocks of height $1$ in any   mesoscopic box $C_{i,j}$ has Poisson distribution with mean $\sigma \ell_h \ell_v \leq 1$;
\item $p\ell_v/\ell_h \leq 1$.
\end{enumerate}
\end{lemma}

\begin{remark}\label{boxdimrk}
On average, a compatible path that jumps only in one direction (right or left) crosses a horizontal distance of $p\ell_v$ in a time $\ell_v$, so that the quantity $p\ell_v/\ell_h$ can be interpreted as the mean proportion of a mesoscopic box that can be covered horizontally. The fact that this quantity is smaller than $1$ means that it is \emph{difficult} for an average path to fully cross a box horizontally.
\end{remark}

\begin{proof}[Proof of Lemma \ref{lemma:dimrk}]$\,$
	\begin{enumerate}
		\item By choice of $p$ we have 
		$\Gamma=(\sigma p T^2)^{\tfrac{1}{2-\zeta}}\geq  \sigma^{\tfrac{1}{2-\zeta}}T\geq \sigma T$ since $2-\zeta \geq 1$, so that $\ell_h\geq 1$.
		\item Obvious from the definition of the model and the dimensions of $C_{i,j}$.
		\item By Part 1, we can use the fact that $2\lfloor x\rfloor \geq x$ when $x\geq 1$,
		so that 
		\[
		\frac{p\ell_v}{\ell_h} \leq \frac{\sigma p T^2}{(\sigma p T^2)^{\tfrac{2}{2-\zeta}}} =(\sigma p T^2)^{-\tfrac{\zeta}{2-\zeta}},
		\]
		which, since $\sigma p T^2 > 1$, is bounded from above by 1.
	\end{enumerate}
\end{proof}

Let $r:=\lceil T/\ell_v \rceil \geq 1$ be the minimum number of mesoscopic boxes $C_{i,j}$
needed to vertically cover the interval $[0,T]$.
In particular, $r$ is also the minimum number of boxes 
needed to cover (the graph of) any compatible path $\gamma \in \cC_T$. 
Note that the boxes from row $r$ may fail to be entirely contained 
in $\R \times [0,T]$
and can in fact exceed the height $T$
(but this will not make any difference in the argument). 
Given a path $\gamma \in\cC_T$, let $B(\gamma)$ be the minimal collection of boxes $C_{i,j}$ covering the graph of $\gamma$. A collection of boxes $C_{i,j}$ will be called a \textit{compatible configuration} if it equals $B(\gamma)$ for some path $\gamma\in\cC_T$. For $k\geq r$, let $\cB_k$ denote the set of all compatible
configurations $\beta$ consisting of $k$ boxes. We shall encode the configurations of boxes in $\cB_k$ in the following manner: an element $\beta\in\cB_k$ will be denoted by a sequence
\[
\beta=(i_j,k_j)_{j=1}^r
\] where $i_j$ represents the column index of the leftmost
box in row $j$ and $k_j$ the number of boxes in row $j$.
Since we are considering compatible configurations of boxes and all graphs of paths $\gamma \in \cC_T$ cover $[0,T]$ vertically,
there must be at least one box in each row  so that 
$k_1,\dots,k_r \geq 1$ and also, since there are $k$ boxes in total,
we have $k_1+\cdots+k_r=k$. Furthermore, since almost~surely there are no points of the Poisson process at times $t=\ell_v j$, for any fixed $\sigma\in (0,1]$,
$p\in [T^{-\zeta}, 1]$
and all $j \in \Z$, at least one box in each row has~to~be adjacent to (i.e. share one of the sides with) one of the boxes in the row~above (except for the top row, of course). Finally, since any path $\gamma \in \cC_T$ ends at $x=0$, any compatible configuration must contain the box $C_{0,r}$ and so $i_r \leq 0$. 
Therefore, we have the following description for $\cB_k$:
\[
\cB_k=\Bigg\lbrace
(i_j,k_j)_{j=1}^r\,\Bigg|\,\begin{matrix}
	k_1,\dots,k_r\geq1\,,\ k_1+\cdots+k_r=k\,,\
	-k_r+1\leq i_r\leq 0\\
	i_{j+1}-k_j+1\leq i_j\leq i_{j+1}+k_{j+1}-1\, \text{ for all $1\leq j\leq r-1$}
\end{matrix}
\Bigg\rbrace.
\] This description allows us to estimate the size of $\cB_k$ in the following way: 
\begin{lemma}\label{lemma:cardBk}
	For every $k\geq r\geq 2$, we have
	\[
	|\cB_k|\leq	\mathrm{e}^{3k}.
	\]
\end{lemma}
\begin{proof} Let $\mathcal{B}_k(k_1,\dots,k_r)$ be the collection of compatible  configurations of $k$ boxes which have exactly $k_j$ boxes on the $j$-th row, for each $j=1,\dots,r$. Since at least one box in each row has to be adjacent to one of the boxes in the row above,
	\[
	|\mathcal{B}_k(k_1,\dots,k_r)|=  (k_1+k_2-1)\dots(k_{r-1}+k_r-1)k_r.
	\] Using the bound $x \leq e^x$ for all $x > 0$, we obtain that 
	\[
	|\mathcal{B}_k(k_1,\dots,k_r)|\leq \exp\Big\{  (k_1+k_2)+\dots+(k_{r-1}+k_r)+k_r\Big\} \leq \mathrm{e}^{2k}.
	\] Hence, by summing over all possible choices of $k_1,\dots,k_r$, we obtain that
	\[
	|\mathcal{B}_k| \leq \binom{k-1}{r-1} \mathrm{e}^{2k}.
	\]
	Since the binomial coefficient is bounded from above by $2^{k-1}$, this readily gives the desired bound.
\end{proof}

Now, given any configuration of boxes $\beta$, let $n_\beta$ denote
the number of points contained in $\beta$ with height mark $\eta$ equal to $1$.
\begin{lemma}\label{smallk} Let $u > 0$ and $k < u\Gamma$. Then, for any $\beta\in\cB_k$,
	\[
	\mathbb{P}\big(n_\beta>u \Gamma
	\big) \leq 
	\exp\Big\{ -k-u\Gamma\Big( \log \frac{u\Gamma}{k} -1\Big)\Big\}.
	\]
\end{lemma}
\begin{proof} 
	Fix $u > 0$, 
	$k<u\Gamma$ and 
	$\beta \in \mathcal{B}_k$. 
	Since the different boxes $C_{i,j}$ do not overlap, 
	the random variable $n_\beta$ is Poisson distributed 
	with parameter $k\sigma \ell_h\ell_v \leq k$ by Lemma \ref{lemma:dimrk}. Hence,
	by the exponential Chebyshev inequality, for every $\lambda>0$,
	$$
	\P\Big(
	n_\beta > u\Gamma
	\Big)
	<\mathrm{e}^{
		-u\lambda \Gamma}
	\E\big(
	\mathrm{e}^{\l n_\beta}
	\big)\,
	\leq\,
	\exp\Big(
	-u\lambda \Gamma+k(\mathrm{e}^\lambda-1)
	\Big).
	$$
	Minimizing with respect to $\lambda$,
	we find that 
	$\lambda=\log(u\Gamma/k)>0$ is optimal,
	which in turn gives the desired bound.  
\end{proof}

Recall that, given a path $\gamma\in\cC_T$, we defined $B(\gamma)$ to be the minimal cover~of~$\gamma$ by mesoscopic boxes $C_{i,j}$. We have the following result:
\begin{lemma}\label{largek} If $k>16r$ then for any $\b\in\cB_k$ we have
	\begin{align*}
		\P\big(
		\exists\, \gamma\in\cC_T: B(\gamma)=\b
		\big)\,&\leq\
		\exp\Big\{ \ell_h\Big[ k - (k-4r)\log\Big(\frac{k}{r}-4\Big)\Big]\Big\}\\
		&\leq\exp\Big\{ k\Big[ 1 - \frac{3}{4}\log\Big(\frac{k}{r}-4\Big)\Big]\Big\}.
	\end{align*}
\end{lemma}
\begin{proof}
	Note that a path $\gamma$ can only change its $x$ coordinate
	if it encounters a sticky point, so that the occurrence of the event $\lbrace\exists\,\gamma\in\cC_T: B(\gamma)=\b\rbrace$ depends only on the points in $\xi^{\text{[st]}}$ (recall \eqref{eq:defs}). Thus, in the sequel we restrict our attention only to the subset $\xi^{\text{[st]}}$ of $\xi$, which is itself a Poisson process on $\Z \times \R_+$ with intensity measure $p(n_\Z \otimes \lambda_{\R^+})$. 
	
	Let $k > 16r$ and $\beta \in \mathcal{B}_k$.
	For $\beta=(i_j,k_j)_{j=1}^r$, the event 
	$\lbrace\exists\, \gamma\in\cC_T: B(\gamma)=\b\rbrace$
	implies that in each row $j$ such that $k_j\geq 3$ there exists a path that crosses a horizontal distance (in row $j$) of length at least $(k_j-2)\ell_h$  in a time lapse shorter than $\ell_v$. 
	Since these crossing events are independent for different~rows, by ignoring those rows on which $\beta$ has less than three boxes, we can bound the probability of the event 
	$\lbrace\exists\,\gamma\in\cC_T: B(\gamma)=\b\rbrace$
	by
	\begin{equation}\label{eq:prod}
		\P\big(
		\exists\, \gamma\in\cC_T: B(\gamma)=\b
		\big) \leq 
		\prod_{j\,:\, k_j \geq 3} \mathbb{P}(\tau_{(k_j-2)\ell_h} \leq \ell_v),
	\end{equation} 
	where $\tau_0 \equiv 1$ and, for $n \geq 1$, $\tau_n$ represents a sum of $n$ independent Exp$(p)$ random variables, i.e., $\tau_n \sim \Gamma(n,p)$. 
	By Chernov's bound, for every $\lambda\geq 0$ and $k_j \geq 3$ we have
	\begin{equation}\label{eq:jb0}
		\mathbb{P}(\tau_{(k_j-2)\ell_h} \leq \ell_v) 
		\leq 
		\exp\Big\{ \lambda \ell_v +(k_j-2)\ell_h \log \frac{p}{\lambda+p}\Big\}.
	\end{equation} Taking $\lambda= (k_j -2)\ell_h/\ell_v - p \geq 0$ (which is  optimal) yields
	\begin{align} 
		\mathbb{P}(\tau_{(k_j-2)\ell_h} \leq \ell_v) 
		&\leq \exp\Big\{ (k_j-2)\ell_h-p\ell_v + (k_j-2)\ell_h\log \frac{p\ell_v}{(k_j-2)\ell_h}\Big\} \nonumber\\
		& \leq \exp \Big\{ -\ell_h (k_j-2)\Big(\log\frac{(k_j-2)\ell_h}{p\ell_v}-1\Big)\Big\} \nonumber\\
		& \leq \exp \Big\{ -\ell_h (k_j-2)\Big(\log (k_j-2)-1\Big)\Big\},
		\label{eq:jb}
	\end{align} 
	where the last inequality follows from the fact that $p\ell_v/\ell_h\leq 1$.
	
	Plugging \eqref{eq:jb} back into \eqref{eq:prod} gives the bound
	\begin{align*}
		\mathbb{P}( \exists\, \gamma\in\cC_T: B(\gamma)=\b
		) &\leq \exp\Big\{ \ell_h\Big(\sum_{j\,:\,k_j \geq 3}(k_j-2) -\sum_{j\,:\,k_j \geq 3}(k_j-2)\log(k_j-2)\Big)\Big\}\\
		& \leq \exp\Big\{ \ell_h\Big[ k - M_3\Big(\frac{1}{M_3}\sum_{j\,:\,k_j \geq 3}(k_j-2)\log(k_j-2)\Big)\Big]\Big\},
	\end{align*} where $M_3:=|\{j:k_j\geq3\}|$. By the convexity of $x \mapsto x \log x$ for $x \geq 1$, using that $M_3 \leq r$ and also that $\sum_{j:k_j\geq 3} k_j \geq k-2r$, we conclude that for $k > 4r$, 
	\[
	\mathbb{P}( \exists\,\gamma\in\cC_T: B(\gamma)=\b) \leq \exp\Big\{ \ell_h\Big[ k - (k-4r)\log\Big(\frac{k}{r}-4\Big)\Big]\Big\}\,.
	\] 
	This is the first inequality of the lemma, 
	the second one follows from the first upon noticing that if $k > 16r$ then $k-4r > \frac{3}{4}k$ and
	\[
	1 - \frac{3}{4}\log\Big(\frac{k}{r}-4\Big) < 0,
	\] so that now the fact that $\ell_h \geq 1$ implies the desired inequality.
\end{proof}

We are now ready to finish the proof of Theorem \ref{theo:bbp} in the case $\Gamma > 1$ (which is the only remaining part). 
Fix $K\in\N$ and define the event 
\[
\cE_K:=\big\lbrace
\exists\,\gamma\in\cC_T:B(\gamma)\in\bigcup_{k>K}\cB_k
\,\big\rbrace.
\]
Then, for any $u>0$,
\[
\P\big(
h_T(0) >u\Gamma
\big)\leq
\P\big(\{h_T(0) >u\Gamma\} \cap \cE_K^c
\big)+
\P\big(
\cE_K
\big).
\]
Observe that, if the events $\lbrace h_T(0) >u\Gamma\rbrace$
and $\cE_K^c$
are both realized,
there~must exist a configuration of boxes $\beta$
in $\bigcup_{r\leq k\leq K}\cB_k$ which contains more than
$u\Gamma$ positive Bernoulli marks. Therefore,
by the union bound, we have 
\begin{multline}\label{eq:smallLarge}
	\P\big(
	h_T(0)>u\Gamma
	\big)
	\leq \sum_{k=r}^K
	\sum_{\beta \in\cB_k}
	\P\big( n_\beta >u\Gamma
	\big)\\
	+\sum_{k>K}
	\sum_{\beta \in\cB_k}
	\P\big(
	\exists\, \gamma\in\cC_T: B(\gamma)=\beta
	\big).
\end{multline}
As long as $16r<K<u\Gamma$,
we can use the three Lemmas~\ref{lemma:cardBk}, \ref{smallk} and~\ref{largek} to bound the above sums by
\begin{multline*}
	\sum_{r\leq k\leq K} \mathrm{e}^{3k}\exp\bigg(
	-k-u\Gamma\Big(
	\log\frac{u\Gamma}{k}-1
	\Big)
	\bigg)
	+\sum_{k>K} \mathrm{e}^{3k}\exp\bigg( k\Big[ 1 - \frac{3}{4}\log\Big(\frac{k}{r}-4\Big)\Big]\bigg)\\
	\leq K\exp\bigg(
	2K-u\Gamma\Big(
	\log\frac{u\Gamma}{K}-1
	\Big)
	\bigg)
	+\sum_{k>K} \exp\bigg( k\Big[ 4 - \frac{3}{4}\log\Big(\frac{K}{r}-4\Big)\Big]\bigg).
\end{multline*}
Now, let us make the following choice of $u$ and $K=K(u)$: we take  
\[
u > 8\mathrm{e}^{14}\,,\qquad
K:=\bigg\lfloor
\frac{u\Gamma}{\mathrm{e}^3}
\bigg\rfloor.
\] 
Observe that this choice of $u$ and $K$ implies in particular that
$K \geq 4\mathrm{e}^{11}\Gamma$.
But, since $r=\lceil T/\ell_v\rceil=\lceil 2\Gamma \rceil$ and $\Gamma > 1$,
we have $r < 3\Gamma$ and thus $K > \mathrm{e}^{11}r$.
Hence, for this choice of $u$ and $K$, we obtain the bound	
\begin{multline}
	\P\big(
	h_T(0)>u\Gamma
	\big)\leq
	\frac{u\Gamma}{\mathrm{e}^3}e^{-u\Gamma}
	+\frac{1}{1-\mathrm{e}^{-2}}\exp\Big(
	-\frac{2u\Gamma}{\mathrm{e}^3}
	\Big)\\
	\leq
	\frac{2u\Gamma}{\mathrm{e}^3}\exp\Big(
	-\frac{2u\Gamma}{\mathrm{e}^3}
	\Big)
\end{multline}
which is valid for all $u>8e^{14}$. Therefore, we reach the following bound for
the expectation of $h_T(0)$:
\begin{multline*}
	\E(h_T(0))=\int_0^\infty \P(h_T(0)>x)\mathrm{d}x=
	\Gamma\int_0^\infty \P(h_T(0)>u \Gamma)\mathrm{d}u\\
	\leq\Gamma \bigg(8\mathrm{e}^{14}+\int_{8\mathrm{e}^{14}}^\infty \P(h_T(0)>u \Gamma)\mathrm{d}u\bigg)
	\leq 8\mathrm{e}^{14}\Gamma+\mathrm{e}^3/2,
\end{multline*}
where the last inequality follows from the bound
\[
\int_{8\mathrm{e}^{14}}^\infty \P(h_T(0)>u \Gamma)\mathrm{d}u \leq \int_{0}^\infty \frac{2u\Gamma}{\mathrm{e}^3}\exp\Big(
-\frac{2u\Gamma}{\mathrm{e}^3}\Big) \mathrm{d}u = \frac{\mathrm{e}^3}{2\G}.
\] Since we are assuming that
$\Gamma=\sigma p T^2>1$, the previous bound implies that
\[
\E(h_T(0)) \leq 9\mathrm{e}^{14} (\sigma p T^2)^{\tfrac{1}{2-\zeta}}.
\] Taking into consideration \eqref{eq:bbd1}, by setting $C:=9\mathrm{e}^{14}$ we obtain \eqref{eq:hml} and thus conclude the proof of Theorem \ref{theo:bbp}.

\subsection{Controlling the $L^1$-norm of $\widetilde{R}^{(T)}_k$ on a good event} 

Our next step in the proof of Proposition \ref{prop:remdiscrete2} will be to obtain suitable control over the expectation of $\widetilde{R}^{(T)}_k$ outside of some ``bad'' event on which the heights $M_i^{(T)}$ are not well-behaved.

To this end, for each $T > 0$ let us consider the random variable $N_T$, which is the number of attainable space-time points in $\cV_T$. Recall that the law of~$N_T$ depends on the sticking parameter $p$, even though this is not made explicit in the notation. With this in mind, let us write
\[
\cB^{(T)}:=\{ \tfrac{1}{2}pT^2 \leq N_T \leq \tfrac{3}{2}pT^2\}
\]and, in analogy with \cite[Equation~(3.8)]{HamblyMartin07}, for each $k \in \N_0$ with $T >~(2k)^{\tfrac{1}{\alpha}}$ define the ``good event'' $\cG^{(T)}_k$ as
\[
\cG^{(T)}_k:= \cB^{(T)} \cap \widetilde{\cG}^{(T)}_k,
\] where
\[
\widetilde{\cG}^{(T)}_k:=\{ F^{-1}\big(1-\tfrac{2i}{N_T}\big) \leq M_{i}^{(T)} \leq F^{-1}\big(1-\tfrac{1}{N_T}\big) \text{ for all }k \leq i \leq N_T\},
\] with the convention that $F^{-1}(y):=0$ if $y \leq 0$. 
Note that, since $p \in [T^{-\zeta},1]$ and $\zeta \in (0,2-\alpha)$, we have $\tfrac{1}{2}pT^2 > k$ so that on $\cB^{(T)}$ the condition $N_T > k$ in the definition of $\widetilde{\cG}^{(T)}_k$ is always satisfied and thus $\cG^{(T)}_k$ is well-defined.

\begin{lemma}\label{lemma:badset} For any fixed $\zeta \in (0,(2-\alpha) \wedge 1)$ we have that
\[
\lim_{k \rightarrow \infty} \left[\sup_{T > (2k)^{1/\alpha}}\left( \sup_{p \geq T^{-\zeta}} \P_p\big((\cG^{(T)}_k)^c\big) \right)\right] = 0.
\]
\end{lemma}

\begin{proof} Using the union bound, we have that
\[
\P_p\big((\cG^{(T)}_k)^c\big) \leq \P_p((\cB^{(T)})^c) + \sum_{N=k+1}^\infty \P_p\big((\widetilde{\cG}^{(T)}_k)^c | N_T = N\big)\P_p( N_T = N). 
\] By Lemma \ref{lemma:control} there exists a constant $C>0$ such that for all $k \geq 1$, 
\begin{equation}
\label{eq:convtp}
\sup_{T > (2k)^{1/\alpha}}\left( \sup_{p \geq T^{-\zeta}} \P_p\big((\cB^{(T)})^c\big)\right)\leq \mathrm{e}^{-Ck^{1/\alpha}}.
\end{equation} On other hand, for $N > k$ we can bound $\P_p\big((\widetilde{\cG}^{(T)}_k)^c | N_T = N\big)$ from above by
\[
\P_p\big(M_{k}^{(N)} > F^{-1}(1-\tfrac{1}{N}) | N_T = N\big) + \sum_{i=k}^{
\lfloor N/2\rfloor}\P_p\big(M_{i}^{(N)} < F^{-1}(1-\tfrac{2i}{N}) | N_T = N\big),
\]
where, since $F(F^{-1}(x))\geq x$ for all $x \in [0,1)$,
\[
\P_p\big(M_{k}^{(N)} > F^{-1}(1-\tfrac{1}{N}) | N_T = N\big) \leq \P(\text{Binomial}(N,\tfrac{1}{N}) \geq k) \leq \frac{1}{k}
\] by Markov's inequality and 
\[
\P_p\big(M_{i}^{(N)} < F^{-1}(1-\tfrac{2i}{N}) | N_T = N\big) \leq \P(\text{Binomial}(N,\tfrac{2i}{N}) < i) \leq \mathrm{e}^{-\frac{i}{4}}
\] by Chernov's bound. Therefore, we obtain that 
\[
\sum_{N=k}^\infty \P_p\big((\widetilde{\cG}^{(T)}_k)^c | N_T = N\big)\P_p( N_T = N) \leq \frac{1}{k} +  \frac{1}{1-\mathrm{e}^{-\frac{1}{4}}}\mathrm{e}^{-\frac{k}{4}}
\] uniformly over $T > 0$ and $p \in [T^{-\zeta},1]$ which, together with the bound in~\eqref{eq:convtp}, implies the result.
\end{proof}

Next, let us define for each $k \in \N$ and $T > 0$, the quantity
\[
L_k^{(T)}:= \max_{A \in \cC_T^{(k)}} |A|,
\] i.e. $L_k^{(T)}$ is the maximum number of the points $U_1^{(T)},\dots,U_{k \wedge N_T}^{(T)}$ which can be collected by a single path $\gamma \in \cC_T$. Once again we point out that the law of $L_k^{(T)}$ depends on the sticking parameter $p$. Taking this into consideration, we have the following uniform control over the expectation of $L_k^{(T)}$ on $\cB^{(T)}$:

\begin{lemma}\label{lemma:hml} There exists a constant $C>0$ such that, for any $\zeta \in [0,1]$, $T \geq 1$, $p \in [T^{-\zeta},1]$ and $k \in \N$, we have
\begin{equation} \label{eq:boundhm}
	\E_p(L_{k}^{(T)} ; \cB^{(T)}) \leq C k^{\tfrac{1}{2-\zeta}}.
\end{equation}
\end{lemma}

\begin{remark} In the particular case when $\zeta = 0$, \eqref{eq:boundhm} becomes the analogue (for ballistic deposition) of the estimate from \cite[Lemma 3.5]{HamblyMartin07} obtained in the context of last passage percolation. Therefore, we can view Lemma \ref{lemma:hml} as an extension of \cite[Lemma 3.5]{HamblyMartin07} (in the context of ballistic deposition) to the~case in which the sticking parameter is allowed to depend (suitably) on~$T$.
\end{remark}

\begin{proof} Note that the law of $L_k^{(T)}$ does not depend on the height distribution, since $L_k^{(T)}$ is a measurable function of only $\xi$ and its $\varepsilon$-marks. In particular, we may assume $L_k^{(T)}$ is the one corresponding to the BBD model introduced in Subsection~\ref{sec:bbd}. Therefore, if $h_T(0;\sigma)$ denotes the height of the cluster over $0$ at time $T > 0$ for~the BBD model with height distribution $\text{Bernoulli}(\sigma)$ then, for any $N \in \N$, we have
\begin{equation} \label{eq:descompbbd}
\E_p(h_T(0;\sigma) ; N_T = N) = \sum_{r=0}^N \P( \text{Binomial}(N,\sigma)=r)\E_p(L_r^{(T)} ; N_T = N).
\end{equation} Indeed, let us call a set $A \subseteq \{1,\dots,N_T\}$ \textit{attainable} if there exists some path $s \in \cC_T$ collecting all the points $U_i^{(T)}$ with $i \in A$ and, given $B \subseteq \{1,\dots,N_T\}$, let us define
\[
\mathcal{C}_T^{B}:= \{ A \subseteq B : A \text{ attainable}\}.
\] (Note that $\cC^{(k)}_T = \cC^B_T$ with $B=\{1,\dots,k\wedge N_T\}$.) Then, since the distribution of $(U_i^{(T)} : i=1,\dots,N_T)$ is invariant under permutations of the indices $i$, it~follows that $\cC^{B}_T$ has the same law as $\cC^{(k)}_T$ for any $B$ such that $|B|=k\wedge N_T$. Hence, since $h_T(0)\overset{d}{=}\max_{A \in \cC^{\widehat B}_T}|A|$ with $\widehat B:=\{ i=1,\dots,N_T : \eta(U_i^{(T)})=1\}$ and the $\eta$-marks are independent of both the exact positions of the $U_i^{(T)}$ and their respective $\varepsilon$-marks, we obtain that, for any $0\leq r \leq N$, 
\[
\E_p(h_T(0;\sigma) \,|\, |\widehat{B}|=r\,,\,N_T=N) = \E_p(L_r^{(T)}|N_T = N)
\] from where \eqref{eq:descompbbd} now readily follows. In particular, since $\E_p(L_r^{(T)} ; N_T = N)$ is increasing in $r$ by definition and $L^{(T)}_{k \wedge N}=L^{(T)}_k$ on the event that $N_T=N$, for any $k,N \in \N$ we have that
\begin{equation} \label{eq:bound}
\P( \text{B}_{N,\sigma}\geq k \wedge N)\E_p(L_{k}^{(T)} ; N_T = N) \leq \sum_{r=0}^N \P( \text{B}_{N,\sigma}=r)\E_p(L_r^{(T)} ; N_T = N),
\end{equation} where we have abbreviated the $\text{Binomial}(N,\sigma)$ distribution by $\text{B}_{N,\sigma}$.

Now, observe that \eqref{eq:bound} holds for all $\sigma \in (0,1]$. Let us take $\sigma:= \min \{ \frac{4k}{pT^2} , 1\}$. Then, whenever $4k \geq pT^2$ we have that $\sigma=1$ and thus \eqref{eq:bound} becomes 
\begin{equation}\label{eq:lowerbound1}
\E_p(L_{k}^{(T)} ; N_T = N) \leq \sum_{r=0}^N \P( \text{B}_{N,\sigma}=r)\E_p(L_r^{(T)} ; N_T = N).
\end{equation} On the other hand, whenever $4k < pT^2$ we have that $k < \frac{N}{2}$ for all $N \geq \tfrac{1}{2}pT^2$ so that $k \wedge N = k$ and thus, since $k=\tfrac{1}{4}\sigma pT^2$ in this case, it holds that  
\begin{equation}\label{eq:lowerbound2}
\P( \text{B}_{N,\sigma}\geq k \wedge N) \geq \P( \text{B}_{N,\sigma} \geq  \tfrac{1}{2}N\sigma) \geq 1 - \mathrm{e}^{-\tfrac{1}{8}\sigma N} \geq 1 - \mathrm{e}^{-\tfrac{1}{4}}
\end{equation} by Chernov's bound. Combining \eqref{eq:lowerbound1} and \eqref{eq:lowerbound2}, we conclude that there exists some constant $C_1>0$ such that, for all $k \in \N$ and $N \geq \tfrac{1}{2}pT^2$, 
\begin{align*}
\E_p(L_{k}^{(T)} ; N_T = N) &\leq C_1 \sum_{r=0}^N \P( \text{B}_{N,\sigma}=r)\E_p(L_r^{(T)} ; N_T = N)\\ & = C_1 \E_p(h_T(0;\sigma) ; N_T = N),
\end{align*} where the equality on the second line is a consequence of \eqref{eq:descompbbd}. Therefore, by first summing over $N \in [\frac{1}{2}pT^2,\tfrac{3}{2}pT^2]$ and then using Theorem \ref{theo:bbp}, we see that there~exists a constant $C_2 > 0$ such that 
\[
\E_p(L_{k}^{(T)} ; \cB^{(T)}) \leq C_2 (\sigma p T^2)^{\tfrac{1}{2-\zeta}} \leq C_2(4k)^{\frac{1}{2-\zeta}} 
\] which, upon taking $C:=4C_2$, completes the proof.
\end{proof}

Next, we shall need a bound on the $L^1$-norm of the random variables $\widetilde{M}_i^{(T)}$ on the event $\widetilde{\cG}_k^{(T)}$.

\begin{lemma}\label{lemma:3.8} Given $\rho \in (0,\tfrac{1}{\alpha})$ there exist constants $c_0, c_1 >0$ and $c_2 \in~(0,1)$ (depending only on $\rho$ and $F$) such that
\[
\E( \tfrac{1}{a_N}M_i^{(T)}\mathbf{1}_{\widetilde{\cG}_k^{(T)}}| N_T = N ) \leq c_0i^{-\tfrac{1}{\alpha}+\rho} +  \frac{c_1}{a_{N}}\mathbf{1}_{\{i \geq c_2 N\}}
\] for all $N$, $i$ and $k$ satisfying $2(1+\tfrac{1}{\alpha})<k\leq i \leq N$.
\end{lemma}

\begin{proof} This is exactly the statement shown in \cite[Lemma 3.8]{HamblyMartin07}, with $N$ in place of $n^2$ and the parameter $k$ here playing the role of $k+1$ in said reference. We therefore omit the proof.
\end{proof}

Finally, the bound on the expectation of $\widetilde{R}_k^{(T)}$ we shall need is the following:

\begin{lemma}\label{lemma:l1} Given $\rho \in (0,\tfrac{1}{\alpha})$ there exists a constant $C>0$ (depending~only on $\rho$ and $F$) such that, for all $k \in \N$ sufficiently large (depending only on~$F$), if $T >(2k)^{1/\alpha}$ and $p \in [T^{-(2-\alpha)},1]$ then
	\[
	\E_p(\widetilde{R}^{(T)}_k ; \cG_k^{(T)}) \leq C\left[\sum_{i=k}^{\infty}\E_p(L_i^{(T)};\cB^{(T)})i^{-\tfrac{1}{\alpha}-1+\rho}+ T(pT^2)^{-\tfrac{1}{\alpha}+\rho} + \frac{T}{a_{pT^2}}\right]
	\] 
\end{lemma}
\begin{proof}
	Since $N_T$ is almost surely finite, for each $k \in \N$ we may always find some $A^{(T)}_k \subseteq \{k,k+1,\dots\}$ such that 
	\[
	\widetilde{R}^{(T)}_k = \sum_{i \in A^{(T)}_k} \widetilde{M}_i^{(T)}.
	\] For such $A^{(T)}_k$, define $I_i:= |A^{(T)}_k \cap \{1,\dots,i\}|$ for each $i \geq k$ and set $I_{k-1} \equiv 0$. We then have that $I_{i}-I_{i-1} = \mathbf{1}_{\{i \in A^{(T)}_k\}}$ and $I_i \leq L_i^{(T)}$ for each $i \geq k$, so that, using the convention $\widetilde{M}_{N_T+1}^{(T)}:=0$, we obtain the bound
	\begin{align}
	\widetilde{R}^{(T)}_k &= \sum_{i \in A^{(T)}_k} \widetilde{M}_i^{(T)} \nonumber \\
	& = \sum_{i=k}^{N_T} \widetilde{M}_i^{(T)}(I_i - I_{i-1}) \nonumber \\
	& =\sum_{i=k}^{N_T-1} I_i (\widetilde{M}_i^{(T)} - \widetilde{M}_{i+1}^{(T)}) + I_{N_T}\widetilde{M}_{N_T}^{(T)}\nonumber\\
	& \leq \sum_{i=k}^{N_T} L_i^{(T)} (\widetilde{M}_i^{(T)} - \widetilde{M}_{i+1}^{(T)}) \label{eq:boundrt} 
	\end{align} since $\widetilde{M}_i^{(T)} \geq \widetilde{M}_{i+1}^{(T)} \geq 0$ for all $i=k,\dots,N_T-1$.

Now, for each $N \in [\tfrac{1}{2}pT^2,\tfrac{3}{2}pT^2]$ define $\cG_k^{(T),N}:=\{ N_T = N\} \cap \widetilde{\cG}_k^{(T)}$. Then, by taking expectations on \eqref{eq:boundrt} and since $\widetilde{M}^{(T)}:=(\widetilde{M}_i^{(T)} : i=1,\dots,N_T)$ and $\cV_T$ are independent given $N_T$, we obtain that
\[
\E_p(\widetilde{R}^{(T)}_k ; \widetilde{\cG}_k^{(T),N})\leq \sum_{i=k}^{N} \E_p(L_i^{(T)} ; N_T = N)  \E([\widetilde{M}_i^{(T)}-\widetilde{M}_{i+1}^{(T)}]\mathbf{1}_{\widetilde{\cG}_k^{(T)}}|N_T=N).
\] Furthermore, if for each $i$ we abbreviate 
\[
\ell_N(i):=\E_p(L_i^{(T)} ; N_T = N)\hspace{0.5cm} \text{ and }\hspace{0.5cm} m_N(i):=\E(\widetilde{M}_i^{(T)}\mathbf{1}_{\widetilde{\cG}_k^{(T)}}|N_T=N)
\] then, recalling that $m_N(N+1)=0$, we can rewrite the above as
\begin{align*}
\E_p(\widetilde{R}^{(T)}_k ; \widetilde{\cG}_k^{(T),N})&\leq \sum_{i=k}^{N} \ell_N(i)(m_N(i)-m_N(i+1))\\& = \ell_N(k)m_N(k) + \sum_{i=k+1}^{N} m_N(i)(\ell_N(i)-\ell_N(i-1)).
\end{align*}
Now, let us define for $i\geq k$ the quantity
\begin{equation} \label{eq:defmistar}
	m^*(i):=\frac{a_{\tfrac{3}{2}pT^2}}{a_{pT^2}}\left[ c_0i^{-\tfrac{1}{\alpha}+\rho} +  \frac{c_1}{a_{\tfrac{1}{2}pT^2}}\mathbf{1}_{\{i \geq \tfrac{c_2}{2}pT^2\}}\right],
\end{equation} where $c_0,c_1,c_2$ are the constants from Lemma \ref{lemma:3.8}, and note that this lemma implies that $m_N(i) \leq m^*(i)$ for all $k \leq i \leq N$ with $N \in [\tfrac{1}{2}pT^2,\tfrac{3}{2}pT^2]$. Hence, since $\ell_N(i)-\ell_N(i-1)=0$ for all $i\geq N+1$, we can further bound 
\begin{align*}
\E_p(\widetilde{R}^{(T)}_k ; \widetilde{\cG}_k&^{(T),N}) \leq \ell_N(k)m^*(k) + \sum_{i=k+1}^{\lceil \tfrac{3}{2}pT^2 \rceil} m^*(i)(\ell_N(i)-\ell_N(i-1)) \\
& = \sum_{i=k}^{\lceil \tfrac{3}{2}pT^2 \rceil} \ell_N(i)(m^*(i)-m^*(i+1)) + \ell_N(\lceil \tfrac{3}{2}pT^2 \rceil)m^*(\lceil \tfrac{3}{2}pT^2 \rceil+1).
\end{align*}
Now, by summing the last bound over $N \in [\tfrac{1}{2}pT^2,\tfrac{3}{2}pT^2]$, we conclude that $\E_p(\widetilde{R}^{(T)}_k ; \cG_k^{(T)})$ is bounded from above by
\begin{equation} \label{eq:lsum}
\sum_{i=k}^{\lceil \tfrac{3}{2}pT^2 \rceil}\E_p(L_i^{(T)}; \cB^{(T)})(m^*(i)-m^*(i+1))+ \E_p(L_{\lceil \tfrac{3}{2}pT^2 \rceil}^{(T)}; \cB^{(T)})m^*(\lceil \tfrac{3}{2}pT^2 \rceil+1).
\end{equation} Furthermore, since 
\begin{equation} \label{eq:bsup}
\sup\left\{ \max\left\{\frac{a_{\tfrac{3}{2}pT^2}}{a_{pT^2}}\,,\,\frac{a_{pT^2}}{a_{\tfrac{1}{2}pT^2}}\right\} : T > (2k)^{1/\alpha}\,,\,p \in [T^{-(2-\alpha)},1]\right\} < \infty
\end{equation} for all $k$ large enough by our assumptions on $F$, a simple computation using the mean value theorem shows that, for any $i \geq k$,
\[
m^*(i) - m^*(i-1) \leq C_1 i^{-\tfrac{1}{\alpha}-1+\rho}
\] for some constant $C_1 > 0$ depending only on $\rho$ and $F$. On the other hand, by Lemma \ref{lemma:hml} and the bound $\lceil x \rceil \leq 2x$ valid for all $x \geq \tfrac{1}{2}$, we have that
\[
\E_p(L_{\lceil \tfrac{3}{2}pT^2 \rceil}^{(T)}; \cB^{(T)}) \leq C_2 T
\] for some constant $C_2 > 0$. Finally, since by \eqref{eq:bsup} we have
\[
m^*(\lceil \tfrac{3}{2}pT^2 \rceil+1) \leq C_3\left( (pT^2)^{-\tfrac{1}{\alpha}+\rho} + \frac{1}{a_{pT^2}}\right)
\] for some $C_3 > 0$ depending only on $\rho$ and $F$, plugging all the estimates above in \eqref{eq:lsum} we conclude the result.
\end{proof}

\subsection{Conclusion of the proof} \label{sec:derivation}

With all the previous estimates, we now finish the proof of Proposition~\ref{prop:remdiscrete2}. Fix $\delta > 0$ and, using Markov's inequality, let us decompose
\[
\P_p\big( \widetilde{R}_k^{(T)} > \delta\big) \leq \P_p\big((\cG^{(T)}_k)^c\big) + \frac{1}{\delta}\E_p(\widetilde{R}^{(T)}_k ; \cG_k^{(T)}).
\] In light of Lemma \ref{lemma:badset}, it will suffice to show that
\begin{equation}\label{eq:enuff}
\lim_{k \rightarrow \infty} \left[\sup_{T > (2k)^{1/\alpha}}\left( \sup_{p \geq T^{-\zeta}}\E_p(\widetilde{R}^{(T)}_k ; \cG_k^{(T)}) \right)\right] = 0.
\end{equation} Fix $\rho \in (0,\tfrac{1}{\alpha})$ small enough so that $(2-\zeta)(\tfrac{1}{\alpha}-\rho)>1$ (which is possible since $\zeta < 2-\alpha$). By Lemma \ref{lemma:l1}, for any $T>(2k)^{1/\alpha}$ and $p \in [T^{-\zeta},1]$, 
\begin{equation}\label{eq:boundtail}
\E_p(\widetilde{R}^{(T)}_k ; \cG_k^{(T)}) \leq C\left[\sum_{i=k}^{\infty}\E_p(L_i^{(T)};\cB^{(T)})i^{-\tfrac{1}{\alpha}-1+\rho}+ T(pT^2)^{-\tfrac{1}{\alpha}+\rho} + \frac{T}{a_{pT^2}}\right].
\end{equation} Since $a_t=t^{1/\alpha}L(t)$ for some slowly varying function~$L$, by choice of $\rho$ and $\zeta$  we have
\[
\lim_{T \rightarrow \infty} \left[\sup_{p \geq T^{-\zeta}} \left(  T(pT^2)^{-\tfrac{1}{\alpha}+\rho}+\frac{T}{a_{pT^2}}\right)\right] = 0.
\] 
On the other hand, by Lemma \ref{lemma:hml} we have that
\[
\sup_{T > (2k)^{1\alpha}}\left(\sup_{p \geq T^{-\zeta}}\left[ \sum_{i=k}^{\infty}\E_p(L_i^{(T)};\cB^{(T)})i^{-\tfrac{1}{\alpha}-1+\rho} \right]\right)\leq C \sum_{i=k}^{\infty}i^{\tfrac{1}{2-\zeta}-\tfrac{1}{\alpha}-1+\rho} \underset{k \rightarrow \infty}{\longrightarrow} 0
\] since $\tfrac{1}{2-\zeta}-\tfrac{1}{\alpha}-1+\rho < -1$ by choice of $\rho$. Taking into consideration \eqref{eq:boundtail}, this yields \eqref{eq:enuff} and thus concludes the proof. 

\section{Technical Appendix}\label{sec:techapp}

This final section is devoted to giving the proofs of various technical lemmas used throughout Section \ref{coupling} for the proof of Proposition \ref{prop:coupling2}. 

\subsection{Proof of Lemma \ref{lemma:control}}\label{sec:proofrelemma}

We split the proof of the lemma into three parts:

\noindent \textbf{Proof of \eqref{eq:expectations}}.
We begin by showing that
\begin{equation}
\label{eq:npartialc}
\E_p(N^\partial_T)=2T-\frac{1-\mathrm{e}^{-pT}}{p}.
\end{equation} If for $A \in \Z \times \R_+$ we define $n_T(A):=|\xi_T \cap A|$ then, by inclusion-exclusion, we have that 
\begin{equation}\label{eq:npartial}
N^\partial_T = n_T(G^+_T)+n_T(G^-_T)-n_T(G^+_T\cap G^-_T),
\end{equation} where, for $t \in [0,T]$, we define $G^{+}_t:=\text{graph}(\gamma^{+,T}|_{[T-t,T]})$ and $G^-_t$ analogously. Observe that $(n_T(G^+_t))_{t \in [0,T]}$ and $(n_T(G^-_t))_{t \in [0,T]}$ are both Poisson processes with rate $1$. In particular, in light of \eqref{eq:npartial}, \eqref{eq:npartialc} will follow once we show that
\begin{equation} \label{eq:b2}
\E_p(n_T(G^+_T\cap G^-_T)) =\frac{1-\mathrm{e}^{-pT}}{p}.
\end{equation} To this end, let $H_T:=\inf\{ t \in [0,T] : s^{+,T}(t)\neq s^{-,T}(t)\}$ be the first time that the reversed right/leftmost paths $s^{+,T}$ and $s^{-,T}$ disagree (the first time that they reach a sticky point), and set $H_T \equiv T$ if these paths never split. Since $s^{+,T}$ and $s^{-,T}$ coincide until $H_T$ and then disagree forever afterwards, it follows that
\begin{equation} \label{eq:martingale}
n_T(G^+_T \cap G^-_T) = n_T(G^+_{H_T}). 
\end{equation} In particular, the optional stopping theorem for Poisson processes then yields that $\E_p(n_T(G^+_T\cap G^-_T)) = \E_p(H_T)$. Upon noticing that $H_T$ is distributed as $E_p \wedge T$, where $E_p$ is an exponential random variable with parameter $p$, \eqref{eq:b2} (and therefore \eqref{eq:npartialc}) now follows by a straightforward computation.

Let us now turn to the proof of the formula
\begin{equation} \label{eq:ninteriorc}
\E_p(N^\circ_T)=pT^2-T+\frac{1-\mathrm{e}^{pT}}{p}.
\end{equation}
To this end, we first observe that, since the conditional distribution of $N^\circ_T$ given $\partial \Delta_T$ is Poisson of parameter $n_\Z \otimes \lambda_{\R_+}(\Delta_T^\circ)$ (see paragraph below~\eqref{eq:decomp}), by conditioning on $\partial \Delta_T$ we obtain 
\[
\E_p(N^\circ_T)= \E_p( n_\Z \otimes \lambda_{\R_+}(\Delta_T^\circ)).
\] Furthermore, by the definition of $\Delta^\circ_T$ and the Fubini-Tonelli theorem, 
\begin{align}
n_\Z \otimes \lambda_{\R_+}(\Delta_T^\circ)&=\int_0^{T-H_T}\big(\gamma^{+,T}(t)-1 -\gamma^{-,T}(t)\big)\mathrm{d}t \nonumber \\
&=\int_0^{T}\big(\gamma^{+,T}(t) -\gamma^{-,T}(t)\big)\mathrm{d}t - T + H_T \label{eq:boundarea}
\end{align} where $H_T$ is, as before, the first time the time-reversed paths $s^{+,T}$ and $s^{-,T}$ split. Since these reversed paths are both Poisson jump processes with rate~$p$, recalling that $\E_p(H_T)=\frac{1-\mathrm{e}^{-pT}}{p}$ we conclude that
\begin{align*}
\E_p(N^\circ_T) &= 2\int_0^T \E_p(\gamma^{+,T}(t))\mathrm{d}t - T + \frac{1-\mathrm{e}^{-pT}}{p}.\\
&= 2\int_0^T p(T-t)\mathrm{d}t - T + \frac{1-\mathrm{e}^{-pT}}{p} = pT^2 - T +\frac{1-\mathrm{e}^{-pT}}{p}
\end{align*} and so \eqref{eq:ninteriorc} now follows. Together with \eqref{eq:npartialc}, this concludes the proof of~\eqref{eq:expectations}.

\noindent \textbf{Proof of \eqref{eq:triangle}}.
		The process $\widetilde s(t):=s^+(tT)/pT$, $t\in [0,1]$,  
		is a Poisson process with generator of the form
		$$A^h f(x)=\frac{1}{h}\big(
		f(x+h)-f(x)
		\big),$$
		where, in our case, $h=(pT)^{-1}$.
		We are looking for a large deviations result for the trajectories of such a Poisson process, in the spirit of~\cite{FreidlinWentzell}.
		In view of Section~5.3 therein, the action functional for this process is
		given by:
		$$S(\phi)\,:=\,\begin{cases}
			\displaystyle\int_0^1\big(
			\phi'(t)\ln\phi'(t)-\phi'(t)+1
			\big)\,dt &\text{if}\ \phi\ \text{is abs. cont. and $\phi'\geq 0$},\\
			+\infty &\text{otherwise}.
		\end{cases}$$
		This action functional has a unique minimizing trajectory, namely $\phi^*(t)=t$. The Poisson process with generator $A^h$
		satisfies a Large Deviations Principle
		with action functional $S$ and normalizing coefficient $1/h$
		(cf. Theorem~5.3.3 in~\cite{FreidlinWentzell}).
		Let us denote by $d$ the supremum metric on the space of trajectories from $[0,1]$
		to $\R$ and let us define, for $\eta \geq 0$, the sets of trajectories
		\[
		\Phi(\eta):=\big\lbrace
		\phi : S(\phi)\leq \eta 
		\big\rbrace.
		\]
		The sets $\Phi(\eta) $ are compact and, furthermore, $\Phi(0)$
		contains only the optimal trajectory $\phi^*$.
		Using formula~(3.3.2) in~\cite{FreidlinWentzell}, 
		we have that for any $\d',\gamma,\eta>0$, there exists
		an $h_0>0$ such that for all $h\leq h_0$
		\[
		P_h\Big(
		d\big(\widetilde s, \Phi(\eta)\big)\geq \d'
		\Big)\leq
		\exp\bigg(
		-\frac{\eta- \gamma}{h}
		\bigg),
		\] where $P_h$ denotes the law of the process $\widetilde{s}$ having generator $A^h$.
		This implies that, if we fix 
		$\zeta,\d',\gamma,\eta>0$, there exists
		an $T_\zeta>0$ such that for all $T\geq T_\zeta$
		and $p\in[T^{-\zeta},1]$ we have 
		\begin{equation} \label{eq:ldp}
		\P_p\Big(
		d\big(
		\widetilde s, \Phi(\eta)\big)\geq \delta'
		\Big)\leq
		\exp\big(
		-T^{1-\zeta}(\eta- \gamma)
		\big).
		\end{equation}
		However, since this inequality is only valid for $\eta>0$, this will be insufficient for our purposes.
	Therefore, obtaining the desired estimate will require a bit more work. Indeed, given $\delta > 0$ let us consider the set 
		\[
		A_\delta:=\big\lbrace
		\phi : d(\phi,\phi^*)\geq \delta
		\big\rbrace.
		\]
		The values of $S$ on the closed set $A_\delta$ are all strictly positive,
		therefore so is the infimum value of $S$ over this set.
		Pick $\gamma > 0$ (depending on $\delta$) so that 
		\[
		\inf \big\lbrace
		S(\phi):\phi\in A_\delta
		\big\rbrace> 2\gamma.
		\]
		In particular, this choice of $\gamma$ gives $A_\delta\cap \Phi(2\gamma)=\emptyset$.
		Next, choose $\delta'$ such~that 
		$$\delta'<d(A_\delta,\Phi(2\gamma)).$$
		Then, by using \eqref{eq:ldp} with $2\gamma$ in place of $\eta$, we obtain
		\[
		\sup_{p \geq T^{-\zeta}}\P_p\Big(
		d(\widetilde s, \phi^*)\geq \delta
		\Big)\leq \exp\big(
		-T^{1-\zeta}\gamma
		\big).
		\] The analogous estimate for $s^{-,T}$ follows by symmetry. By combining the two, this concludes the proof of \eqref{eq:triangle}.

\noindent \textbf{Proof of Equation \eqref{eq:T2pts}}. Since $N_T = N^\partial_T + N^\circ_T$, to obtain \eqref{eq:T2pts} it will suffice to show that for any $\zeta \in (0,1)$ and $\delta>0$ there exist $T_{\delta,\zeta},C_{\delta,\zeta} > 0$ such that, for all $T>T_{\delta,\zeta}$, 
\begin{equation}\label{eq:Tp2eq1}
\sup_{p \geq T^{-\zeta}} \P_p\Big(|N^\partial_T - 2T| > \delta T\Big)\leq \mathrm{e}^{-C_{\delta,\zeta}T^{1-\zeta}}
\end{equation}
and
\begin{equation}\label{eq:Tp2eq2}
	\sup_{p \geq T^{-\zeta}} \P_p\Big(|N^\circ_T - pT^2| > \delta pT^2\Big)\leq \mathrm{e}^{-C_{\delta,\zeta}T^{1-\zeta}}.
	\end{equation}
Let us show \eqref{eq:Tp2eq1} first. Since $n_T(G^+_T) \sim \text{Poisson}(T)$, using Chernov's bound 
\begin{equation}\label{eq:chernov}
\P\big(|X-\lambda| >  x \big) \leq 2\mathrm{e}^{- \frac{x^2}{2 (\lambda + x)}}
\end{equation} valid for all $x > 0$ and any random variable $X \sim \text{Poisson}(\lambda)$, it follows that for any $\delta \in (0,1)$,
\begin{equation} \label{eq:poissonldp}
\P_p(|n_T(G^+_T) - T| > \delta T) \leq 2\mathrm{e}^{- \tfrac{\delta}{4}T}.
\end{equation} By symmetry, the same bound holds for $n_T(G^-_T)$. On the other hand, in view of \eqref{eq:martingale}, by the union bound and \eqref{eq:chernov} again, for any $\delta \in (0,1)$ and $T > \frac{2}{\delta}$ we obtain that
\begin{align}
\P_p(n_T(G^+\cap G^-) > \delta T) &\leq \P_p\big( n_T(G^+_{H_T}) > \delta T\,,\,H_T \leq \tfrac{\delta}{2}T\big) + \P_p(H_T > \tfrac{\delta}{2}T) \nonumber \\
& \leq  \P_p\Big( n_T(G^+_{\tfrac{\delta}{2}T}) > \delta T\Big)+\mathrm{e}^{-\tfrac{\delta}{2}pT} \nonumber \\
& \leq 2\mathrm{e}^{-\frac{2\delta}{3}T} + \mathrm{e}^{-\tfrac{\delta}{2}T^{1-\zeta}}. \label{eq:finalbound2}
\end{align} Taking into consideration \eqref{eq:npartial}, combining \eqref{eq:poissonldp} and \eqref{eq:finalbound2} then yields \eqref{eq:Tp2eq1}.

Let us now turn to the proof of \eqref{eq:Tp2eq2}. In light of \eqref{eq:triangle}, it will be enough to show that for some sufficiently small  $\delta'>0$ (depending only on $\delta$) there exist  $T'_{\delta,\zeta},C'_{\delta,\zeta} > 0$ such that, for all $T>T'_{\delta,\zeta}$,
\begin{equation}\label{eq:checkint}
\sup_{p \geq T^{-\zeta}} \P_p( \{|N^\circ_T - pT^2| > \delta pT^2\} \cap \mathcal{P}_{T,\delta'}) \leq \mathrm{e}^{-C'_{\delta,\zeta}T^{2-\zeta}},
\end{equation} where $\mathcal{P}_{T,\delta'}$ denotes the complement of the event in \eqref{eq:triangle} with $\delta'$ in place of~$\delta$. To see \eqref{eq:checkint}, note that,  for all $T$ large enough (depending only on $\zeta$ and $\delta$) so as to have $T \leq \delta'pT^2$ for all $p \in [T^{-\zeta},1]$, on the event $\mathcal{P}_{T,\delta'}$ we have
\begin{equation}\label{eq:boundarea2}
|n_\Z \otimes \lambda_{\R_+}(\Delta^\circ_T)-pT^2| \leq 3\delta'pT^2.
\end{equation} Indeed, by \eqref{eq:boundarea} we have that
\[
 \bigg|n_\Z \otimes \lambda_{\R_+}(\Delta^\circ_T)  - \int_0^{T}\big(s^{+,T}(t) -s^{-,T}(t)\big)\mathrm{d}t \bigg| \leq T \leq \delta'pT^2
 \] so that, upon noticing that $\sup_{t \in [0,T]} |(s^{+,T}(t) -s^{-,T}(t)) - 2pt| \leq 2\delta'pT$ on the event $\mathcal{P}_{T,\delta'}$, \eqref{eq:boundarea2} now follows by computing a simple integral. Therefore, since conditional on $\partial \Delta_T$ the random variable $N^\circ_T$ is Poisson distributed with parameter $n_\Z \otimes \lambda_{\R_+}(\Delta_T^\circ)$, \eqref{eq:checkint} now follows by conditioning on $\partial \Delta_T$ and using Chernov's bound in \eqref{eq:chernov} with the help of \eqref{eq:boundarea2}, we omit the details which are straightforward. This concludes the proof of~\eqref{eq:T2pts} and thus of Lemma~\ref{lemma:control}.
 
\subsection{Verifying condition (C1')} \label{sec:techappc1}

We now complete the details of the proof of condition (C1') in Section \ref{sec:condc1}, by proving Lemmas \ref{lemma:c1.a} and \ref{lemma:c1.b}.

\begin{proof}[Proof of Lemma \ref{lemma:c1.a}] 
	Since $a_N^{-1}M^\circ_{i,(N)} \overset{as}{\longrightarrow} M_i$ as $N \rightarrow \infty$ for each $i \in \N$ and also $N^\circ_T \overset{\P}{\longrightarrow} \infty$ uniformly over all $p \in [T^{-\zeta},1]$ by \eqref{eq:T2pts}, we have that
	\[
	\lim_{T \rightarrow \infty} \left[\sup_{p \geq T^{-\zeta}} \P_p\left( \sum_{i=1}^{k} |M_i - (a_{N^\circ_T})^{-1}M_i^{\circ,(T)}| > \delta\,,\,N^\circ_T\geq k\right)\right]=0.
	\] However, since $\frac{a_{N^\circ_T}}{a_{pT^2}} \overset{\P}{\longrightarrow} 1$ holds uniformly over all $p \in [T^{-\zeta},1]$ by \eqref{eq:T2pts} again, this readily implies \eqref{eq:c1b}.
\end{proof}

\begin{proof}[Proof of Lemma \ref{lemma:c1.b}]
Notice that, conditional on $N_T$, the collection of space-time points having the $k$ largest heights in $\cV_T$ is uniformly distributed among all subsets of $N_T \wedge k$ elements in $\cV_T$. Thus, by conditioning on $N^\circ_T$ and $N^\partial_T$, it follows that
\[
\P_p(\Omega_k)= \E_p( \mathbf{1}_{N^\circ_T \geq k} \P( \cH(N_T,N^\partial_T,k) = 0 \,|\, N^\circ_T,N^\partial_T)),
\] for $\cH(N,D,k)$ the hypergeometric distribution with total population size~$N$, total number of defective objects in the population $D$ and sample size $k$. Moreover, on the event $\{N^\circ_T \geq k\}$ we have that
\[
\P( \cH(N_T,N^\partial_T,k) = 0 \,|\, N^\circ_T,N^\partial_T) = \frac{\binom{N^\circ_T}{k}}{\binom{N_T}{k}} = \prod_{i=0}^{k-1} \left(1 - \frac{N^\partial_T}{N_T-i}\right).
\] Therefore, since $\frac{N^\partial_T}{N_T-i} \overset{\P}{\longrightarrow} 0$ uniformly over $p \in [T^{-\zeta},1]$ for all $i=0,\dots,k-1$ due to the fact that $\zeta <1$, by \eqref{eq:boundnt} and the bounded convergence theorem (which can be applied uniformly over $p \in [T^{-\zeta},1]$ due to our uniform control) we conclude that \eqref{eq:boundomega} holds.
\end{proof}

\subsection{Verifying condition (C2')}\label{sec:techappc2}

Next, we complete the verification of (C2') by giving the proof of Lemma~\ref{lemma:c2.1}.

\begin{proof}[Proof of Lemma \ref{lemma:c2.1}] We only show \eqref{eq:suff3}, \eqref{eq:suff4} follows by a similar argument. For each $p,\rho \in (0,1]$, let $\Delta^{(p)}(\rho)$ be the $p$-slope triangle with apex $(0,1-\rho)$ given by the formula
\[
\Delta^{(p)}(\rho):=\big\lbrace
(x,t)\in\R^2 : 0\leq t\leq 1-\rho\,,\, |x|\leq p(1-\rho-t)\big\rbrace.
\] Since each $U_i^{(p)}$ is uniformly distributed on $\Delta^{(p)}$ and, furthermore, since by the Fubini-Tonelli theorem we have
\[
\frac{\lambda_{\R} \otimes \lambda_{[0,1]}(\Delta^{(p)}\setminus \Delta^{(p)}(\rho))}{\lambda_{\R} \otimes \lambda_{[0,1]}(\Delta^{(p)})}\leq \frac{(2p)\rho}{p} = 2\rho,
\] it follows that, for each $i \in \N$,
\[
\lim_{\rho \rightarrow 0^+} \left[\sup_{p \in (0,1]} \P( U_i^{(p)} \notin \Delta^{(p)}(\rho) ) \right]= 0.
\] Thus, in order to obtain \eqref{eq:suff3}, it will suffice to show that, for any $\rho \in (0,1)$, one can find $\delta' > 0$ such that, for all $T > 0$ large enough, on the event $\mathcal{P}_{T,\delta'}$ one has the inclusion
\begin{equation} \label{eq:finalinc}
\varphi_T(\Delta^{(p)}(\rho)) \subseteq \Delta^\circ_T
\end{equation} for all $p \in (0,1]$. To this end, observe that for each $(x,t) \in \R \times [0,1]$ we have
\begin{equation} \label{eq:goodinc}
\left\|\frac{1}{T}\varphi_T(x,t) - (x,t)\right\|=\left|x - \frac{1}{T}\lfloor Tx +1/2\rfloor\right| \leq \frac{1}{2T},
\end{equation} so that for all $T$ large enough (depending on $\rho$) we have 
\[
\frac{1}{T}\varphi_T(\Delta^{(p)}(\rho)) \subseteq \Delta^{(p)}(\rho/2).
\] But it follows immediately from the definition of $\mathcal{P}_{T,\delta'}$ that, if we take $\delta'<\frac{\rho}{2}$, on the event $\mathcal{P}_{T,\delta'}$ we have 
\begin{equation}\label{eq:fakeinc}
\Delta^{(p)}(\rho/2) \subseteq \frac{1}{T} \Delta^\circ_T.\footnote{In reality, since we have defined $\Delta^\circ_T$ as a subset of $\Z \times [0,T]$, the inclusion in \eqref{eq:fakeinc} is only true if we intersect the left-hand side with $\frac{1}{T}\Z \times [0,1]$. However, this does not affect the argument since $\frac{1}{T}\varphi_T(\Delta^{(p)})$ is also contained in $\frac{1}{T}\Z \times [0,1]$.}
\end{equation} In combination with \eqref{eq:goodinc}, this yields \eqref{eq:finalinc} and so \eqref{eq:suff3} follows. 

The argument to show \eqref{eq:suff4} is completely analogous, one only has to replace $\Delta^{(p)}$ with $\Delta^\circ_T$ and $\Delta^{(p)}(\rho)$ with
\[
\Delta_T(\rho):=\{ (x,t) \in \Z \times [0,T(1-\rho)] : \gamma^{-,T}(t+\rho T) < x < \gamma^{+,T}(t+\rho T)\}. 
\] We omit the details.
\end{proof}

\subsection{Verifying condition (C3')}\label{sec:techappc3}

To conclude, we now give the proofs of the various lemmas appearing in the proof of condition (C3') in Section \ref{sec:condc3}. Recall that, throughout the latter, we often used the abbreviation $(n_i,t_i):=U_i^{\circ,(T)}$ for each $i=1,\dots,N^\circ_T$, which we shall continue to use here during the proofs.

\begin{proof}[Proof of Lemma \ref{lemma:c3.1}] We begin by giving ourselves some ``room to maneuver''. Observe that, since almost surely we have $t(U_i)>0$ and also that the pair $\{U_i,U_j\}$ is not contained in any line of slope $\pm 1$, we may as well assume that there exists $\rho>0$ such that $U_j \in \Delta(U_i;2\rho)$. More precisely, since 
\[
\lim_{\rho \searrow 0} \P( U_j \in \Delta(U_i) \,,\, U_j \notin \Delta(U_i;2\rho))=0,
\] we see that, in order to prove the lemma, by the union bound it will suffice to show that, for any $i \neq j \in \{1,\dots,k\}$ and  $\rho \in (0,1)$, 
\begin{equation} \label{eq:c3suff3}
		\lim_{T \rightarrow \infty} \left[\sup_{p \geq T^{-\zeta}} \P_p( \{ U_j \in \Delta(U_i;2\rho)\,,\,U^{\circ,(T)}_j \notin \Delta_T(U^{\circ,(T)}_i)\} \cap \Omega''_{k} \cap \mathcal{I}^{(T)}_{i,\rho} )\right] = 0.
	\end{equation}  

To show \eqref{eq:c3suff3}, recall first the notation $r_{p,T}(n,t):=r_p(\frac{1}{T}(n,t))=(\frac{n}{pT},\frac{t}{T})$ and observe that, since on $\Omega''_k$ we have 
	$\|U_j - r_{p,T}(U_j^{\circ,(T)})\| \leq T^{\zeta-1} \rightarrow 0$ as 
	$T \rightarrow \infty$ (see~\eqref{eq:boundrp}), if $T$ is large enough (depending only on $\rho$ and $\zeta$) then on the event $\{U_j \in \Delta(U_i;2\rho)\} \cap \Omega''_{k} \cap \mathcal{I}^{(T)}_{i,\rho}$ we have that 
	\[
	r_{p,T}(U_j^{\circ,(T)}) \subseteq \Delta(U_i;\rho) \subseteq  r_{p,T}(\Delta_T(U_i^{\circ,(T)}))
	\] and therefore that $U^{\circ,(T)}_j \in \Delta_T(U^{\circ,(T)}_i)$. In particular, for all such $T$ the event in \eqref{eq:c3suff3} is empty, so that \eqref{eq:c3suff3} and hence the entire lemma now follow. 
\end{proof}

\begin{proof}[Proof of Lemma \ref{lemma:bigpathscoincide}] To establish \eqref{eq:equality}, note that $s^{+,T}_{(n_i,t_i)}$ coincides with $\widehat{s}^{\,+,T}_{(n_i,t_i)}$ until the first time $t > 0$ in which these two paths encounter a sticky point in $\cV_T^{\circ,\,\text{small}}$ (other than $(n_i,t_i)$ itself). In particular, the event in \eqref{eq:equality} is contained in the intersection 
\begin{equation} \label{eq:equality3}
	\{N^{\circ,\,\text{small}}_T \geq i\} \cap \{ | \cV_T^{\circ,\,\text{small}} \cap \text{graph}(\widehat{\gamma}^{+,T}_{(n_i,t_i)}|_{[0,t_i)})| \geq 1\}, 
\end{equation} where $\widehat{\gamma}^{+,T}_{(n_i,t_i)}$ denotes the (non-time-reversed) rightmost path in $\widehat{\cC}_T(n_i,t_i)$ and $\widehat{\gamma}^{+,T}_{(n_i,t_i)}|_{[0,t_i)}$ is its restriction to $[0,t_i)$ (i.e. we exclude $(n_i,t_i)$ from the graph). By conditioning on $\cV_T^{\circ,\,\text{big}} \cup \partial \cV_T$ and $N^{\circ,\,\text{small}}_T$ (and also $(n_i,t_i)$, to be precise), the $\P_p$-probability of the event in \eqref{eq:equality3} is bounded from above by 
\begin{equation} \label{eq:probap}
	\sup_{p \geq T^{-\zeta}} \P_p\left( \text{Binomial}\left(N^{\circ,\,\text{small}}_T -1, q_{p,T}(U^{\circ,(T)})\right) \geq 1\right),
\end{equation} where $U^{\circ,(T)}$ is uniformly distributed on $\Delta_T^\circ$ and, for each $(n,t) \in \Z \times \R_+$, we write
\[
q_{p,T}(n,t):=\P_p\left( U^{\circ,(T)} \in \text{graph}(\widehat{\gamma}^{+,T}_{(n,t)}) \,\Big|\, \cV_T^{\circ,\,\text{big}} \cup \partial \cV_T \right).
\] Upon noticing that for any $t \in [0,T]$ we have $n_\Z \otimes \lambda_{\R_+} ( \widehat{\gamma}^{+,T}_{(n,t)}) \leq t \leq T$ since the path $\widehat{\gamma}^{+,T}_{(n,t)}$ is piecewise constant and, moreover, that on the event $\mathcal{P}_{T,\tfrac{1}{4}}$ (the complement of the event in \eqref{eq:triangle} with $\delta=\frac{1}{4}$) by \eqref{eq:boundarea2} we have
\[
n_\Z \otimes \lambda_{\R_+}(\Delta^\circ_T) \geq pT^2 - \frac{3}{4}pT^2 = \frac{1}{4}pT^2
\] for all $T$ sufficiently large (depending only on $\zeta$), by the Chebyshev inequality we obtain that the probability in \eqref{eq:probap} can be bounded from above by 
\begin{align}
	\E_p(N^{\circ,\,\text{small}}_T q_{p,T}(U^{\circ,(T)})) & \leq \left(\E_p([N^{\circ,\,\text{small}}_T]^2) \P_p((\mathcal{P}_{T,\tfrac{1}{4}})^c)\right)^{\tfrac{1}{2}} + \E_p(N^{\circ,\,\text{small}}_T)\frac{T}{\frac{1}{4}pT^2} \nonumber\\
	& \leq \sqrt{18}T^\beta\mathrm{e}^{-C_\zeta T^{1-\zeta}} + 4T^{\beta-1}, \label{eq:ob1}
\end{align} for all $T$ sufficiently large and some constant $C_\zeta > 0$ depending only on~$\zeta$. Indeed, the first inequality above follows from the linearity of expectation and the Cauchy-Schwarz inequality, while \eqref{eq:ob1} does so from \eqref{eq:triangle} and the~fact~that the conditional distribution of $N^{\circ,\,\text{small}}_T$ given $\partial \cV_T$ is Poisson with parameter $T^{\beta-2}(n_\Z \otimes \lambda_{\R_+}(\Delta^\circ_T))$, so that by \eqref{eq:expectations} and the bound
\[
n_\Z \otimes \lambda_{\R_+}(\Delta^\circ_T) \leq \int_0^{T}\big(s^{+,T}(t) -s^{-,T}(t)\big)\mathrm{d}t \leq T\big(s^{+,T}(T) -s^{-,T}(T)\big)
\] which follows from \eqref{eq:boundarea}, a standard computation yields the crude bound
\[
\E_p([N^{\circ,\,\text{small}}_T]^2) \leq 18T^{2\beta}.
\] Finally, since the expression in \eqref{eq:ob1} tends to zero as $T \rightarrow \infty$ uniformly over all $p \in [T^{-\zeta},1]$ by choice of $\beta$ and $\zeta$, we obtain~\eqref{eq:equality} and thus the result.
\end{proof}

\begin{proof}[Proof of Lemma \ref{lemma:couplepoisson}] Observe that by  \eqref{eq:ncirc1} it will suffice to prove that
\begin{equation}\label{eq:ncoincide}
\lim_{T \rightarrow \infty}\left[\sup_{p \geq T^{-\zeta}} \P_p\left( \Omega''_k \cap \{ \widehat{s}^{\,+,T}_{(n_i,t_i)}(t) \neq n_i+ \sigma^{+,T}(t) \text{ for some }t \in [0,t_i]\}\right)\right]=0
\end{equation} since $\Omega''_k=\Omega_k \cap \Omega'_k$, the intersection of the events from \eqref{eq:defomega1} and~\eqref{eq:defomega2}, satisfies
\begin{equation}\label{eq:boundomega3}
\lim_{T \rightarrow \infty} \left[\inf_{p \geq T^{-\zeta}} \P_p(\Omega''_k)\right] = 1
\end{equation} by \eqref{eq:boundomega}, \eqref{eq:triangle} and \eqref{eq:suff2}. The first step towards proving \eqref{eq:ncoincide} will be to give us some ``room to maneuver'' by showing that, with overwhelming probability as $T \rightarrow \infty$, the point $(n_i,t_i)$ is located at a macroscopical distance from the topological boundary of $\Delta_T$. More precisely, we wish to show that
\begin{equation} \label{eq:macroscopical}
	\lim_{\rho \searrow 0} \liminf_{T \rightarrow \infty} \left[ \inf_{p \geq T^{-\zeta}} \P_p (\Omega_k'' \cap \{ r_{p,T}(n_i,t_i) \in \Delta((0,1);\rho)\,,\, t_i \geq \rho T\})\right] = 1,
\end{equation} where $r_{p,T}(n,t)=r_p(\tfrac{n}{T},\tfrac{t}{T})=(\tfrac{n}{pT},\tfrac{t}{T})$.To prove \eqref{eq:macroscopical}, we notice that, by \eqref{eq:boundrp}, on the event $\Omega''_k$ we have that  
\[
\|U_i - r_{p,T}(n_i,t_i)\| \leq T^{\zeta-1} \hspace{1cm}\text{ and }\hspace{1cm}t_i=t(U_i)T,
\] so that the event in \eqref{eq:macroscopical} occurs, in particular, for all $T$ sufficiently large (depending only on $\rho$ and $\zeta$) whenever
\[
U_i \in \Delta((0,1);2\rho) \hspace{1cm}\text{ and }\hspace{1cm}t(U_i) \geq \rho.
\] Since almost surely we have that $U_i$ is not located on the topological boundary of $\Delta=\Delta(0,1)$, we have that 
\[
\lim_{\rho \searrow 0} \P( U_i \in \Delta((0,1);2\rho)\,,\,t(U_i) \geq \rho) = 1,
\] from where, in combination with \eqref{eq:boundomega3},  \eqref{eq:macroscopical} now follows. 

Let us now turn to the proof of \eqref{eq:ncoincide}. By \eqref{eq:macroscopical}, it will suffice to show that
\begin{equation}\label{eq:ncoincide2}
\lim_{T \rightarrow \infty}\left[\sup_{p \geq T^{-\zeta}} \P_p\left( \Omega''_{k;i,\rho}\cap \{ \widehat{s}^{\,+,T}_{(n_i,t_i)}(t) \neq n_i+ \sigma^{+,T}(t) \text{ for some }t \in [0,t_i]\}\right)\right]=0,
\end{equation} where we set $\Omega''_{k;i,\rho}:=\Omega''_k \cap \{ r_{p,T}(n_i,t_i) \in \Delta((0,1);\rho)\,,\, t_i \geq \rho T\}$. To this~end, observe that since $\widehat{s}^{\,+,T}_{(n_i,t_i)}$ and $n_i+ \sigma^{+,T}$ agree until the hitting time of $\partial \Delta_T$, on the event in \eqref{eq:ncoincide2} the path $n_i+\sigma^{+,T}$ must reach the boundary $\partial \Delta_T$ and therefore, for some $t \in [0,t_i]$, we must have
	\begin{equation} \label{eq:alter}
	n_i+\sigma^{+,T}(t)= s^{+,T}(T-t_i+t).
	\end{equation}  However, since a direct computation reveals that, for $T$ sufficiently large, \eqref{eq:alter} can only occur on the event $\Omega''_{k;i;\rho}$ if either 
	\begin{equation} \label{eq:devsigma}
	\sup_{0 \leq u \leq 1} \left[ \frac{1}{p(1-T^{\beta-2})t_i}\sigma^{+,T}(ut_i) - u \right] > \frac{\rho}{4}
	\end{equation} or
	\[
	\inf_{0 \leq u \leq 1} \left[ \frac{1}{pT}s^{+,T}(uT)  - u \right] < -\frac{\rho}{4}
\] (this is precisely where we use that $r_{p,T}(n_i,t_i) \in \Delta((0,1);\rho)$), by \eqref{eq:triangle} (applied to both $\sigma^{+,T}$ and $s^{+,T}$) \eqref{eq:ncoincide2} immediately follows (note that we can indeed use \eqref{eq:triangle} in \eqref{eq:devsigma} because the length $t_i$ of the path $\sigma^{+,T}$ satisfies $t_i \geq \rho T \rightarrow \infty$ as $T \rightarrow \infty$ on the event $\Omega''_{k;i;\rho}$). This concludes the proof.
\end{proof}

\section*{Acknowledgements} The authors are immensely grateful to Vladas Sidoravicius for encouraging us to take on this problem and for the numerous enlightening discussions we had together during our joint time at NYU-Shanghai.

The work of S.S. was supported in part at the Technion by a fellowship from the Lady Davis Foundation, the Israeli Science Foundation grants no. 1723/14 and 765/18, Fondecyt grant no. 11200690, Iniciativa Cient\'ifica Milenio ``Modelos Estoc\'asticos de Sistemas Complejos y Desordenados'' and by the United States-Israel Binational Science Foundation (BSF) grant no. 2018330.

\bibliographystyle{plain}
\bibliography{bd}
\end{document}